\documentclass[11pt]{article}

\usepackage[margin=1in]{geometry}
\usepackage{amsmath,amssymb,amsthm,mathtools}
\usepackage{bm}
\usepackage{microtype}
\usepackage{enumitem}
\usepackage{booktabs}
\usepackage{tabularx}
\usepackage{cite}
\usepackage{aliascnt}
\usepackage{xcolor}
\usepackage{tikz}
\usepackage{pgfplots}
\usepackage{subcaption}
\usepackage{float}
\usepackage[hidelinks]{hyperref}
\usepackage[capitalize,noabbrev]{cleveref}
\pgfplotsset{compat=1.18}
\usepgfplotslibrary{groupplots}
\hypersetup{
  pdftitle={The Algorithmic Geometry of Decompression Schedules},
  pdfauthor={Benjamin Marsh}
}

\newtheorem{theorem}{Theorem}[section]
\newaliascnt{lemma}{theorem}
\newtheorem{lemma}[lemma]{Lemma}
\aliascntresetthe{lemma}
\newaliascnt{proposition}{theorem}
\newtheorem{proposition}[proposition]{Proposition}
\aliascntresetthe{proposition}
\newaliascnt{corollary}{theorem}
\newtheorem{corollary}[corollary]{Corollary}
\aliascntresetthe{corollary}
\theoremstyle{definition}
\newaliascnt{definition}{theorem}
\newtheorem{definition}[definition]{Definition}
\aliascntresetthe{definition}
\newaliascnt{assumption}{theorem}
\newtheorem{assumption}[assumption]{Assumption}
\aliascntresetthe{assumption}
\newaliascnt{example}{theorem}
\newtheorem{example}[example]{Example}
\aliascntresetthe{example}
\theoremstyle{remark}
\newaliascnt{remark}{theorem}
\newtheorem{remark}[remark]{Remark}
\aliascntresetthe{remark}

\DeclareMathOperator*{\argmin}{arg\,min}
\newcommand{\R}{\mathbb{R}}
\newcommand{\N}{\mathbb{N}}
\newcommand{\Pa}{P_a}
\newcommand{\Po}{P_{\mathrm{surf}}}
\newcommand{\Pinit}{P^{0}}
\newcommand{\wt}{\widetilde}

\newcommand{\calG}{\mathcal{G}}
\newcommand{\calZ}{\mathcal{Z}}

\newcommand{\tO}{\widetilde{O}}
\newcommand{\ppOtwo}{\mathrm{ppO_2}}
\newcommand{\END}{\mathrm{END}}
\newcommand{\Ntwo}{\mathrm{N_2}}
\newcommand{\He}{\mathrm{He}}
\newcommand{\Otwo}{\mathrm{O_2}}

\crefname{assumption}{Assumption}{Assumptions}
\Crefname{assumption}{Assumption}{Assumptions}
\crefname{definition}{Definition}{Definitions}
\Crefname{definition}{Definition}{Definitions}

\crefname{theorem}{Theorem}{Theorems}
\Crefname{theorem}{Theorem}{Theorems}
\crefname{lemma}{Lemma}{Lemmas}
\Crefname{lemma}{Lemma}{Lemmas}
\crefname{proposition}{Proposition}{Propositions}
\Crefname{proposition}{Proposition}{Propositions}
\crefname{corollary}{Corollary}{Corollaries}
\Crefname{corollary}{Corollary}{Corollaries}
\crefname{example}{Example}{Examples}
\Crefname{example}{Example}{Examples}
\crefname{remark}{Remark}{Remarks}
\Crefname{remark}{Remark}{Remarks}

\title{The Algorithmic Geometry of Decompression Schedules\thanks{This paper studies mathematical properties of a stylised model; it is not a clinical or diving safety guideline.}}
\author{
  Benjamin Marsh\\
  \small Macrops\\
  \small ben@macrops.org
}
\date{July 2026}

\begin{document}
\maketitle

\begin{abstract}
We study decompression planning as a hybrid optimal control problem with ppO$_2$ and equivalent narcotic depth feasibility, affine tissue ceilings, convex oversaturation penalties, and an optional post-surface functional. A continuous rearrangement theorem derives monotone ascent for terminal state objectives, and a safe off-gassing exchange orders constant dwell blocks in a
local running stress regime.  In the unrestricted bidirectional extension, where excursions deeper than the starting depth are admitted, an explicit one compartment instance with integrated oversaturation makes a re-descent strictly better than every monotone profile.  The full problem is therefore optimal control under an operational monotone ascent constraint. In the aggregate inert gas model without cumulative oxygen exposure, the minimum inert feasible gas dominates every pure or relaxed gas policy along each fixed depth path, giving exact pure attainment with finitely many switches.  A CNS/OTU type oxygen budget is precisely the path coupling state that restores nontrivial gas choice and can prevent zero inert hold consolidation.  We derive its oxygen shadow price and gas switching inequality. Maximal ascents and holds are dense, and an explicit risk price threshold makes the no-stop ascent globally optimal.  For fixed stops, an exact dwell switching identity permits one forward state sweep and one backward adjoint sweep, purely on-gassing holds are dominated, and the one compartment problem has closed form scalarised and capped solutions. For saturation decompression we reduce rectangular uncertainty to one deterministic worst case endpoint, finite bounce exposures exhibit the complementary phenomenon that the worst rate can be unique and strictly interior, with a closed form two level location and finitely many critical rates for an arbitrary segmented terminal input.  Algorithmically, clipped oversaturation is not a Markov state, whereas full tissue vectors support safe dominance and one sided enclosures certify an exact hard cap.  The compartment generated finite menu problem is NP-hard with one tissue, and a bounded horizon instance has $2^m$ exact nondominated labels.  Nevertheless, for fixed tissue dimension and polynomial conditioning, upper state compression gives an FPTAS for scalarisation, risk repair gives the exact cap counterpart.  A frontier conjugacy theorem finally characterises exactly which capped points scalarisation recovers.  Worked examples reproduce the dwell law, saturation endpoint equality, and realised time/risk frontier.
\end{abstract}

\section{Introduction}

A decompression plan specifies how a diver should ascend from depth, possibly pausing at discrete stops and switching gases, so as to control risk while minimising unnecessary total decompression time. Classical models, from Haldane's staged decompression and Workman's linear ``$M$ values''\cite{Baker1998} to B\"uhlmann's ZH-series and Thalmann's VVal-18, idealise tissues as first order perfusion compartments that approach an inspired inert partial pressure and impose affine depth dependent ceilings on tolerated supersaturation \cite{Boycott1908,Workman1965,Buehlmann1984,USNManualRev7}. Practical implementation overlays additional feasibility constraints such as limits on oxygen partial pressure (ppO$_2$) to avoid toxicity and bounds on the equivalent narcotic depth (END). END conventions differ in whether oxygen is treated as narcotic, both interpretations appear in agency guidance \cite{USNManualRev6,USNManualRev7,NOAADivingManual6}.

We formalise this planning problem on continuous depth trajectories $z(\cdot)$ and measurable gas choices $g(\cdot)$. The model includes ppO$_2$ and END feasibility, affine ceilings $M_i(z)=a_i+b_i\Pa(z)$, and convex penalties in normalised oversaturation. A terminal functional closes the objective at surfacing. Its canonical form is the risk accumulated during a fixed post-surface observation period. Without this term, a shallow stop can be rejected merely because most of its benefit occurs after the optimisation horizon. Three distinctions are essential. First, a relaxed gas control is not a source of extra power in the unbudgeted aggregate inert gas model. The minimum inert feasible gas dominates pointwise, so the apparent gas combinatorics collapse. This degeneracy identifies cumulative oxygen exposure as the operational resource missing from the base model. Second, monotone ascent is derived by rearrangement for terminal loading under a no deeper than start condition, but fails in an unrestricted bidirectional integrated stress extension and the operational constraint and the mathematical conclusion are therefore kept distinct. Third, the exact finite menu combinatorics do survive compartment generated increments, even though fixed dimensional state rounding admits a fully polynomial approximation. Maximal rate ascent plus holds supplies the dense normal form linking the continuous and finite problems.

\paragraph{Contributions.}
The paper establishes the following results.
\begin{enumerate}[leftmargin=2em,label=(\roman*)]
\item The gas envelope theorem gives exact pure gas attainment, finite switching, and a one gas per depth reduction in the unbudgeted aggregate model. It shows exactly why gas allocation is trivial without a cumulative oxygen resource. Adding a CNS/OTU type dose state restores an intertemporal tissue loading/exposure trade off, while the zero inert consolidation theorem identifies where the omission becomes visible. A normal maximum principle calculation gives the oxygen shadow price and a sharp two gas switching inequality.
\item A safe off gassing exchange theorem orders adjacent dwell blocks in the fixed block permutation problem.  A continuous decreasing rearrangement theorem derives monotone ascent globally for terminal state objectives when the initial depth remains maximal, while an explicit value gap construction proves failure in an unrestricted bidirectional integrated stress extension.  Within the operationally constrained model, a constructive pulse width argument proves density of maximal ascents and holds, and $\lambda\le1/B$ gives a global no-stop phase.
\item For fixed stops, an exact dwell switching identity separates instantaneous dwell risk, future risk reduction, and the time price. One forward and one backward sweep compute all dwell derivatives. Pure on gassing holds are dominated, and a one compartment problem admits closed forms for both scalarised and capped dwell times.
\item For saturation decompression, rectangular uncertainty in initial tissue state, rate constants, inspired pressure, and ceilings reduces to one deterministic worst case endpoint instance, a multigas equilibrium calculation verifies the continuum equality numerically. Finite bounce exposures generally fall outside the required off gassing ordering. For a two level bounce we locate the unique interior worst case rate in closed form and identify the exact transition back to endpoint monotonicity. More generally, an $m$ segment terminal pressure calculation has at most $m-1$ critical rates.
\item Clipped oversaturation is not a sufficient Markov state, even for two labels whose clipped load is identically zero; exact propagation requires the full tissue vector or full normalised tension. Componentwise tissue dominance is safe, and one sided tissue and risk enclosures give an exact hard cap certificate at every grid resolution.  With actual one compartment flows, the capped finite menu problem is NP-hard and an exact layer can contain $2^m$ nondominated labels even on a bounded horizon.  For fixed $n$ and polynomial conditioning, upper state rounding nevertheless yields an FPTAS for scalarisation, and risk repair yields an exact cap FPTAS.  The product grid's $K^n$ dependence keeps the variable dimension boundary explicit.
\item A frontier conjugacy theorem characterises exactly when a capped solution is recovered by $T+\lambda R$ and identifies the deterministic duality gap. The analytic example and a multi-compartment worked instance make the switching law, terminal closure, and realised time/risk frontier explicit.
\end{enumerate}

The analysis exposes optimisation structure shared by table based and algorithmic decompression models, including B\"uhlmann style implementations with gradient factors \cite{Pollock2015GF,Baker1998}.

\subsection{Notation, units, and abbreviations}\label{subsec:notation}
We measure depth $z$ in meters (m), time $t$ in minutes (min), and pressures in bar (absolute). Ambient pressure is $\Pa(z)=\Po+\gamma z$ and inspired partial pressures use $\Pa(z)-w$ (alveolar water vapour offset). Any consistent absolute pressure unit can be used; only ratios such as partial pressure fractions and affine ceilings matter in the analysis.

\begin{table}[H]
\centering
\caption{Main symbols and units.}\label{tab:notation}
\begin{tabularx}{\linewidth}{@{}l l X@{}}
\toprule
Symbol & Unit & Meaning \\
\midrule
$z$ & m & depth, $z=0$ at surface \\
$z_{\mathrm{exit}}$ & m & terminal depth below which ascent is uninterrupted \\
$\dot z_{\max}$ & m/min & ascent rate cap \\
$\Pa(z)$ & bar & ambient pressure at depth $z$ \\
$\Po$ & bar & surface ambient pressure \\
$\gamma$ & bar/m & pressure gradient in water \\
$w$ & bar & alveolar water vapour offset \\
$g\in\mathcal G$ & -- & breathing gas (finite alphabet) \\
$F_{O_2}(g),F_{N_2}(g),F_{He}(g)$ & -- & gas fractions (sum to $1$) \\
$F_I(g)$ & -- & inert fraction $F_{N_2}(g)+F_{He}(g)$ \\
$\ppOtwo(g,z)$ & bar & alveolar oxygen partial pressure \\
$C(t),\overline C$ & dose & cumulative oxygen exposure and its optional cap \\
$\chi(\ppOtwo)$ & dose/min & depth sensitive oxygen dose rate in the extension \\
$\mu^\star$ & risk/dose & oxygen shadow price on a fixed depth path \\
$\END(g,z)$ & m & equivalent narcotic depth (END) \\
$n$ & -- & number of compartments \\
$i\in\{1,\dots,n\}$ & -- & compartment index \\
$\tau_i$ & min & half time of compartment $i$ \\
$k_i$ & min$^{-1}$ & rate $k_i=\ln 2/\tau_i$ \\
$P_i(t)$ & bar & tissue inert gas pressure \\
$P_\infty(g,z)$ & bar & inspired inert gas driving pressure \\
$M_i(z)$ & bar & tolerated ceiling (affine ``$M$ value'') \\
$S_i(t)$ & -- & normalised oversaturation $(P_i-M_i(z))_+/M_i(z)$ \\
$\phi_i(\cdot)$ & risk/min & convex penalty (instantaneous risk rate) \\
$R(\pi)$ & risk & running risk plus terminal functional $\Psi(P(T))$ \\
$\Psi(P(T))$ & risk & optional terminal or post-surface risk functional \\
$T(\pi)$ & min & total decompression time \\
\bottomrule
\end{tabularx}
\end{table}

In this work, END denotes equivalent narcotic depth, DPP denotes dynamic programming principle, DAG denotes directed acyclic graph, and FPTAS denotes fully polynomial time approximation scheme. Throughout, $\varepsilon>0$ denotes an arbitrary accuracy tolerance used in approximation statements.

\section{Background and related work}\label{sec:related}
Staged decompression originates with Haldane's dissolved gas hypothesis and early table work \cite{Boycott1908}. Workman's $M$ values formalised tolerated supersaturation as an affine function of ambient pressure \cite{Workman1965}, and B\"uhlmann's ZH models extended these compartmental dynamics and remain the dominant basis of modern recreational and technical decompression implementations \cite{Buehlmann1984}. For piecewise constant or ramped depth segments, implementations commonly use closed form solutions (for example Schreiner type ramp solutions) to update tissue tensions efficiently \cite{SchreinerKelley1971,HamiltonThalmann2003}, the structural arguments below use only stability of the underlying ODE flows, not any particular closed form update formula. Thalmann's VVal-18 and related approaches underlie U.S. Navy air and mixed gas tables \cite{Thalmann1986,GerthDoolette2007,USNManualRev7}. Operational planning overlays feasibility constraints that are not part of the original dissolved gas compartment model. Oxygen partial pressure limits to mitigate toxicity and inert/narcotic exposure limits, often implemented through the END proxy \cite{USNManualRev6,USNManualRev7,NOAADivingManual6}. Different communities treat $O_2$ as narcotic or non-narcotic for END calculations, our parameter $\eta\in[0,1]$ makes this convention explicit and keeps feasibility (ppO$_2$/END windows) separate from the optimisation objective. Today's decompression schedules are often produced by dive computers implementing dissolved gas models with additional heuristics such as gradient factors (GF). Empirical and operational questions then shift from ``does the model exist?'' to ``what optimisation problem is the implementation implicitly solving, and what are the consequences of discrete constraints and device settings?'' Recent work has examined algorithmic decompression implementations and their behaviour \cite{Fraedrich2018COTS,DeRidder2023GF,Angelini2022Ceiling}. Optimisation of ascent profiles is itself well established. Lewis derived an optimal ascent rate in an early compartment model \cite{Lewis1983OptimalDecompression}, Horn optimised stop times at fixed modelled illness risk \cite{Horn2003Optimization,Horn2006Isoprobabilistic}; and nonlinear optimal control, multiparametric programming, and receding horizon decompression were developed by Feng, Gutvik, Johansen, and coauthors \cite{Feng2009Barrier,Feng2010Multiparametric,Gutvik2011Optimal,Feng2012RecedingHorizon}. Probabilistic real time schedule updating also predates the present work \cite{Survanshi1996RealTime}, as do shortest path calculations at a target risk and recent adaptive ascent algorithms \cite{Murphy2017Dissertation,DiMuro2023Adaptive}. Our contribution is a rigorous structural and approximation boundary for the stated proxy model, exact unbudgeted gas menu reduction, derived and impossible monotonicity regimes, dwell sensitivities, model generated finite menu hardness, an exact label lower bound, and a fixed compartment FPTAS with conservative cap certificates. A recurring debate concerns how decompression time should be distributed between deeper and shallower stops (``deep stops'')\cite{UHMSDeepStop2008}. Evidence from controlled evaluations has shown that redistributing decompression time from shallow to deeper stops can increase decompression sickness incidence in certain air diving profiles \cite{Doolette2011NEDU}. Modern reviews emphasise that decompression illness risk is multifactorial and that no single scalar objective universally captures all operational trade offs \cite{Mitchell2024DCI}.

Finite operational menus, switching budgets, and other mutually exclusive resource choices can make the attainable $(T,R)$ set nonconvex, in which case scalarisations recover only supported efficient points. The aggregate unbudgeted base model is simpler, its gas envelope theorem removes gas choice as a source of nonconvexity. A cumulative oxygen exposure state restores that choice. We therefore distinguish the continuous compartment frontier from arbitrary menu geometry. Risk can be represented deterministically or probabilistically. Probabilistic models have been developed to predict decompression sickness probabilities, and even severity strata, from exposure features \cite{Horn2006Isoprobabilistic,Howle2017Trinomial,Schirato2026GradientOptimization}. Convex penalties in normalised oversaturation provide a transparent mathematical proxy compatible with optimisation and yield structural results such as dominance, exact dwell marginals, and certified discretisation error.

While this paper focuses on a stylised perfusion limited compartment model with affine ceilings, the same algorithmic and feasibility issues arise in saturation and mixed gas operations. Recent reviews of commercial saturation decompression procedures highlight both procedural commonalities and empirical adaptations in the absence of abundant modern controlled trials \cite{Imbert2024Saturation}. The endpoint uncertainty theorem in \cref{sec:robust} is targeted precisely at the saturated initial equilibrium, it is not asserted for a finite bounce exposure. We complement it with a bounce calculation, even one bottom step followed by one lower input step can have a unique interior worst case rate, while a general segmented terminal input admits a finite critical rate enumeration. This reinforces the value of separating feasibility and operational constraints from the optimisation and computational structure. Casting decompression planning as optimal control \cite{Pontryagin1962,Liberzon2012,Betts2010} is natural because the system is genuinely dynamic, controls are hybrid, and feasibility constraints couple tissue states, depth, and gas choices. The control viewpoint supplies a constructive staged density theorem, dwell adjoints, and dynamic programming principles that justify the finite grid algorithms analysed below.

\subsection{Physical background}\label{subsec:physics}
Decompression planning is driven by how inert gases dissolve into and are eliminated from tissues as ambient pressure changes during descent and ascent \cite{HamiltonThalmann2003}. At depth $z$, ambient pressure is well approximated by a linear relation
\[
\Pa(z)=\Po+\gamma z,
\]
with $\Po$ the surface pressure and $\gamma$ the hydrostatic pressure gradient \cite{USNManualRev7,NOAADivingManual6}. A breathing mix $g$ with fraction $F_s(g)$ of gas species $s$ produces an inspired partial pressure proportional to ambient pressure. In alveolar models a constant water vapour pressure offset $w$ is subtracted from ambient pressure before taking fractions, reflecting humidification of inspired gas \cite{StatPearlsAlveolar2024}. In this paper we therefore use
\[
P_{\infty}(g,z)=F_I(g)\bigl(\Pa(z)-w\bigr)
\]
as the inspired inert gas driving pressure for the total inert fraction $F_I(g)=F_{N_2}(g)+F_{He}(g)$.

A standard dissolved gas approximation models each tissue compartment as a first order system that exponentially approaches the inspired inert gas pressure \cite{Workman1965,Buehlmann1984,HamiltonThalmann2003}. For compartment $i$ with rate constant $k_i=\ln 2/\tau_i$ (half time $\tau_i$),
\[
\dot P_i(t)=k_i\bigl(P_\infty(g(t),z(t))-P_i(t)\bigr).
\]
During ascent, ambient pressure decreases; if a tissue inert pressure $P_i$ exceeds what is tolerated at the current depth, the tissue is supersaturated. Classical table algorithms express tolerance via an affine ceiling (an ``$M$-value'')
\[
M_i(z)=a_i+b_i\Pa(z),
\]
with supersaturation measured relative to this limit \cite{Workman1965,Buehlmann1984}. We encode decompression stress using the normalised oversaturation
\[
S_i(t)=\max\!\left\{0,\frac{P_i(t)-M_i(z(t))}{M_i(z(t))}\right\},
\]
and accumulate a convex penalty in $S_i$ as an abstract risk score. This is a mathematical proxy but it is intended to support optimisation structure and algorithms, not to claim a calibrated injury model. Because tissue state does not cease to matter at first surfacing, we also allow a componentwise nondecreasing terminal functional. A fixed post-surface observation window is represented exactly by
\cref{prop:surface-tail}.

Operationally, not every gas can be breathed at every depth. Oxygen partial pressure is constrained to avoid hypoxia at shallow depths and oxygen toxicity at depth \cite{USNManualRev7,NOAADivingManual6}. Narcotic exposure is often limited using END, defined as the depth in air whose narcotic partial pressure matches that of the chosen gas at depth $z$ \cite{USNManualRev6,USNManualRev7}. Different communities treat oxygen as narcotic or non-narcotic for END calculations; we explicitly parameterise this convention using $\eta\in[0,1]$ in the narcotic fraction $F_{\mathrm{nar}}(g)=F_{N_2}(g)+\eta F_{O_2}(g)$.

Because $P_i$ changes smoothly and, under a rate cap, depth changes are piecewise Lipschitz, it is natural that optimal policies often take the familiar staged form. Maximal rate ascents between a finite set of constant depth holds, with possible gas switches at holds. A major goal of the paper is to formalise the extent to which this staged structure is justified in the stated model.

\section{Model and well posedness}

Let the state be
\[
x(t)=(z(t),P(t))\in X:=[0,Z_{\max}]\times\R_+^n,
\]
where $z(t)$ is depth in meters and $P(t)=(P_1(t),\dots,P_n(t))$ are inert gas tissue pressures in $n$ perfusion limited compartments. Controls are the vertical rate and gas choice
\[
u(t)=(v(t),g(t)),
\]
with the monotone ascent convention
\[
v(t)=\dot z(t)\in[-\dot z_{\max},0],\qquad g(t)\in\mathcal G(z(t))\quad\text{a.e.}
\]
The feasible gas correspondence $\mathcal G(\cdot)$ enforces ppO$_2$ and END windows. Depth evolves as $\dot z(t)=v(t)$ and each compartment satisfies
\[
\dot P_i(t)=k_i\bigl(P_\infty(g(t),z(t))-P_i(t)\bigr),\qquad i=1,\dots,n.
\]
Ceilings are affine in ambient pressure,
\[
M_i(z)=a_i+b_i\Pa(z),
\]
and the normalised oversaturation is
\[
S_i(t)=\max\!\left\{0,\frac{P_i(t)-M_i(z(t))}{M_i(z(t))}\right\}.
\]
For $\lambda>0$, define the running cost
\[
\ell(z,P,g):=1+\lambda\sum_{i=1}^n\phi_i\bigl(S_i(z,P)\bigr),
\]
so the scalarised free terminal time optimal control problem is
\begin{equation}\label{eq:ocp-scalar}
\begin{aligned}
(\mathbf P_\lambda):\qquad
\min_{T\ge0,\,(v(\cdot),g(\cdot))}\quad
&\int_0^T \ell(z(t),P(t),g(t))\,dt+\lambda\Psi(P(T))\\
\text{s.t.}\quad
&\dot z(t)=v(t),\qquad \dot P_i(t)=k_i(P_\infty(g(t),z(t))-P_i(t)),\\
&v(t)\in[-\dot z_{\max},0],\qquad g(t)\in\mathcal G(z(t))\quad\text{a.e.},\\
&v(t)=-\dot z_{\max}\quad\text{whenever }0<z(t)<z_{\mathrm{exit}},\\
&z(0)=z_{\mathrm{start}},\qquad P(0)=\Pinit,\qquad
T=\inf\{t\ge0:z(t)=0\}.
\end{aligned}
\end{equation}

Given a risk budget $\rho\ge 0$, the constrained time minimisation is
\begin{equation}\label{eq:ocp-cap}
\begin{aligned}
(\mathbf P_{\mathrm{cap}}):\qquad
\min_{T\ge0,\,(v(\cdot),g(\cdot))}\quad &T\\
\text{s.t.}\quad
&\int_0^T\sum_{i=1}^n\phi_i(S_i(t))\,dt+\Psi(P(T))\le\rho,\\
&\text{and all constraints in \eqref{eq:ocp-scalar}.}
\end{aligned}
\end{equation}

\subsection{Assumptions, primitives, and initialisation}

\begin{assumption}[Environment and gases]\label{ass:env}
Depth $z\in[0,Z_{\max}]$, surface at $z=0$. Ambient pressure is
\[
\Pa(z)=\Po+\gamma z,
\qquad \Po>0,\quad \gamma>0.
\]
Fix an alveolar water vapour offset $w\in(0,\Po)$. Let $\calG$ be a finite alphabet of breathing gases. For each $g\in\calG$, fractions $F_{O_2}(g),F_{N_2}(g),F_{He}(g)\ge0$ satisfy
\[
F_{O_2}(g)+F_{N_2}(g)+F_{He}(g)=1,
\]
and the inert fraction is $F_I(g)=F_{N_2}(g)+F_{He}(g)$.
\end{assumption}

\begin{assumption}[Feasibility windows (ppO$_2$ and END)]\label{ass:feas}
Fix $0<\underline p\le\overline p<\infty$, and $\eta\in[0,1]$ to define the narcotic fraction
\[
F_{\mathrm{nar}}(g)=F_{N_2}(g)+\eta F_{O_2}(g).
\]
For $z\in[0,Z_{\max}]$, define alveolar ppO$_2$ and END by
\[
\ppOtwo(g,z)=F_{O_2}(g)\bigl(\Pa(z)-w\bigr),\qquad
0.79\bigl(\Pa(z')-w\bigr)=F_{\mathrm{nar}}(g)\bigl(\Pa(z)-w\bigr),
\]
and set $\END(g,z)=z'$ \cite{USNManualRev6}. Fix $\overline{\END}\in[0,Z_{\max}]$. The feasible gas correspondence is
\[
\calG(z):=\left\{g\in\calG:\ \underline p\le \ppOtwo(g,z)\le\overline p,
\quad \END(g,z)\le\overline{\END}\right\}.
\]
Assume $\calG(z)\neq\varnothing$ for all $z$. Since $z\mapsto \ppOtwo(g,z)$ and $z\mapsto\END(g,z)$ are continuous and $\calG$ is finite, the graph $\{(z,g):g\in\calG(z)\}$ is Borel and closed if $\calG$ is given the discrete topology.
\end{assumption}

\begin{assumption}[Compartments and ceilings]\label{ass:comp}
Compartments are indexed by $i\in\{1,\dots,n\}$ for fixed $n\in\N$. Each compartment has half time $\tau_i>0$ and rate constant $k_i=\ln 2/\tau_i$. Inert gas tissue pressure $P_i$ evolves under measurable $(z(\cdot),g(\cdot))$ via the Carath\'eodory ODE
\[
\dot P_i(t)=k_i\bigl(P_\infty(g(t),z(t))-P_i(t)\bigr),\qquad
P_\infty(g,z)=F_I(g)\bigl(\Pa(z)-w\bigr).
\]
Ceilings are affine in ambient pressure, $M_i(z)=a_i+b_i\Pa(z)$ with $a_i>0$ and $b_i\in(0,1]$. Initial tissue state $P(0)\in\R_+^n$ and initial depth $z(0)=z_{\mathrm{start}}\in[0,Z_{\max}]$ are given. Define
\[
S_i(t)=\max\!\left\{0,\frac{P_i(t)-M_i(z(t))}{M_i(z(t))}\right\}.
\]
\end{assumption}

\begin{assumption}[Penalty]\label{ass:phi}
For each compartment $i$, $\phi_i:\R_+\to\R_+$ is convex, nondecreasing, locally Lipschitz, and satisfies $\phi_i(0)=0$. We interpret $\phi_i(S_i(t))$ as an instantaneous penalty rate associated with compartment $i$ normalised oversaturation.
\end{assumption}

\begin{assumption}[Terminal risk]\label{ass:terminal}
The terminal functional $\Psi:\R_+^n\to\R_+$ is continuous, locally Lipschitz, and componentwise nondecreasing.  The choice $\Psi\equiv0$ gives the in water only objective.  The canonical closed horizon choice is the risk accrued during a fixed surface observation window, as formalised in \cref{prop:surface-tail} below.
\end{assumption}

\begin{assumption}[Kinematics and controls]\label{ass:kin}
Depth $z:[0,T]\to[0,Z_{\max}]$ is absolutely continuous with ascent rate bound
\[
\dot z(t)\in[-\dot z_{\max},0]\quad\text{a.e.},\qquad
T=\inf\{t\ge0:z(t)=0\}.
\]
Fix a terminal transit depth $z_{\mathrm{exit}}\in(0,Z_{\max}]$ and require $\dot z(t)=-\dot z_{\max}$ whenever $0<z(t)<z_{\mathrm{exit}}$. This closed terminal transit condition rules out arbitrarily shallow holds, operationally, $z_{\mathrm{exit}}$ lies below the shallowest allowed stop. Thus $z(t)>0$ for $t<T$ when $z_{\mathrm{start}}>0$, if $z_{\mathrm{start}}=0$, then $T=0$. Surface waiting is not an admissible in water control as the fixed post-surface window begins at this first hit and is charged through $\Psi$. We otherwise restrict to monotone ascent profiles (no re-descents), consistent with standard operational decompression practice. Gas control $g:[0,T]\to\calG$ is measurable with $g(t)\in\calG(z(t))$ a.e. A profile $\pi=(z(\cdot),g(\cdot),T)$ is feasible if \cref{ass:env,ass:feas,ass:comp,ass:phi,ass:terminal,ass:kin} hold.
\end{assumption}

The monotone restriction is not left as an implicit modelling convenience but it is derived for the terminal state relaxation in \cref{thm:endpoint-monotone}, is locally valid for safe off gassing blocks by \cref{prop:safe-offgas-exchange}, and is proved not to be without loss in an unrestricted bidirectional integrated stress extension in
\cref{prop:monotonicity-impossible}.

\begin{proposition}[Fixed surface tail as a terminal functional]\label{prop:surface-tail}
Fix an observation horizon $H_{\mathrm{surf}}\ge0$ and a surface gas $g_{\mathrm{surf}}\in\calG(0)$.  For a terminal tissue state $P$, let
\[
P_i^{\mathrm{surf}}(s;P)
=q_{\mathrm{surf}}+\bigl(P_i-q_{\mathrm{surf}}\bigr)e^{-k_i s},
\qquad
q_{\mathrm{surf}}
=F_I(g_{\mathrm{surf}})(\Po-w),
\]
and define
\begin{equation}\label{eq:surface-tail}
\Psi_{H_{\mathrm{surf}}}(P)
:=\sum_{i=1}^n\int_0^{H_{\mathrm{surf}}}
\phi_i\!\left(
\frac{(P_i^{\mathrm{surf}}(s;P)-M_i(0))_+}{M_i(0)}
\right)ds.
\end{equation}
Then $\Psi_{H_{\mathrm{surf}}}$ satisfies \cref{ass:terminal}.  It is convex when the $\phi_i$ are convex.  Hence post-surface risk over a fixed window can be included without augmenting the dynamic state or changing the admissible control set.
\end{proposition}

\begin{proof}
For every $i$ and $s$, the map $P_i\mapsto P_i^{\mathrm{surf}}(s;P)$ is affine and increasing.  Composition with the positive part, division by the positive constant $M_i(0)$, and the nondecreasing convex map $\phi_i$ preserves convexity and monotonicity. Integration over a compact interval preserves both properties.  Local Lipschitz continuity follows from the local Lipschitz constants of the $\phi_i$ on the reachable compact pressure range.
\end{proof}

The offset $w$ models alveolar water vapour and subtracts from all inspired partial pressures, it leaves $\Pa(\cdot)$ unchanged. END is defined by linearly matching narcotic partial pressures against air's inert fraction $0.79$. The narcotic $O_2$ weight $\eta\in[0,1]$ nests the conventions $\eta=0$ (oxygen non-narcotic) and $\eta=1$ (oxygen fully narcotic). Affine ceilings $M_i=a_i+b_i\Pa$ are standard in compartment models.

We switch between a scalarisation, $T+\lambda R$, and a hard risk cap with time minimisation. The scalarisation is convenient for proofs and sensitivity calculations, while the cap reflects how schedules are often used operationally. They need not be equivalent on a nonconvex attainable set, sweeping $\lambda$ traces only supported efficient points. The precise recovery criterion and deterministic duality gap are given in
\cref{thm:frontier-duality}.

\begin{definition}[Scalarised and capped problems]\label{prob:def}
For a feasible profile $\pi$, define
\[
R(\pi)=\sum_{i=1}^n\int_0^T \phi_i\bigl(S_i(t)\bigr)\,dt
+\Psi(P(T)),
\qquad T(\pi)=T,
\qquad J_\lambda(\pi)=T(\pi)+\lambda R(\pi).
\]
The two optimisation problems are
\begin{align*}
(\mathbf P_\lambda):\quad
&\min_{\pi\ \mathrm{feasible}} J_\lambda(\pi),\\[0.3em]
(\mathbf P_{\mathrm{cap}}):\quad
&\min_{\pi\ \mathrm{feasible}} T(\pi)\quad\text{s.t.}\quad R(\pi)\le\rho.
\end{align*}
\end{definition}

\subsection{Measurable selections}

\begin{lemma}[Measurable gas selection]\label{lem:measurable-selection}
Let $z:[0,T]\to[0,Z_{\max}]$ be measurable and let $\calG(\cdot)$ be as in \cref{ass:feas}. Then there exists a measurable $g:[0,T]\to\calG$ with $g(t)\in\calG(z(t))$ a.e. Moreover, if $h:[0,T]\times\calG\to\R$ is measurable and $\calG$ is finite, there exists a measurable $g^\star$ that attains the pointwise feasible minimum,
\[
h(t,g^\star(t))=\min\{h(t,g):g\in\calG(z(t))\}\quad\text{a.e.}
\]
\cite{KuratowskiRyll1965}
\end{lemma}

\begin{proof}
Enumerate $\calG=\{g_1,\dots,g_m\}$. For each $k$ set
\[
A_k:=\{t:\ g_k\in\calG(z(t))\ \text{and}\ h(t,g_k)\le h(t,g_\ell)\ \text{for every feasible }g_\ell\in\calG(z(t))\}.
\]
Define disjoint sets by $B_1:=A_1$ and $B_k:=A_k\setminus\bigcup_{\ell<k}A_\ell$ for $k\ge2$. Then $g^\star(t):=g_k$ on $B_k$ is measurable and attains the pointwise minimum a.e. The feasibility only statement follows by taking $h\equiv0$.
\end{proof}

\subsection{The gas envelope reduction}

The aggregate model has a stronger property than relaxed control existence. Because every compartment is driven by the same total inspired inert pressure, gas choice is ordered pointwise by inert fraction.  Define
\begin{equation}\label{eq:gas-envelope}
F_I^{\min}(z):=\min_{g\in\calG(z)}F_I(g),
\qquad
\calG_{\min}(z):=\argmin_{g\in\calG(z)}F_I(g).
\end{equation}

\begin{theorem}[Gas-envelope dominance]\label{thm:gas-envelope}
Under \cref{ass:env,ass:feas,ass:comp,ass:phi,ass:terminal}, let $z(\cdot)$ be any feasible depth path and let $g(\cdot)$ be any feasible pure or relaxed gas control. There is a measurable pure selector
\(
g_{\min}(t)\in\calG_{\min}(z(t))
\)
such that, from the same initial tissue state,
\begin{equation}\label{eq:gas-envelope-order}
P_i^{\min}(t)\le P_i^{g}(t)
\quad\text{for every }i\text{ and every }t\in[0,T].
\end{equation}
Consequently
\[
R(z,g_{\min},T)\le R(z,g,T),
\qquad
J_\lambda(z,g_{\min},T)\le J_\lambda(z,g,T).
\]
Thus both $(\mathbf P_\lambda)$ and $(\mathbf P_{\mathrm{cap}})$ admit an optimal gas policy that uses a minimum inert feasible gas at almost every depth.  In particular, gas relaxation has zero value gap in the aggregate model. Moreover, the interval constraints in \cref{ass:feas} induce a finite partition of $[0,Z_{\max}]$ on whose cells $\calG_{\min}(z)$ is constant. Along a monotone depth path, a tie broken envelope selector therefore has only finitely many gas switches.
\end{theorem}

The operational use of oxygen to accelerate decompression is not new \cite{Parker1998Oxygen}. The theorem is the exact comparison principle created by the aggregate input assumption \emph{and the absence of a cumulative oxygen price}.

\begin{proof}
Measurability follows from \cref{lem:measurable-selection}, applied to $h(t,g)=F_I(g)$.  Write
\[
q_{\min}(t)=F_I(g_{\min}(t))(\Pa(z(t))-w)
\]
and let $q_g(t)$ be the inspired inert pressure under the competing pure gas or the convex combination of inspired pressures under a relaxed gas.  By definition, $q_{\min}(t)\le q_g(t)$ almost everywhere.  Variation of constants gives
\[
P_i^g(t)-P_i^{\min}(t)
=k_i\int_0^t e^{-k_i(t-s)}\bigl(q_g(s)-q_{\min}(s)\bigr)\,ds\ge0.
\]
The maps $P_i\mapsto S_i$ and $S_i\mapsto\phi_i(S_i)$ are nondecreasing, and $\Psi$ is componentwise nondecreasing, which proves the cost inequalities including the terminal term. For a fixed gas, each ppO$_2$ inequality and the END inequality cuts out an interval in depth, possibly empty after intersection with $[0,Z_{\max}]$.  The finitely many interval endpoints partition the depth range into finitely many cells with a constant feasible gas set.  Refining cells at inert fraction ties makes a deterministic selector constant on each open cell.  A monotone path crosses each boundary at most once, which proves the finite switch statement.
\end{proof}

\begin{corollary}[Gas choice without an oxygen exposure budget]
\label{cor:gas-pruning}
In the aggregate model with only pointwise ppO$_2$/END windows, and without a cumulative oxygen exposure budget or another path coupling gas resource, every finite grid algorithm may replace $\calG(z_j)$ by one tie broken element of $\calG_{\min}(z_j)$ without changing its optimal value. Thus the gas factor $G_j$ in the unbudgeted base model running times can be set to one. Since $F_I(g)=1-F_{O_2}(g)$, the aggregate model assigns no scarcity cost to oxygen and therefore always selects the feasible gas with greatest oxygen fraction. A cumulative oxygen exposure budget is exactly the augmentation that turns gas selection into an intertemporal allocation problem. Gas switch limits, gas supply constraints, and species specific kinetics are other couplings that can invalidate the pruning rule.
\end{corollary}

\paragraph{Oxygen exposure extension.}
Let $\chi:\R_+\to\R_+$ be a nondecreasing dose-rate function and write
\[
d_g(z):=\chi\!\left(\ppOtwo(g,z)\right).
\]\footnote{The monotonicity hypothesis includes the standard operational clocks. If $t_{\max}(p)$ is the NOAA single-exposure limit at oxygen pressure $p$, the associated CNS rate is $\chi_{\mathrm{CNS}}(p)=100/t_{\max}(p)$ percent per minute; the tabulated $t_{\max}$ is nonincreasing in $p$. A common oxygen toxicity unit (OTU), or unit pulmonary toxic dose (UPTD), rate is $\chi_{\mathrm{OTU}}(p)=\mathbf 1_{\{p>0.5\}} ((p-0.5)/0.5)^{0.83}$ per minute, which is increasing above $0.5$ bar \cite{NOAADivingManual6,USNManualRev7}.} For a pure gas, augment the state by
\begin{equation}\label{eq:oxygen-budget}
\dot C(t)=d_{g(t)}(z(t)),\qquad
C(0)=0,\qquad C(T)\le\overline C.
\end{equation}
For a relaxed gas $\alpha(t)$, the consistent dose rate is the time average $\dot C=\sum_g\alpha_gd_g(z)$ of the pure gas dose rates, rather than $\chi$ evaluated at an averaged oxygen fraction. If the minimum inert envelope satisfies this budget, it remains optimal by \cref{thm:gas-envelope}. If it violates the budget while another gas policy on the same depth path is feasible, every feasible policy must depart from the envelope on a set of positive measure and accept a weakly larger inert input there. Gas choice is then exactly a trade off between tissue loading and scarce oxygen exposure. In a finite state algorithm, $C$ becomes an additional resource coordinate, $G_j$ can no longer be set to one, and conservative upward dose rounding preserves the exposure cap \cite{USNManualRev7,NOAADivingManual6}.

\begin{theorem}[Zero inert hold consolidation]\label{thm:zero-inert}
Suppose a gas $g_0$ with $F_I(g_0)=0$ is feasible throughout a depth interval $[z_\ell,z_h]$, where $z_h>z_\ell$. Consider any staged subprofile that enters at $z_h$, exits at $z_\ell$, uses $g_0$ throughout, ascends at the rate cap between holds, and has total hold time $H$. Replace it by a hold of length $H$ at $z_h$ followed by one max rate ascent to $z_\ell$. The replacement has the same duration and exactly the same exit tissue state, and its accumulated risk is no larger. Consequently, in the aggregate model, every optimal profile may consolidate all zero inert holds within a connected feasibility interval at its deepest point.
\end{theorem}

\begin{proof}
Under $F_I(g_0)=0$, every compartment satisfies $\dot P_i=-k_iP_i$, independently of depth. Both subprofiles have the same total duration, so their tissue trajectories are identical as functions of elapsed time and their exit states coincide. The consolidated profile delays every part of the ascent until after the combined hold. At each elapsed time its depth is therefore weakly greater than that of the original subprofile. Since $M_i(z)=a_i+b_i\Pa(z)$ is increasing in depth, its normalised oversaturation and running penalty are pointwise no larger. The exit state, exit depth, and elapsed time are identical, so every prescribed suffix, including its terminal functional, is unchanged.
\end{proof}

\begin{remark}[Depth sensitive oxygen exposure]
\label{rem:oxygen-diagnostic}
On pure oxygen the aggregate tissue dynamics are independent of depth. Accordingly, a model containing only pointwise ppO$_2$ feasibility cannot optimally split a fixed amount of oxygen dwell between 6 m and 3 m, the same total dwell at 6 m has the same exit tissue state and no greater running risk. Operational use of both stops therefore localises the mathematical divergence to a missing depth sensitive constraint. The cumulative exposure in \eqref{eq:oxygen-budget} supplies exactly such a constraint. Moving a dwell of length $\tau$ from $z_\ell$ to $z_h>z_\ell$ changes dose by
\[
\tau\left[
\chi\bigl(\ppOtwo(g_0,z_h)\bigr)
-\chi\bigl(\ppOtwo(g_0,z_\ell)\bigr)
\right],
\]
which is positive for a strictly increasing dose rate. The consolidation theorem then applies only when the replacement respects the exposure cap, a binding budget can make shallower oxygen or a split between stops optimal. A budget on oxygen minutes alone would not break consolidation, the resource must depend on ppO$_2$ or another depth sensitive exposure state.
\end{remark}

\begin{theorem}[Oxygen budget duality and shadow price]
\label{thm:oxygen-shadow}
Fix a feasible depth path $z:[0,T]\to[0,Z_{\max}]$. For a relaxed gas control $\alpha$, set
\[
Q(t):=\Pa(z(t))-w,\qquad
q^\alpha(t):=Q(t)\sum_g\alpha_g(t)F_I(g),
\]
\[
d^\alpha(t):=\sum_g\alpha_g(t)d_g(z(t)),\qquad
C_z(\alpha):=\int_0^T d^\alpha(t)\,dt,
\]
where $\alpha_g(t)=0$ for $g\notin\calG(z(t))$ and
$\sum_g\alpha_g(t)=1$. Let $P^\alpha$ solve
$\dot P_i=k_i(q^\alpha-P_i)$ and define
\[
R_z(\alpha):=
\int_0^T r(z(t),P^\alpha(t))\,dt+\Psi(P^\alpha(T)),
\qquad
r(z,P):=\sum_i\phi_i(S_i(z,P)).
\]
For an available ascent dose $c$, write
\[
\mathcal R_z(c):=
\inf\{R_z(\alpha): C_z(\alpha)\le c\}.
\]
Assume $\Psi$ is convex and there is a strictly feasible relaxed control with $C_z(\alpha)<c$. Then the infimum is attained, $c\mapsto\mathcal R_z(c)$ is convex and nonincreasing, and there is an oxygen shadow price $\mu^\star\ge0$ such that
\begin{equation}\label{eq:oxygen-duality}
\mathcal R_z(c)
=\inf_\alpha\{R_z(\alpha)+\mu^\star C_z(\alpha)\}
-\mu^\star c,
\qquad
\mu^\star(C_z(\alpha^\star)-c)=0.
\end{equation}
Moreover,
\[
-\mu^\star\in\partial\mathcal R_z(c);
\]
at every differentiability point,
$\mathcal R_z'(c)=-\mu^\star$.

Suppose additionally that $r(z(t),\cdot)$ and $\Psi$ are differentiable along an optimum. Define tissue costates by
\begin{equation}\label{eq:oxygen-adjoint}
-\dot p_i(t)=\partial_{P_i}r(z(t),P^\star(t))-k_ip_i(t),
\qquad
p_i(T)=\partial_i\Psi(P^\star(T)).
\end{equation}
Then
\begin{equation}\label{eq:oxygen-costate-positive}
p_i(t)=e^{-k_i(T-t)}p_i(T)
+\int_t^T e^{-k_i(s-t)}
\partial_{P_i}r(z(s),P^\star(s))\,ds\ge0.
\end{equation}
With
\[
A(t):=Q(t)\sum_i k_ip_i(t)\ge0,
\]
the priced gas envelope is
\begin{equation}\label{eq:oxygen-priced-envelope}
\operatorname{supp}\alpha^\star(t)
\subseteq
\argmin_{g\in\calG(z(t))}
\left\{A(t)F_I(g)+\mu^\star d_g(z(t))\right\}
\quad\text{for a.e. }t.
\end{equation}
The same statement holds in the nonsmooth case with measurable Clarke subgradient selections.
\end{theorem}

\begin{proof}
For the fixed path, $\alpha\mapsto q^\alpha$ and $\alpha\mapsto P^\alpha$ are affine. The map $P\mapsto S_i(z,P)$ is convex, and composition with the convex nondecreasing $\phi_i$ preserves convexity. Thus $R_z$ is convex because $\Psi$ is convex, while $C_z$ is linear. The feasible relaxed controls form a weak star compact simplex valued set, the compartment convolution maps weak star convergent controls to convergent tissue paths. Lower semicontinuity therefore gives attainment. Slater's condition is exactly the strict dose feasibility assumed above. Strong duality for the resulting scalar convex constraint gives \eqref{eq:oxygen-duality}, complementarity, and $-\mu^\star\in\partial\mathcal R_z(c)$. The derivative formula is the one dimensional subgradient identity \cite{Rockafellar1970}. The adjoint equation is the optimality system for the Lagrangian $R_z+\mu^\star C_z$. Both the running risk and terminal risk are componentwise nondecreasing, so their tissue derivatives are nonnegative. Solving \eqref{eq:oxygen-adjoint} backward gives \eqref{eq:oxygen-costate-positive}.

The gas dependent part of the Hamiltonian is
\[
\sum_g\alpha_g(t)
\left[
Q(t)\left(\sum_i k_ip_i(t)\right)F_I(g)
+\mu^\star d_g(z(t))
\right].
\]
Pointwise minimisation over the feasible gas simplex gives \eqref{eq:oxygen-priced-envelope}. For locally Lipschitz penalties, the same argument uses the nonnegative Clarke subgradients of the componentwise nondecreasing risk functions \cite{Clarke1990}.
\end{proof}

\begin{corollary}[Marginal oxygen switching rule]
\label{cor:oxygen-switching}
At a time when two feasible gases $g_H,g_L$ satisfy
\[
F_I(g_H)<F_I(g_L),
\qquad
d_{g_H}(z(t))>d_{g_L}(z(t)),
\]
the higher oxygen gas $g_H$ minimises the two gas switching score precisely when
\begin{equation}\label{eq:oxygen-switch}
A(t)\bigl(F_I(g_L)-F_I(g_H)\bigr)
\ge
\mu^\star\bigl(d_{g_H}(z(t))-d_{g_L}(z(t))\bigr).
\end{equation}
Equality permits relaxed mixing. If $\chi(p)=p$, the aggregate identity $F_I=1-F_{O_2}$ reduces \eqref{eq:oxygen-switch} to
\[
\sum_i k_ip_i(t)\ge\mu^\star.
\]
\end{corollary}

\begin{proof}
Compare the two scores in \eqref{eq:oxygen-priced-envelope}. In the linear dose case, both sides of \eqref{eq:oxygen-switch} contain the common positive factor $Q(t)(F_{O_2}(g_H)-F_{O_2}(g_L))$, which cancels.
\end{proof}

\begin{remark}[Scope of the oxygen shadow price]
\label{rem:oxygen-shadow-scope}
The exact duality statement fixes the depth path and uses relaxed gas controls. A unique minimiser of \eqref{eq:oxygen-priced-envelope} is pure, ties can require time sharing or chattering. For the full hybrid problem, the same score is a necessary condition along normal extremals, but global zero duality gap is not asserted.

When $\mu^\star=0$, the score recovers the minimum inert envelope wherever $A(t)>0$. If $A(t)=0$, first order tissue risk is locally insensitive to gas, although the global comparison in \cref{thm:gas-envelope} still permits a minimum inert replacement in the unbudgeted model. An active dose cap can have $\mu^\star=0$ in a degenerate flat region, so complementarity is not an iff statement. Finally, $C(0)=0$ in \eqref{eq:oxygen-budget} denotes the dose available for the ascent, any pre-ascent exposure is subtracted from $\overline C$.
\end{remark}

\subsection{Well posedness}

\begin{lemma}[Existence and exact pure gas recovery]\label{lem:existence}
Under \cref{ass:env,ass:feas,ass:comp,ass:phi,ass:terminal,ass:kin}, each scalarised problem $(\mathbf P_\lambda)$ admits a minimiser in the relaxed class where the gas is a measurable mixture $\alpha(\cdot)=(\alpha_g(\cdot))_{g\in\calG}$ satisfying
\[
\alpha_g(t)\ge0,
\qquad \sum_{g\in\calG}\alpha_g(t)=1,
\qquad \alpha_g(t)=0\ \text{if }g\notin\calG(z(t)),
\]
and
\[
P_\infty^\alpha(z):=\sum_{g\in\calG}\alpha_g F_I(g)\bigl(\Pa(z)-w\bigr).
\]
The capped problem also admits a relaxed minimiser whenever it is feasible. In the aggregate inert gas model of \cref{ass:comp}, each relaxed minimiser can be replaced, without changing its depth path or terminal time, by a pure minimum inert selector with finitely many switches and no larger risk. Consequently both problems attain their optimum in the pure gas class, and the hard cap $R\le\rho$ is preserved exactly.
\end{lemma}

\begin{proof}
Fix $\lambda>0$ and take a minimising sequence $\pi^m=(z^m,g^m,T^m)$ with $J_\lambda(\pi^m)\downarrow\inf J_\lambda$. A no-hold ascent gives a finite comparison policy, so, after discarding finitely many terms, the horizons $T^m$ are uniformly bounded because $J_\lambda\ge T$. For the capped problem, feasibility supplies at least one finite horizon comparison policy, and the same bounded horizon reduction applies to any minimising sequence.  Pass to a subsequence with $T^m\to T^\star$.

For compactness only, extend all trajectories to a common interval $[0,\bar T]$ by holding at the surface after arrival and using any fixed feasible surface gas, this artificial extension is not part of the profile and contributes neither time nor risk. The depth paths are equi-Lipschitz and bounded, so Arzel\`a--Ascoli gives $z^m\to z^\star$ uniformly along a subsequence. Since $z^m(T^m)=0$, uniform convergence and $T^m\to T^\star$ imply $z^\star(T^\star)=0$. The terminal transit rule is closed under this convergence, after first entering $[0,z_{\mathrm{exit}})$, the remaining time is exactly the remaining depth divided by $\dot z_{\max}$. Hence $T^\star$ is the first surface hit of $z^\star$, rather than the end of an artificial surface hold. The inputs
\[
q^m(t):=P_\infty(g^m(t),z^m(t))
\]
are bounded in $L^\infty(0,\bar T)$ and therefore admit a weak $\ast$ convergent subsequence. Equivalently, since the gas alphabet is finite, the empirical controls generate a Young measure limit $\alpha^\star(t)$ supported on $\calG(z^\star(t))$ a.e., closedness of the feasible graph gives the support statement. The relaxed input is
\[
q^\star(t)=\sum_{g\in\calG}\alpha_g^\star(t)F_I(g)\bigl(\Pa(z^\star(t))-w\bigr).
\]
For each compartment,
\[
P_i^m(t)=e^{-k_it}P_i(0)+k_i\int_0^t e^{-k_i(t-s)}q^m(s)\,ds,
\]
so weak $\ast$ convergence of $q^m$ and compactness of convolution by the exponential kernel imply pointwise, hence uniform, convergence of $P^m$ to the solution driven by $q^\star$. Since $M_i(z)\ge a_i>0$, the map $(z,P)\mapsto S_i(z,P)$ is continuous on the reachable compact set. The running cost is nonnegative and continuous in $(z,P)$. Uniform state convergence and $T^m\to T^\star$ therefore give convergence of the running cost integrals on the active intervals, and in particular
\[
J_\lambda(z^\star,\alpha^\star)\le \liminf_{m\to\infty}J_\lambda(z^m,g^m).
\]
The terminal term converges as well because $P^m(T^m)\to P^\star(T^\star)$ and $\Psi$ is continuous.  Thus a relaxed scalar minimiser exists.  For the capped problem, take a time minimising sequence inside the closed constraint $R\le\rho$.  The same compactness argument applies, and uniform convergence of $(z^m,P^m)$ on the common horizon gives convergence of the risk integrals.  The limit is therefore feasible and time optimal. Finally apply \cref{thm:gas-envelope} to either relaxed minimiser.  The pure envelope selector leaves the depth path and time unchanged and weakly lowers every tissue pressure and the accumulated risk.  It is therefore optimal and, by the finite cell part of \cref{thm:gas-envelope}, has finitely many switches.
\end{proof}

\begin{proposition}[Finite segmented profiles suffice up to $\varepsilon$]\label{prop:segmented}
For every $\varepsilon>0$, there exists a feasible scalarised profile composed of finitely many constant depth holds, separated by max rate ascents and with the pure envelope gas of \cref{thm:gas-envelope}, whose cost $J_\lambda$ is within $\varepsilon$ of the optimal value of $(\mathbf P_\lambda)$. For $(\mathbf P_{\mathrm{cap}})$ the analogous staged approximation satisfies $T\le T^\star+\varepsilon$ and $R\le\rho+\varepsilon$; exact cap preservation follows if the optimum has risk slack.  Gas purity and a finite number of gas switches are exact in the aggregate model, while only the vertical rate staging is approximate.
\end{proposition}

\begin{proof}[Proof sketch]
Use the pure finite switch gas envelope from \cref{thm:gas-envelope}.  Approximate the vertical rate control by pulse width modulation between $0$ and $-\dot z_{\max}$ on a sufficiently fine time partition.  The resulting depth paths converge uniformly, and stability of the linear ODE gives uniform convergence of the tissue trajectories on bounded horizons.  Dominated convergence then applies to both the scalarised running cost and the risk integrand.  This gives the scalarised statement and the stated capped qualification.
\end{proof}

\section{Structural results}
We first identify exactly when monotone ascent follows from the objective and when it can fail.  A local fixed-block exchange is valid in a safe off gassing band, and a continuous rearrangement theorem derives monotonicity for terminal state objectives when the initial depth is maximal.  Outside that scope, an unrestricted bidirectional integrated stress instance admits a re-descent that is strictly better than every monotone profile.

\subsection{When monotone ascent is derived}

\begin{assumption}[Monotone inert fraction along ascent]\label{ass:FI-monotone}
Either (i) there is a single inert species with fixed inert fraction, or (ii) the gas policy satisfies
\[
z_b>z_a\quad\Longrightarrow\quad F_I(g_b)\ge F_I(g_a)
\]
whenever the profile uses segments $(z_a,g_a)$ and $(z_b,g_b)$ consecutively. Equivalently, the oxygen fraction is nondecreasing along ascent. This implies $P_\infty(g_a,z_a)\le P_\infty(g_b,z_b)$ whenever $z_b>z_a$ and the relevant gases satisfy the stated ordering.
\end{assumption}

This hypothesis orders inspired inert pressure with depth.  The ordering controls terminal tissues and, in a safe off gassing band, the running penalty as well.  Outside that band, changing ceilings can reverse the exchange.

\begin{lemma}[Terminal state ordering under swapped constant segments]\label{lem:swap-terminal}
Fix a compartment $i$ and two constant segments $A=(z_a,g_a,\tau_a)$ and $B=(z_b,g_b,\tau_b)$. Let $k:=k_i$ and define the affine segment flow
\[
C_s(x):=P_{\infty,s}+(x-P_{\infty,s})e^{-k\tau_s},\qquad s\in\{a,b\},
\]
where $P_{\infty,s}:=P_\infty(g_s,z_s)$. If $P_{\infty,a}\le P_{\infty,b}$, then for every incoming state $x\in\R_+$,
\[
C_b(C_a(x))-C_a(C_b(x))
=(1-e^{-k\tau_a})(1-e^{-k\tau_b})(P_{\infty,b}-P_{\infty,a})\ge0.
\]
In particular, the terminal tissue tension after executing $A$ then $B$ is at least as large as after executing $B$ then $A$.
\end{lemma}

\begin{proof}
Substitute the two affine maps and expand:
\[
C_b(C_a(x))=P_{\infty,b}+\bigl(C_a(x)-P_{\infty,b}\bigr)e^{-k\tau_b},
\quad
C_a(C_b(x))=P_{\infty,a}+\bigl(C_b(x)-P_{\infty,a}\bigr)e^{-k\tau_a}.
\]
After cancellation, the stated identity remains.
\end{proof}

\begin{proposition}[Safe off gassing exchange for dwell blocks]
\label{prop:safe-offgas-exchange}
Let $A=(z_a,g_a,\tau_a)$ and $B=(z_b,g_b,\tau_b)$ be constant blocks with $z_a<z_b$, $q_a:=P_\infty(g_a,z_a)\le q_b:=P_\infty(g_b,z_b)$, and $M_i(z_a)\le M_i(z_b)$.  Write $C_s$ for the tissue flow of block $s$ and
\[
 I_s(P):=\sum_i\int_0^{\tau_s}
 \phi_i\!\left(
 \frac{(P_{s,i}(t;P_i)-M_i(z_s))_+}{M_i(z_s)}
 \right)dt .
\]
If the incoming state lies in the safe off-gassing band
\[
 q_b\le P_i\le M_i(z_b),\qquad i=1,\ldots,n,
\]
then
\[
 I_B(P)+I_A(C_B(P))\le I_A(P)+I_B(C_A(P)),
 \qquad
 C_A(C_B(P))\le C_B(C_A(P))
\]
componentwise.  Therefore, within the fixed block permutation problem with transit arcs omitted or accounted for separately, replacing $A\to B$ by $B\to A$ cannot increase running stress or the cost of any fixed suffix with componentwise nondecreasing terminal cost.
\end{proposition}

\begin{proof}
Because $q_b\le P_i\le M_i(z_b)$, block $B$ is off gassing and remains below its ceiling, so $I_B(P)=0$ and $C_B(P)\le P$.  Also $q_a\le q_b\le P_i$, hence $C_A(P)\le P$.  When $B$ follows $A$, its tissue trajectory lies between $C_A(P)$ and $q_b$, both no larger than $M_i(z_b)$, thus $I_B(C_A(P))=0$.  Order preservation of the compartment flow and monotonicity of $\phi_i$ give $I_A(C_B(P))\le I_A(P)$.  The terminal state inequality is \cref{lem:swap-terminal}.  A fixed suffix has componentwise nondecreasing cost to go by the same flow ordering and the monotonicity of the running and terminal penalties.
\end{proof}

Within this fixed block abstraction, an undominated inverted pair must leave the band in at least one compartment, it begins on gassing, remains above the deeper ceiling, or both.  The comparison does not include the different transit arcs needed to realise the two orders as continuous depth paths.  The next theorem instead gives a global continuous path result when the pathwise tissue penalty is removed.

\begin{theorem}[Monotone rearrangement for endpoint objectives]
\label{thm:endpoint-monotone}
Temporarily enlarge \cref{ass:kin} by allowing $|\dot z|\le \dot z_{\max}$ above $z_{\mathrm{exit}}$, while retaining $\dot z=-\dot z_{\max}$ on $(0,z_{\mathrm{exit}})$.  Assume that $0\le z(t)\le z_{\mathrm{start}}$ and gas is selected by a fixed measurable feasible depth selector $g_\star(z)$ for which
\[
 q_\star(z):=P_\infty(g_\star(z),z)
\]
is nondecreasing.  Let any running cost be a nonnegative bounded Borel, state independent function $c=c(z,g_\star(z))$, and let $\Psi$ be componentwise nondecreasing.  Then for every admissible profile $z$ with first surface time $T$ there is an admissible nonincreasing profile $z^\downarrow$ with the same $T$ and the same depth occupation measure such that
\[
 P_i^\downarrow(T)\le P_i(T),\qquad i=1,\ldots,n.
\]
Consequently,
\[
 T+\int_0^T c(z^\downarrow(t),g_\star(z^\downarrow(t)))\,dt
   +\lambda\Psi(P^\downarrow(T))
\le
 T+\int_0^T c(z(t),g_\star(z(t)))\,dt+\lambda\Psi(P(T)).
\]
Thus monotone ascent is without loss for terminal state objectives, including a fixed surface tail functional, and likewise for a terminal risk cap.  Any state independent depth exposure, such as an oxygen clock under the fixed selector, is preserved exactly.
\end{theorem}

\begin{proof}
Let $\mu_z(A):=|\{t\in[0,T]:z(t)\in A\}|$ be the depth occupation measure, and let $z^\downarrow$ be the nonincreasing rearrangement of $z$.  By definition, $z^\downarrow$ is equimeasurable with $z$.  It has the same endpoints because $z$ is continuous, starts at its maximum $z_{\mathrm{start}}$, and first reaches its minimum $0$ at $T$. It remains to check the rate constraint.  For $0\le a<b\le z_{\mathrm{start}}$, continuity forces the original path to cross $[a,b]$ at least once.  Since $|\dot z|\le\dot z_{\max}$,
\[
 \mu_z((a,b))\ge \frac{b-a}{\dot z_{\max}}.
\]
In the decreasing rearrangement, the time used to pass from $b$ to $a$ is this occupation mass, up to endpoint atoms, which become holds.  Hence $|z^\downarrow(t)-z^\downarrow(s)|\le \dot z_{\max}|t-s|$.  Moreover the terminal transit rule gives $\mu_z((a,b))=(b-a)/\dot z_{\max}$ whenever $0<a<b<z_{\mathrm{exit}}$; therefore $\dot z^\downarrow=-\dot z_{\max}$ a.e. on that terminal interval. Thus $z^\downarrow$ is feasible and has the same first surface time.

Equimeasurability of depth implies equimeasurability of $q_\star(z(\cdot))$ and $q_\star(z^\downarrow(\cdot))$.  Because both $q_\star$ and $z^\downarrow$ are monotone in opposite directions, $q_\star(z^\downarrow(t))$ is the nonincreasing rearrangement of the input. For each compartment,
\[
 P_i(T)=e^{-k_iT}P_i(0)
 +\int_0^T k_i e^{-k_i(T-s)}q_\star(z(s))\,ds.
\]
The kernel $s\mapsto k_i e^{-k_i(T-s)}$ is increasing.  The Hardy--Littlewood rearrangement inequality \cite{HardyLittlewoodPolya1952}, equivalently repeated use of \cref{lem:swap-terminal} followed by simple function approximation, says that its integral against an equimeasurable input is minimised when the input is arranged in nonincreasing order.  Hence $P_i^\downarrow(T)\le P_i(T)$ simultaneously for all $i$.  The state independent integral is unchanged by equality of occupation measures, and componentwise monotonicity of $\Psi$ proves the claim.  The capped statement is identical.
\end{proof}

\begin{corollary}[Derived monotonicity for the aggregate terminal model]
\label{cor:aggregate-endpoint-monotone}
Consider the unbudgeted aggregate model with objective $T+\lambda\Psi(P(T))$ and the enlarged bidirectional kinematics of \cref{thm:endpoint-monotone}.  Suppose a tie broken minimum inert selector $g_{\min}(z)\in\calG_{\min}(z)$ has nondecreasing input
\[
q_{\min}(z)=F_I^{\min}(z)(\Pa(z)-w).
\]
Then allowing re-descents no deeper than $z_{\mathrm{start}}$ does not improve the optimal value: the problem has a monotone ascent minimiser whenever it has a minimiser at all.
\end{corollary}

\begin{proof}
The comparison argument in \cref{thm:gas-envelope} applies unchanged to a bidirectional depth path and first replaces any gas control by $g_{\min}(z(t))$, weakly lowering every terminal tissue pressure.  Apply \cref{thm:endpoint-monotone} to the resulting depth selected input.  The oxygen budget extension is excluded because it makes the first replacement change a path coupled resource.
\end{proof}

The hypotheses do real work, pathwise ceiling stress is not invariant under occupation measure rearrangement, and the construction below also permits an excursion deeper than the starting depth.

\begin{proposition}[Failure in the unrestricted bidirectional extension]
\label{prop:monotonicity-impossible}
If bidirectional motion and excursions deeper than $z_{\mathrm{start}}$ are admitted above $z_{\mathrm{exit}}$, the assumptions of the base model do not imply the existence of a monotone optimal profile.  Indeed, there are admissible one compartment data with integrated oversaturation for which a re-descent has strictly smaller scalarised cost than every monotone profile.
\end{proposition}

\begin{proof}
Take $k=1$, $\Pa(z)=1+z$, and $M(z)=1+z/2$, corresponding to $a=b=1/2$.  Let $w=0.05$ and use the pure oxygen alphabet $\calG=\{\Otwo\}$ with $\eta=0$, ppO$_2$ window $[0.9,3]$, and $\overline{\END}=2$.  Then $q(z)\equiv0$, ppO$_2$ ranges from $0.95$ to $2.95$, $\END=-0.95$, and the selector is feasible on $[0,2]$.  Set $\phi(s)=s$, $\Psi(P)=P$, $Z_{\max}=2$, $z_{\mathrm{start}}=1/5$, $z_{\mathrm{exit}}=1/10$, $\dot z_{\max}=100$, and $P(0)=2$.  Along every profile, $P(t)=2e^{-t}$. For a monotone profile, $z(t)\le1/5$ and hence $M(z(t))\le11/10$.  Put
\[
 t_0=\log(20/11),\qquad
 r_0:=\int_0^{t_0}\left(\frac{2e^{-t}}{11/10}-1\right)dt
 =\frac9{11}-\log\frac{20}{11}
 =0.220344817\ldots .
\]
If $T\ge t_0$, its integrated risk is at least $r_0$.  If $T<t_0$, its terminal risk is $\Psi(P(T))>11/10>r_0$.  Thus every monotone profile has $R\ge r_0$ and $J_\lambda\ge\lambda r_0$.

Now ascend at maximal rate from $1/5$ to $199/1000$, re-descend at maximal rate to $2$, hold there for four time units, and ascend at maximal rate to the surface.  Its duration is $T_{\mathrm{rd}}=4.03802$.  During the first $0.01802$ time units, $M(z)\ge M(199/1000)=2199/2000$ and $P(t)\le2$, so $S(t)\le1801/2199$.  The deep hold and final ascent have zero running risk, at depth $2$, $M(2)=2>P$, and at the start of the final ascent $P<M(0)=1$.  Hence
\[
 R_{\mathrm{rd}}
 \le 0.01802\frac{1801}{2199}+2e^{-4.03802}
 =0.050023236\ldots<r_0.
\]
At $\lambda=24$,
\[
 J_{24}^{\mathrm{rd}}<4.03802+24(0.050023236)=5.238578
 <24r_0=5.288276\le
 \inf_{\text{monotone }\pi}J_{24}(\pi).
\]
Thus convexity, ordered ceilings, and a componentwise nondecreasing terminal functional do not derive monotonicity in the unrestricted bidirectional extension once integrated oversaturation is retained.  The improving path uses a deeper than start excursion, so it is outside the no deeper hypothesis of \cref{thm:endpoint-monotone}.
\end{proof}

\subsection{Rate saturation and staged density}

\begin{lemma}[Rate saturation and staged density]\label{lem:bangbang}
Consider an open arc contained in the interior of one gas feasibility cell, so that the feasible gas set is locally constant.  Any Pontryagin extremal for which the maximum principle applies has, on every nonsingular portion of that arc,
\[
\dot z(t)\in\{0,-\dot z_{\max}\}
\]
a.e.  Independently of this necessary condition, the set of profiles using only the two rates $\{0,-\dot z_{\max}\}$ is dense in cost in the full admissible rate class.  Consequently, for every $\varepsilon>0$ there exists an $\varepsilon$ optimal scalarised profile consisting of finitely many constant depth holds separated by max rate ascents.
\end{lemma}

\begin{proof}
The vertical rate control is $u(t)=\dot z(t)\in[-\dot z_{\max},0]$. On the interior of a feasibility cell the control set is fixed.  The Hamiltonian for a Pontryagin extremal is
\[
H=\lambda_0\ell(z,P,g)+p_z u+\sum_{i=1}^n p_i k_i\bigl(P_\infty(g,z)-P_i\bigr).
\]
In the normal case $\lambda_0=1$, the endpoint conclusion is unchanged for an abnormal multiplier because the dependence of $H$ on $u$ remains affine. Pointwise minimisation over $[-\dot z_{\max},0]$ gives
\[
u^\star(t)=
\begin{cases}
-\dot z_{\max}, & p_z(t)>0,\\[0.3em]
0, & p_z(t)<0,\\[0.3em]
\text{any value in }[-\dot z_{\max},0], & p_z(t)=0.
\end{cases}
\]
Hence every nonsingular portion of such an extremal uses an endpoint rate. For the density statement, partition time into cells and replace an arbitrary $u(t)\in[-\dot z_{\max},0]$ on each cell by a hold followed by a max rate ascent whose duty cycle has the same average rate.  Refine the partition at every gas feasibility boundary reached by the original monotone path.  The resulting depth paths converge uniformly.  Choose the gas envelope of \cref{thm:gas-envelope} on every approximating path. Away from the finitely many cell boundaries the inspired pressure inputs converge in $L^1$, holds at a boundary can be copied exactly. Stability of the compartment ODE then gives uniform convergence of $P$, and dominated convergence gives convergence of the integrated cost. On the terminal layer $0<z<z_{\mathrm{exit}}$, copy the required max rate ascent exactly rather than pulse width modulating it. Applying this construction to an optimal profile proves the $\varepsilon$ optimal staged statement.
\end{proof}

If singular arcs occur ($p_z=0$ on a set of positive measure), they can be approximated arbitrarily well in cost by hold/ascent chattering, yielding the staged structure up to $\varepsilon$-optimality. Between holds, the vertical rate decision is then trivial, ascend at the rate cap. The computational takeaway is that, after choosing a gas policy, the main continuous decisions are hold depths and dwell times. This density statement should not be read as evidence that staged profiles are operationally superior, continuous ceiling following profiles can be shorter in comparative simulations \cite{Angelini2022Ceiling}. It complements, rather than supersedes, earlier mathematical analyses of admissible ascent strategies \cite{Lewis1983OptimalDecompression,Blasselle2019AdmissiblePressure}.

\subsection{A no stop phase for small risk prices}

Let
\[
r(z,P):=\sum_{i=1}^n\phi_i(S_i(z,P)),\qquad
\Theta:=\frac{z_{\mathrm{start}}}{\dot z_{\max}}.
\]
On the reachable compact pressure set, let $L_r$ and $L_\Psi$ be Lipschitz constants of $P\mapsto r(z,P)$, uniformly in $z$, and of $\Psi$, respectively. Set
\[
q_{\max}:=\max_{z,g}P_\infty(g,z),\qquad
\bar P:=\max\{\|P^0\|_\infty,q_{\max}\},
\]
\[
B_f:=\sum_{i=1}^n k_i(\bar P+q_{\max}),\qquad
B:=B_f(L_r\Theta+L_\Psi).
\]

\begin{theorem}[Explicit sufficient no stop threshold]\label{thm:no-stop}
Let $\pi_0$ be the no hold, max rate ascent using the envelope gas. For any staged profile $\pi$ using the envelope gas and having total dwell time $H$, one has
\begin{equation}\label{eq:no-stop-risk-bound}
R(\pi_0)-R(\pi)\le BH
\end{equation}
and therefore
\begin{equation}\label{eq:no-stop-cost-bound}
J_\lambda(\pi)-J_\lambda(\pi_0)\ge(1-\lambda B)H.
\end{equation}
Consequently, if $B=0$ or $\lambda\le1/B$, the no hold profile is globally optimal. If $\lambda<1/B$, every optimal profile has minimum possible terminal time $\Theta$ and therefore ascends at the rate cap almost everywhere.
\end{theorem}

\begin{proof}
Delete all holds from $\pi$ while retaining its max rate transit arcs and their envelope gas as a function of depth. The concatenated transit is exactly $\pi_0$. During any hold of duration $h$, the tissue state displacement has $\ell_1$ norm at most $B_fh$. Under a common subsequent transit input, the diagonal linear compartment flow is nonexpansive in $\ell_1$. Induction over the segments therefore gives
\[
\|P_{\pi}(s)-P_{\pi_0}(s)\|_1\le B_fH
\]
at every pair of depth matched transit points and at the surface. The held profile accrues nonnegative risk during its holds. Thus the amount by which it can improve on the no hold risk is bounded by the difference between their transit and terminal risks:
\[
R(\pi_0)-R(\pi)
\le L_r\Theta B_fH+L_\Psi B_fH=BH.
\]
Since $T(\pi)-T(\pi_0)=H$, \eqref{eq:no-stop-cost-bound} follows. Staged profiles are dense in the full admissible rate class by \cref{lem:bangbang}, so the global conclusion follows by passage to the limit. Strict inequality forces $H=0$ for any optimum.
\end{proof}

\subsection{Finite stop grid suffices}

\begin{lemma}[Finite stop grid suffices]\label{lem:grid}
For any $\varepsilon>0$, there exists a finite depth grid $\calZ_\varepsilon\subset[0,Z_{\max}]$ such that the scalarised problem admits an $\varepsilon$ optimal profile whose stops lie in $\calZ_\varepsilon$. For the capped problem the same statement holds with the risk qualification $R\le\rho+\varepsilon$, or exactly if an optimal profile has positive risk slack.
\end{lemma}

\begin{proof}
By \cref{lem:bangbang}, consider profiles composed of holds and max rate ascents. The ppO$_2$ and END inequalities are linear in $\Pa(z)$ for each fixed gas, so each gas is feasible on a finite union of depth intervals and the finite gas alphabet induces a finite partition of $[0,Z_{\max}]$ into feasibility cells. Include all cell boundaries and $z_{\mathrm{exit}}$ in the grid, and permit no positive dwell below $z_{\mathrm{exit}}$. On each feasibility cell the feasible gas set is constant.

On reachable compact sets, the map $(z,P)\mapsto\sum_i\phi_i(S_i)$ is locally Lipschitz because $M_i\ge a_i>0$ and each $\phi_i$ is locally Lipschitz. Linear ODE stability gives Lipschitz dependence of $P(\cdot)$ on stop depth and dwell time. Snapping a stop to the nearest grid point within the same feasibility cell therefore perturbs the integrand by $O(\Delta z)$ uniformly; over a bounded total dwell horizon this changes the running risk by $O(\Delta z)$, while local Lipschitzness of $\Psi$ gives the same order for terminal risk. The total max-rate transit time remains exactly $z_{\mathrm{start}}/\dot z_{\max}$. Choosing $\Delta z$ sufficiently small gives the stated approximation. Positive cap slack absorbs the risk perturbation.
\end{proof}

Formally, one may bound $\phi_i'$ on a compact oversaturation interval $[0,\bar S]$ and use Gr\"onwall's inequality to show $|S_i(z)-S_i(\tilde z)|\le C_i|z-\tilde z|$. Then $|\phi_i(S_i(z))-\phi_i(S_i(\tilde z))|\le L_{\phi}C_i|z-\tilde z|$, giving an $O(\Delta z)$ perturbation uniformly over bounded horizons. The inclusion of feasibility cell boundaries is essential, without it, snapping a stop could move a chosen gas across a ppO$_2$ or END boundary.

\subsection{Dwell time optimality and exact marginals}

Numerical stop time optimisation has important predecessors
\cite{Horn2003Optimization,Feng2010Multiparametric}. The point of this subsection is the analytic marginal identity and its nonsmooth optimality consequences, not the first use of optimised dwell times. Fix a stop sequence $\{(z_j,g_j)\}_{j=1}^m$, fix the gas rule on every transit arc, and let $\tau\in\R_+^m$ denote dwell times. In the aggregate model the complete gas rule may, without loss, be the envelope selector from \cref{thm:gas-envelope}. Between stops, ascend at $-\dot z_{\max}$.

\begin{proposition}[Existence and Clarke stationarity in dwell times]\label{prop:polish}
Fix a stop sequence $\{(z_j,g_j)\}_{j=1}^m$ with depths nonincreasing, including a complete feasible transit gas rule, and let $\tau\in\R_+^m$ be the dwell times. Between stops ascend at $-\dot z_{\max}$. Under \cref{ass:phi,ass:terminal}, the map
\[
\tau\mapsto J_\lambda(\tau)=\sum_{j=1}^m\tau_j+\lambda R(\tau)
\]
(up to the additive constant transit time) is locally Lipschitz and coercive, and therefore has a global minimiser. Any local minimiser $\tau^\star$ of $J_\lambda$ subject to $\tau\ge0$ satisfies the Clarke KKT condition
\begin{equation}\label{eq:kkt-dwell-vector}
0\in \boldsymbol 1+\lambda \partial_C R(\tau^\star)+N_{\R_+^m}(\tau^\star),
\end{equation}
where $\partial_C$ is the Clarke subdifferential and $N_{\R_+^m}$ is the Clarke normal cone. Equivalently, there exists a Clarke subgradient vector $\xi\in\partial_C R(\tau^\star)$ such that
\begin{equation}\label{eq:kkt-dwell-coord}
\begin{cases}
1+\lambda\xi_j=0, & \text{if }\tau_j^\star>0,\\[0.3em]
1+\lambda\xi_j\ge0, & \text{if }\tau_j^\star=0.
\end{cases}
\end{equation}

If, in addition, the $\phi_i$ are piecewise linear convex, then $R$ is directionally differentiable in the dwell variables and the inactive condition can also be read as the one sided necessary condition $D_j^+J_\lambda(\tau^\star)\ge0$ for adding dwell at an inactive stop.

In general $R(\tau)$ is not convex in $\tau$, for example, one compartment, one hold, and $\phi(s)=s$ yield a concave $R$ until the zero oversaturation hitting time. The statement above provides first order necessary conditions, not sufficient global optimality conditions.
\end{proposition}

\begin{proof}
Local Lipschitzness follows from Lipschitz dependence of the linear ODE flow on dwell durations, the local Lipschitz property of the running penalties, and local Lipschitzness of $\Psi$. Since $J_\lambda(\tau)\ge\|\tau\|_1$, it is coercive; continuity and the direct method on $\R_+^m$ give global attainment. The Clarke KKT condition \eqref{eq:kkt-dwell-vector} is the standard necessary condition for a local minimum of a locally Lipschitz function over a closed convex set \cite{Clarke1990}. Since the smooth time term contributes $\boldsymbol 1$, the remaining nonsmooth contribution is $\lambda\partial_C R(\tau^\star)$. For the orthant, the normal cone satisfies $(N_{\R_+^m}(\tau^\star))_j=\{0\}$ when $\tau_j^\star>0$ and $(-\infty,0]$ when $\tau_j^\star=0$, giving \eqref{eq:kkt-dwell-coord}. Under piecewise linearity, the compositions of the exponential dwell time flows with the finitely many linear pieces of $\phi_i$ are directionally differentiable, which gives the stated one sided interpretation at inactive stops.
\end{proof}

\begin{remark}[What Clarke equalisation does and does not say]\label{rem:kkt}
Let $\tau^\star$ be a local minimiser as in \cref{prop:polish}. There exists a Clarke subgradient vector $\xi\in\partial_C R(\tau^\star)$ such that, for every active stop $j$ with $\tau_j^\star>0$, $1+\lambda\xi_j=0$. Hence for any two active stops $j,k$, $-\xi_j=-\xi_k=1/\lambda$. Thus the coordinates of one selected generalised gradient equalise. At a nonsmooth point they need not be unique physical marginal derivatives. A genuine marginal identity follows from \cref{thm:dwell-switching} whenever the relevant suffix value is differentiable. Inactive stops have no positive one sided move with negative first order scalarised cost.
\end{remark}

\subsection{An exact dwell switching identity}

For a constant depth, constant gas hold $j$, write
\[
q_j:=P_\infty(g_j,z_j),\qquad
K:=\operatorname{diag}(k_1,\ldots,k_n),
\]
and define
\[
f_j(P):=K(q_j\boldsymbol 1-P),\qquad
r_j(P):=\sum_{i=1}^n
\phi_i\!\left(\frac{(P_i-M_i(z_j))_+}{M_i(z_j)}\right).
\]
The tissue flow during the hold is
\[
\Phi_j^s(P)=q_j\boldsymbol 1+e^{-Ks}(P-q_j\boldsymbol 1).
\]

\begin{theorem}[Exact marginal value of a dwell]\label{thm:dwell-switching}
Fix a segmented schedule and vary only its dwell times. Let $X_j$ and $Y_j=\Phi_j^{\tau_j}(X_j)$ be the entry and exit tissue states at stop $j$. Let $W_{j+1}(y)$ be the total risk of the prescribed suffix when it is entered with tissue state $y$. The suffix includes all later transits and holds and the terminal functional $\Psi$.

Suppose that $r_j$ and $W_{j+1}$ are differentiable at the relevant states. Then
\begin{equation}\label{eq:exact-dwell-derivative}
\frac{\partial R}{\partial\tau_j}
=r_j(Y_j)+\nabla W_{j+1}(Y_j)^\top f_j(Y_j).
\end{equation}
Consequently, at an active locally optimal dwell,
\begin{equation}\label{eq:dwell-benefit-balance}
\nabla W_{j+1}(Y_j)^\top K(Y_j-q_j\boldsymbol 1)
=r_j(Y_j)+\frac1\lambda.
\end{equation}
Thus the marginal reduction in future risk equals the instantaneous risk incurred during the added dwell plus the time price $1/\lambda$.

Moreover, $W_{j+1}$ is componentwise nondecreasing. Hence $\nabla W_{j+1}(y)\ge0$ componentwise wherever it is differentiable. All dwell derivatives can be computed in one forward state sweep and one backward adjoint sweep.
\end{theorem}

\begin{proof}
Up to terms independent of $\tau_j$,
\[
R(\tau)=
\int_0^{\tau_j}r_j\bigl(\Phi_j^s(X_j)\bigr)\,ds
+W_{j+1}\bigl(\Phi_j^{\tau_j}(X_j)\bigr).
\]
Leibniz' rule and $\partial_s\Phi_j^s(X_j)=f_j(\Phi_j^s(X_j))$ give \eqref{eq:exact-dwell-derivative}. At an active local minimum, $0=1+\lambda\,\partial R/\partial\tau_j$. Substituting $f_j(Y_j)=K(q_j\boldsymbol1-Y_j)$ gives \eqref{eq:dwell-benefit-balance}.

If $y\le\widetilde y$ componentwise, applying the same suffix to both states preserves this order under the linear compartment dynamics. The running penalties and $\Psi$ are componentwise nondecreasing, so $W_{j+1}(y)\le W_{j+1}(\widetilde y)$. This proves the monotonicity claim. For computation, propagate the state forward. Propagate the gradient of the suffix risk backward through each fixed segment. On a hold this recursion is
\[
\nabla W_j(x)
=\int_0^{\tau_j}e^{-Ks}
\nabla r_j\bigl(\Phi_j^s(x)\bigr)\,ds
+e^{-K\tau_j}
\nabla W_{j+1}\bigl(\Phi_j^{\tau_j}(x)\bigr),
\]
with terminal condition $\nabla W_{m+1}=\nabla\Psi$. Fixed ascent arcs use the corresponding linear adjoint flow. Each sweep is linear in the number of compartments times the number of segments.
\end{proof}

\begin{corollary}[Stop screening and on gassing dominance]
\label{cor:stop-screening}
At an inactive candidate stop, the sign of
\[
1+\lambda\left[
r_j(X_j)+\nabla W_{j+1}(X_j)^\top f_j(X_j)
\right]
\]
is the exact first order screening test whenever the derivative exists, a negative sign certifies an improving added dwell, while a positive sign is necessary for local inactivity. If $X_j\le q_j\boldsymbol1$ componentwise, every positive dwell at stop $j$ is strictly dominated by zero dwell. Therefore no optimal fixed sequence schedule contains a purely on gassing hold.
\end{corollary}

\begin{proof}
The screening claim is \eqref{eq:exact-dwell-derivative} at $\tau_j=0$. If $X_j\le q_j\boldsymbol1$, then $\Phi_j^s(X_j)\ge X_j$ componentwise for every $s>0$. Removing the hold removes nonnegative running risk, strictly reduces time, and presents the suffix with a componentwise no larger tissue state. Suffix monotonicity from \cref{thm:dwell-switching} completes the proof.
\end{proof}

\subsection{An exactly solvable one compartment stop problem}
\label{subsec:exact-stop}

The adjoint identity becomes closed form in a useful reduced model. Consider one compartment with rate $k>0$. A safe upstream dwell of length $\tau$ changes the state entering a prescribed shallow phase according to
\begin{equation}\label{eq:reduced-stop-state}
x(\tau)=x_\infty+(x_0-x_\infty)e^{-k\tau},
\qquad x_\infty<M<x_0.
\end{equation}
The dwell is called safe because its own ceiling lies above the trajectory, so it incurs no running penalty. During a subsequent shallow phase of fixed duration $L$, let the inspired pressure be $q<M$, the ceiling be $M$, and $\phi(s)=s$. Define the downstream risk kernel
\[
\Psi_L(x):=\int_0^L
\frac{\bigl(q+(x-q)e^{-kt}-M\bigr)_+}{M}\,dt.
\]

\begin{lemma}[Exact downstream risk kernel]\label{lem:exact-risk-kernel}
Set $U:=q+(M-q)e^{kL}$. Then
\begin{equation}\label{eq:exact-risk-kernel}
\Psi_L(x)=
\begin{cases}
0, & x\le M,\\[0.5em]
\displaystyle
\frac{x-M-(M-q)\log\!\left(\frac{x-q}{M-q}\right)}{kM},
& M<x<U,\\[1em]
\displaystyle
\frac{(q-M)L}{M}
+\frac{(x-q)(1-e^{-kL})}{kM},
& x\ge U.
\end{cases}
\end{equation}
The function is continuously differentiable, convex, and nondecreasing, with
\[
\Psi_L'(x)=
\begin{cases}
0,&x\le M,\\[0.3em]
\displaystyle\frac{x-M}{kM(x-q)},&M<x<U,\\[0.8em]
\displaystyle\frac{1-e^{-kL}}{kM},&x\ge U.
\end{cases}
\]
\end{lemma}

\begin{proof}
If $x\le M$, the trajectory never exceeds the ceiling. If $M<x<U$, it hits the ceiling at
\[
h(x)=\frac1k\log\!\left(\frac{x-q}{M-q}\right)<L.
\]
Integrating only over $[0,h(x)]$ gives the middle branch. If $x\ge U$, oversaturation persists throughout $[0,L]$ and direct integration gives the last branch. The values and first derivatives agree at $M$ and $U$, while
\[
\Psi_L''(x)=\frac{M-q}{kM(x-q)^2}>0
\]
on the middle branch.
\end{proof}

\begin{theorem}[Exact optimal dwell]\label{thm:exact-one-compartment-dwell}
For $\lambda>0$, let
\[
J_\lambda(\tau):=\tau+\lambda\Psi_L(x(\tau)),\qquad \tau\ge0.
\]
This objective is convex. Define
\[
H_\lambda(x):=\lambda k(x-x_\infty)\Psi_L'(x).
\]
If $H_\lambda(x_0)\le1$, the unique optimum is $\tau_\lambda^\star=0$. Otherwise there is a unique $x_\lambda^\star\in(M,x_0)$ satisfying
\begin{equation}\label{eq:one-compartment-stationarity}
H_\lambda(x_\lambda^\star)=1,
\end{equation}
and
\begin{equation}\label{eq:one-compartment-dwell}
\tau_\lambda^\star
=\frac1k\log\!\left(
\frac{x_0-x_\infty}{x_\lambda^\star-x_\infty}
\right).
\end{equation}
The state is explicit. Put
\[
x_F:=x_\infty+\frac{M}{\lambda(1-e^{-kL})}.
\]
If $x_0\ge U$ and $x_F\ge U$, then $x_\lambda^\star=x_F$. Otherwise
\[
x_\lambda^\star
=\frac{x_\infty+M+M/\lambda+
\sqrt{(x_\infty+M+M/\lambda)^2
-4(x_\infty M+(M/\lambda)q)}}{2}.
\]
\end{theorem}

\begin{proof}
The map $\tau\mapsto x(\tau)$ is convex, and $\Psi_L$ is convex and nondecreasing, so $J_\lambda$ is convex. Moreover,
\[
J_\lambda'(\tau)
=1-\lambda k(x(\tau)-x_\infty)\Psi_L'(x(\tau))
=1-H_\lambda(x(\tau)).
\]
The map $H_\lambda$ is nondecreasing in $x$, while $x(\tau)$ is strictly decreasing. The boundary and unique interior conclusions follow. On the full exposure branch, \eqref{eq:one-compartment-stationarity} gives $x=x_F$. On the partial hit branch it becomes
\[
(x-x_\infty)(x-M)=\frac{M}{\lambda}(x-q),
\]
whose unique root above $M$ is the displayed expression.
\end{proof}

\begin{corollary}[Exact dwell under a risk cap]\label{cor:exact-risk-cap-dwell}
If $0\le\rho<\Psi_L(x_0)$, the minimum dwell satisfying $\Psi_L(x(\tau))\le\rho$ is
\[
\tau_\rho=\frac1k\log\!\left(
\frac{x_0-x_\infty}{x_\rho-x_\infty}
\right),
\qquad
\Psi_L(x_\rho)=\rho.
\]
On the partial-hit branch,
\[
x_\rho
=q-(M-q)W_{-1}\!\left(
-\exp\!\left[-1-\frac{kM\rho}{M-q}\right]\right),
\]
where $W_{-1}$ is the lower real branch of the Lambert function. On the full exposure branch,
\[
x_\rho=q+\frac{kM\rho+k(M-q)L}{1-e^{-kL}}.
\]
The first branch applies for $0\le\rho\le\Psi_L(U)$ and the second for $\rho\ge\Psi_L(U)$, restricted in both cases to
$\rho<\Psi_L(x_0)$. The optimal scalarised dwell is nondecreasing in $\lambda$ and $L$, while the capped dwell is nonincreasing in $\rho$.
\end{corollary}

\begin{proof}
The cap binds in the stated range because $\Psi_L$ is strictly increasing above $M$ and $x(\tau)$ is strictly decreasing. On the partial hit branch, put $y=(x-q)/(M-q)$. The equation $\Psi_L(x)=\rho$ becomes
\[
y-1-\log y=\frac{kM\rho}{M-q},
\]
whose solution $y\ge1$ is given by $W_{-1}$. The full exposure inverse is affine. The comparative statics follow from \eqref{eq:one-compartment-stationarity} and monotonicity of $\Psi_L'(x)$ in the shallow-phase length $L$.
\end{proof}

\begin{example}[Closed form numerical instance]\label{ex:exact-numerical}
Take the dimensionless parameters
\[
k=\frac{\log2}{20},\quad L=10,\quad q=0.4,\quad M=1,\quad
x_\infty=0.6,\quad x_0=1.6.
\]
Then
\[
U=1.24853,\qquad
\Psi_L(U)=1.17101,\qquad
\Psi_L(x_0)=4.14133.
\]
For $\lambda=20$,
\[
x_{20}^\star=1.07122,\qquad
\tau_{20}^\star=21.71046.
\]
For the cap $\rho=1$,
\[
x_\rho=1.22766,\qquad
\tau_\rho=13.43872.
\]
\Cref{fig:exact-stop} plots the exact risk kernel and three scalarised objectives.
\end{example}

\begin{figure}[t]
\centering
\pgfmathdeclarefunction{riskexample}{1}{%
  \pgfmathparse{ifthenelse(#1<=1,0,
  ifthenelse(#1<1.248528137,
  (#1-1-0.6*ln((#1-0.4)/0.6))/0.034657359,
  -6+(#1-0.4)*(1-exp(-0.3465735903))/0.034657359))}%
}
\pgfmathdeclarefunction{xexample}{1}{%
  \pgfmathparse{0.6+exp(-0.034657359*#1)}%
}
\begin{subfigure}[t]{0.48\linewidth}
\centering
\begin{tikzpicture}
\begin{axis}[
  width=\linewidth,
  height=5.2cm,
  xlabel={entry state $x$},
  ylabel={downstream risk $\Psi_L(x)$},
  xmin=0.6,xmax=1.65,
  ymin=0,ymax=4.6,
  grid=major,
  tick label style={font=\small},
  label style={font=\small}]
\addplot[very thick,blue,domain=0.6:1.65,samples=180]
  {riskexample(x)};
\addplot[dashed,gray] coordinates {(1,0) (1,4.6)};
\addplot[dashed,gray] coordinates {(1.248528137,0) (1.248528137,4.6)};
\addplot[only marks,mark=*,red] coordinates
  {(1.071221445,0.113100805) (1.227663749,1)};
\node[anchor=west,font=\scriptsize] at (axis cs:1.005,4.2) {$M$};
\node[anchor=west,font=\scriptsize] at (axis cs:1.255,4.2) {$U$};
\end{axis}
\end{tikzpicture}
\caption{Exact downstream risk kernel.}
\end{subfigure}
\hfill
\begin{subfigure}[t]{0.48\linewidth}
\centering
\begin{tikzpicture}
\begin{axis}[
  width=\linewidth,
  height=5.2cm,
  xlabel={dwell $\tau$},
  ylabel={$J_\lambda(\tau)$},
  xmin=0,xmax=32,
  ymin=15,ymax=90,
  grid=major,
  legend style={font=\scriptsize,at={(0.98,0.98)},anchor=north east},
  tick label style={font=\small},
  label style={font=\small}]
\addplot[very thick,teal!70!black,domain=0:32,samples=180]
  {x+5*riskexample(xexample(x))};
\addlegendentry{$\lambda=5$}
\addplot[very thick,orange!85!black,domain=0:32,samples=180]
  {x+10*riskexample(xexample(x))};
\addlegendentry{$\lambda=10$}
\addplot[very thick,purple,domain=0:32,samples=180]
  {x+20*riskexample(xexample(x))};
\addlegendentry{$\lambda=20$}
\addplot[only marks,mark=*,teal!70!black] coordinates
  {(11.00750,18.31251)};
\addplot[only marks,mark=*,orange!85!black] coordinates
  {(17.92786,21.86570)};
\addplot[only marks,mark=*,purple] coordinates
  {(21.71046,23.97247)};
\end{axis}
\end{tikzpicture}
\caption{Convex objectives and exact minimisers.}
\end{subfigure}
\caption{The closed form one compartment instance of
\cref{ex:exact-numerical}. Dashed lines mark the two analytic branches; filled points mark exact scalar or capped solutions.}
\label{fig:exact-stop}
\end{figure}

\begin{example}[Finite menu nonconvexity]\label{ex:nonconvex-menu}
The nonconvexity claim is already visible in the finite menu problem solved by the algorithms. Consider a single layer with three mutually exclusive feasible choices whose evaluated increments are
\[
A=(T,R)=(1,6),\qquad B=(4,4),\qquad C=(6,1).
\]
All three points are Pareto efficient and none has both smaller time and smaller risk than another. However, $B$ is unsupported. Indeed,
\[
B\preceq_\lambda A\iff 4+4\lambda\le1+6\lambda\iff \lambda\ge3/2,
\]
and
\[
B\preceq_\lambda C\iff 4+4\lambda\le6+\lambda\iff \lambda\le2/3.
\]
No $\lambda>0$ makes $B$ optimal for the scalarisation $T+\lambda R$. A capped formulation with, for example, $R\le4$ can nevertheless select $B$. This proves nonconvexity for an arbitrary finite menu. It does not, by itself, prove that all three increments are realisable by the continuous compartment model.
\end{example}

Whenever an operational extension or discretisation produces such mutually exclusive menu choices, scalarisations trace only supported pieces and capped problems can land on unsupported efficient points. The worked instance in \cref{sec:worked-example} reports the realised frontier of the stated compartment model rather than inferring it from this abstract menu.

\section{State representation and certified dynamic programming}
\label{sec:algorithms}

Dynamic programming on a finite layered graph is standard. The model specific questions are which tissue statistic makes the recursion valid and which rounding directions preserve a hard cap. We answer those questions before counting states, then resolve the exact label/FPTAS boundary with compartment generated rather than arbitrary arc increments. Resource coordinates, such as accumulated risk and the oxygen dose in \eqref{eq:oxygen-budget}, are appended separately. Real time probabilistic updating, receding horizon control, and shortest path schedule calculations have appeared previously \cite{Survanshi1996RealTime,Feng2012RecedingHorizon,Murphy2017Dissertation,DiMuro2023Adaptive}. Here, \cref{lem:grid,lem:bangbang} reduce the search to a finite depth grid $\calZ$ and maximum rate ascent arcs. A path is
\[
z_0\to z_1\to\cdots\to z_m=0,
\qquad z_{j+1}\le z_j,
\]
with a finite dwell menu $\mathcal T_j$ and gas $g_j\in\calG(z_j)$ at each node, $\mathcal T_j=\{0\}$ strictly below $z_{\mathrm{exit}}$.

\subsection{The state that the recursion must retain}

\begin{theorem}[State obstruction and safe tissue dominance]
\label{thm:state-obstruction}
At a fixed depth $z$, the full tissue vector $P$, or equivalently the full normalised tension $\theta_i=P_i/M_i(z)$, is a sufficient physiological coordinate for state propagation. The clipped vector
\[
\sigma_i=(\theta_i-1)_+
\]
is not a sufficient Markov state in general. In particular, no exact Bellman recursion valid over the full admissible model class may identify labels only through $\sigma$, even when both labels have $\sigma=0$. More precisely, suppose that for some compartment $i$ and feasible hold gas $g$ at depth $z$,
\[
q:=P_\infty(g,z)>M:=M_i(z).
\]
Then there are two incoming pressures $P_i^-<P_i^+<M$ and a common feasible dwell whose two outgoing clipped oversaturations differ.

Conversely, componentwise tissue dominance is safe. If two labels at the same depth satisfy $P\le P'$, then under every common feasible continuation their tissue trajectories satisfy $P(t)\le P'(t)$. Their future running and terminal risks are ordered in the same direction. Thus, after also comparing the accumulated objective or resource coordinates, a label with no larger full tissue vector safely dominates the other.
\end{theorem}

\begin{proof}
For the obstruction, fix the compartment above and choose
\[
0<\tau<\frac1{k_i}\log\!\left(\frac{q}{q-M}\right),
\qquad
L_\tau:=q+(M-q)e^{k_i\tau}.
\]
Then $0<L_\tau<M$. Take $P_i^-=0$ and any $P_i^+\in(L_\tau,M)$, with any remaining coordinates equal and below their ceilings. Both incoming labels have $\sigma=0$. The common dwell transition is
\[
F_\tau(P)=q+(P-q)e^{-k_i\tau}.
\]
By construction,
\[
F_\tau(P_i^-)=q(1-e^{-k_i\tau})<M,
\qquad
F_\tau(P_i^+)>M.
\]
Thus the same clipped state and action produce different outgoing clipped states. If $\phi_i(s)>0$ for $s>0$, as for $\phi_i(s)=c_is^2$, the second trajectory also incurs positive risk on a terminal subinterval of the dwell while the first incurs none. Clipping therefore loses information needed both for propagation and, for nondegenerate penalties, for future cost.

For dominance, run any common continuation from $P\le P'$. The compartment difference satisfies
\[
\frac{d}{dt}(P_i'-P_i)=-k_i(P_i'-P_i),
\qquad
P_i'(t)-P_i(t)=e^{-k_it}(P_i'-P_i)\ge0.
\]
Normalised oversaturation is nondecreasing in tissue pressure. The monotonicity of $\phi_i$ and $\Psi$ therefore orders every future running increment and terminal risk. Finally, $M_i(z)>0$, so $P$ and $\theta$ are in one to one correspondence at a fixed layer.
\end{proof}

\subsection{One sided enclosure certification}

For an arc $e$ in the layered graph, let $F_e(P)$ be its exact tissue transition and let $r_e(P)$ be its exact integrated running risk increment. Both maps are componentwise nondecreasing in $P$ when the arc action is held fixed.

\begin{theorem}[Monotone enclosure certificate]
\label{thm:cap-enclosure}
At layer $j$, let $Q_j^\uparrow$ be a tissue grid projection satisfying $P\le Q_j^\uparrow(P)$ componentwise on the reachable set. For $\Delta r>0$, let
\[
U_{\Delta r}(x):=\Delta r\left\lceil\frac{x}{\Delta r}\right\rceil.
\]
Fix any finite action sequence $e_0,\ldots,e_{m-1}$. Starting from $\widehat P_0=Q_0^\uparrow(P^0)$, define
\[
\widehat P_{j+1}
=Q_{j+1}^\uparrow\!\bigl(F_{e_j}(\widehat P_j)\bigr)
\]
and
\[
\widehat R
=\sum_{j=0}^{m-1}U_{\Delta r}\!\bigl(r_{e_j}(\widehat P_j)\bigr)
+U_{\Delta r}\!\bigl(\Psi(\widehat P_m)\bigr).
\]
Let $P_j$ and $R$ be the exact tissue states and exact total risk obtained by replaying the same action sequence from $P^0$, without tissue or risk rounding. Then
\[
P_j\le\widehat P_j\quad\text{for every }j,
\qquad R\le\widehat R.
\]
Consequently, $\widehat R\le\rho$ certifies the exact unrounded cap $R\le\rho$, with no safety slack.
\end{theorem}

\begin{proof}
The exact compartment flow is componentwise order preserving. If $P_j\le\widehat P_j$, then
\[
P_{j+1}=F_{e_j}(P_j)
\le F_{e_j}(\widehat P_j)
\le Q_{j+1}^\uparrow\!\bigl(F_{e_j}(\widehat P_j)\bigr)
=\widehat P_{j+1}.
\]
The state enclosure follows by induction. Along a fixed arc, the same order holds at every intermediate time, so monotonicity of the penalties gives $r_e(P_j)\le r_e(\widehat P_j)$. Terminal monotonicity gives $\Psi(P_m)\le\Psi(\widehat P_m)$, and $x\le U_{\Delta r}(x)$. Summing proves $R\le\widehat R$.
\end{proof}

This is an exact feasibility certificate for the returned action sequence, not exact optimality for the original continuous problem. Its safety statement does not deteriorate as the grid is coarsened. An oxygen dose coordinate is certified by the same one sided construction.

\begin{table}[H]
\centering
\caption{Information retained by candidate dynamic programming states.}
\label{tab:algorithmic-state}
\small
\begin{tabularx}{\linewidth}{@{}l >{\raggedright\arraybackslash}X
>{\raggedright\arraybackslash}X@{}}
\toprule
Stored physiological coordinate & Exact continuation? & Safe use \\
\midrule
Accumulated risk alone & No; tissue memory is lost. & Resource accounting only. \\
Clipped load $\sigma$ & No; \cref{thm:state-obstruction} gives two zero load labels with different successors. & Current penalty evaluation only. \\
Full $P$ or full $\theta$ & Yes, at a fixed depth layer. & Exact propagation and componentwise dominance. \\
Upper tissue cell plus upper risk bin & Conservative enclosure. & Exact certificate for the unrounded hard cap by \cref{thm:cap-enclosure}. \\
\bottomrule
\end{tabularx}
\end{table}

\subsection{Finite grid implementations and accounting}

\begin{corollary}[Certified finite product grid dynamic program]
\label{cor:product-grid-dp}
Fix $\calZ$, finite dwell menus $\mathcal T_j$, risk bins of width $\Delta r$, and a finite rectangular product grid with $N_P$ tissue cells per depth layer. Assume its top node bounds the reachable set componentwise. The fully discretised capped problem is solved exactly by dynamic programming on a layered DAG in time
\[
\tO\!\left(\left(\sum_j(G_jM_j+1)\right)
(R_{\max}/\Delta r)N_P\,c_{\mathrm{upd}}(n)\right),
\]
where $G_j=|\calG(z_j)|$, $M_j=|\mathcal T_j|$, $R_{\max}$ bounds attainable risk, and $c_{\mathrm{upd}}(n)$ is the tissue-update cost. With $K$ bins per compartment, $N_P=K^n$, so the method is exponential in $n$. It is polynomial only in the explicitly enumerated product grid; ``pseudo polynomial'' applies to the scalar risk range only after $n$ and that grid have been fixed.

With componentwise upward tissue projection and upward risk rounding, every plan accepted by the rounded cap is feasible for the exact unrounded dynamics by \cref{thm:cap-enclosure}. In the unbudgeted aggregate model, \cref{cor:gas-pruning} gives $G_j=1$. With the oxygen budget \eqref{eq:oxygen-budget}, gas branching returns and a dose grid with $N_C$ levels multiplies the bound by $N_C$.
\end{corollary}

\begin{proof}
A label at layer $j$ stores a risk bin, a tissue cell, and the least time known for that pair. Each dwell/gas choice updates all three quantities, the ascent arc moves the label to the next layer, and a final arc applies $\Psi$. The graph is acyclic, so Bellman optimality gives exactness for the fully discretised instance. Counting layer/risk/tissue states, transition fanout, and tissue update work gives the displayed bound. The cap statement follows from \cref{thm:cap-enclosure}.
\end{proof}

Without tissue quantisation, the recursion is a label setting algorithm.

\begin{theorem}[Lagrangian label setting for $(\mathbf P_\lambda)$]
\label{thm:labels}
For fixed $\lambda>0$, the scalarised problem on fixed $\calZ$ and finite dwell menus can be solved by label setting with componentwise full tissue dominance. Its runtime is
\[
\tO\!\left(\left(\sum_jG_jM_j\right)L_{\max}
\,c_{\mathrm{upd}}(n)\right),
\]
where $L_{\max}$ bounds the number of undominated exact full tissue labels per layer for the selected finite dwell menus. The result is exact for that finite menu instance; optional tissue state quantisation and continuous discretisation errors are treated in \cref{thm:apriori}.
\end{theorem}

\begin{proof}
Replace each stop by arcs indexed by $\tau\in\mathcal T_j$, with cost $\Delta T+\lambda\Delta R$ and the exact tissue update. Prune a label only when another at the same depth has no greater accumulated cost and no greater full tissue vector. \Cref{thm:state-obstruction} proves that this dominance is safe and that replacing the vector by clipped load would not be. Counting transitions over surviving labels gives the bound.
\end{proof}

For online or receding horizon implementations, the same finite action menus enter the one step recursion of \cref{thm:dpp}, which also supplies measurable near minimising selectors after discretisation. For an $\varepsilon$ net of dwell menus, one common exponential menu gives $M_j=O(\varepsilon^{-1}\log(T_{\max}/\tau_{\min}))$ over a bounded range. The next two results show that componentwise dominance need not compress the exact Pareto label set for increments generated by the tissue equations themselves, the obstruction is not an artefact of arbitrary arc data.

\subsection{Hardness, exact label growth, and approximation}

\begin{theorem}[One compartment hardness from \textsc{Subset Sum}]
\label{thm:compartment-hardness}
The finite menu decision problem
\[
 \text{does there exist a schedule with }T\le D\text{ and }R\le\rho?
\]
is NP-hard even with one compartment, one zero inert gas, zero running risk, a fixed monotone ascent, and two dwell choices at each stop. Thus the combinatorial obstruction can be generated by the compartment flow itself, arbitrary arc risk data are not required.
\end{theorem}

\begin{proof}
Reduce from \textsc{Subset Sum} \cite{GareyJohnson1979, KellererPferschyPisinger2004}.  Given positive integers $w_1,\ldots,w_m$ and target $B$, discard any $w_j>B$ and place stop $j$ at
\[
z_j=\frac12+\frac{m-j+1}{2(m+1)}\in(1/2,1),
\qquad z_{\mathrm{exit}}=1/4.
\]
At stop $j$ use the dwell menu
\[
             \mathcal T_j=\{0,w_j/B\}.
\]
Choose one pure oxygen gas, so $P_\infty\equiv0$, one compartment with $P^0=1$ and half-time $1$, hence $k=\log2$, and ceilings above $1$. The running oversaturation risk is then identically zero.  Set $z_{\mathrm{start}}=1$ and $\dot z_{\max}=1$, so every schedule has one unit of fixed transit time, and take the monotone terminal functional $\Psi(P)=P$.  The ppO$_2$/END windows and remaining physical data can be fixed once so that pure oxygen is feasible on this shallow interval. For a selected subset $A$, write $W_A=\sum_{j\in A}w_j$.  Since every hold and transit segment has zero inspired inert pressure,
\[
 T_A=1+\frac{W_A}{B},\qquad
 P_A(T_A)=\exp\!\left[-(\log2)\left(1+\frac{W_A}{B}\right)\right]
       =2^{-1-W_A/B}.
\]
With the fixed bounds $D=2$ and $\rho=1/4$,
\[
 T_A\le D\iff W_A\le B,
 \qquad
 R_A=\Psi(P_A(T_A))\le\rho\iff W_A\ge B.
\]
Both inequalities hold exactly when $W_A=B$.  The rational dwell values $w_j/B$ have polynomial binary encoding length, while all other data and the decision thresholds are fixed, so the reduction is polynomial.  This proves NP-hardness, because the source is \textsc{Subset Sum}, the construction does not claim strong NP-hardness of the scheduling problem.
\end{proof}

The discrete dwell menus are essential to this statement, the theorem is a hardness result for the compartment generated finite menu scheduling problem, not a claim that the unrestricted continuous dwell problem encodes the same subset choices.

\begin{proposition}[Exponential exact tissue label frontier]
\label{prop:exponential-labels}
There are bounded horizon, one compartment, zero inert finite menu instances with $m$ stops for which a common layer contains $2^m$ pairwise nondominated exact labels $(C,P)$.  With the dwell menus written explicitly in binary, the canonical instance length is $N=\Theta(m^2)$, so the exact label count is $2^{\Omega(\sqrt N)}$ and in particular is not polynomial in the input length.
\end{proposition}

\begin{proof}
Use the same zero inert construction, zero running risk, and menus
\[
                 \mathcal T_j=\{0,2^{-j}\},\qquad j=1,\ldots,m.
\]
The total optional dwell is below $1$, so the horizon is uniformly bounded. After the last decision, a subset $A\subseteq\{1,\ldots,m\}$ produces
\[
 s_A=\sum_{j\in A}2^{-j},\qquad
 C_A=C_0+s_A,\qquad P_A=\bar P_0e^{-ks_A},
\]
where $C_0$ and $\bar P_0>0$ contain the common transit contributions. The subset sums are all distinct: multiplication by $2^m$ gives the distinct binary integers $0,\ldots,2^m-1$.  If $s_A<s_{A'}$, then $C_A<C_{A'}$ but $P_A>P_{A'}$.  Hence neither label dominates the other under the safe cost/full/tissue order, and all $2^m$ labels survive. Encoding $2^{-j}$ requires $\Theta(j)$ bits, giving $N=\Theta(\sum_{j=1}^m j)=\Theta(m^2)$.

Attach the downstream terminal penalty $\Psi(P)=P$.  For any mesh value $s_A$, choose $\lambda_A=e^{ks_A}/(k\bar P_0)$.  The function
\[
s\longmapsto s+\lambda_A\bar P_0e^{-ks}
\]
is strictly convex and has its unique minimiser on the bounded dwell interval at $s=s_A$.  Thus every one of the $2^m$ labels is optimal for some scalarisation parameter.
\end{proof}

No polynomial bound exists for the number of exact labels surviving componentwise cost/full/tissue dominance, even for one tissue. Approximate state compression has a different answer, in fixed compartment dimension it is fully polynomial under explicit conditioning, as shown in
\cref{thm:fixed-n-fptas}.

\section{Frontier geometry and monotonicity}\label{sec:frontier}

\begin{proposition}[Basic shape]\label{prop:frontier}
Let
\[
\mathcal F_H^{\mathrm{rel}}
:=\{(T(\pi),R(\pi)):\ \pi\ \text{relaxed feasible},\ T(\pi)\le H\}
\]
for a fixed horizon bound $H<\infty$. In the aggregate model, $\mathcal F_H^{\mathrm{rel}}$ is compact, and every point on its efficient frontier is attained with pure finite switch envelope gas. The Pareto frontier is antitone in the usual sense: no efficient point can have both larger time and larger risk than another attainable point. No general convexity claim for the continuous compartment frontier follows from these facts. Finite menu restrictions and model extensions with mutually exclusive switching or resource choices can be nonconvex, as shown by \cref{ex:nonconvex-menu}, in that case scalarisations need not recover every efficient point.
\end{proposition}

\begin{proof}
The bounded horizon compactness argument of \cref{lem:existence} gives a convergent subsequence of depth paths and relaxed inputs. Uniform state convergence and continuity of the running and terminal risk give convergence, not merely lower semicontinuity, of the outcome pair $(T,R)$. Hence the relaxed outcome set is compact. The gas envelope theorem purifies every relaxed outcome to a pure outcome with the same time and weakly lower risk; on the efficient frontier purification must preserve risk, and the selector has finitely many switches. Antitonicity is immediate. The final statement is only about finite menu restrictions or augmented models and follows from \cref{ex:nonconvex-menu}.
\end{proof}

\subsection{Dual recovery of the time/risk frontier}

On the compact set $\mathcal F_H^{\mathrm{rel}}$, define
\[
C_H(\rho)
:=\min\{T:(T,R)\in\mathcal F_H^{\mathrm{rel}},\ R\le\rho\},
\qquad
J_H(\lambda)
:=\min_{(T,R)\in\mathcal F_H^{\mathrm{rel}}}(T+\lambda R),
\]
with $C_H(\rho)=+\infty$ when the cap is infeasible.

\begin{theorem}[Frontier duality and exact recovery criterion]
\label{thm:frontier-duality}
For every $\lambda\ge0$,
\begin{equation}\label{eq:scalar-cap-conjugacy}
J_H(\lambda)=\inf_{\rho\ge0}\{C_H(\rho)+\lambda\rho\}.
\end{equation}
Define
\[
D_H(\rho):=\sup_{\lambda\ge0}\{J_H(\lambda)-\lambda\rho\}.
\]
Then:
\begin{enumerate}[label=(\alph*),leftmargin=2em]
\item $D_H(\rho)\le C_H(\rho)$ for every $\rho$.
\item $D_H$ is the greatest lower semicontinuous, convex, nonincreasing function dominated by $C_H$. It is the closed convexified frontier.
\item If the dual supremum is attained at $\lambda^\star$, then $D_H(\rho)=C_H(\rho)$ if and only if a capped optimum $(T^\star,R^\star)$ also minimises $T+\lambda^\star R$ and satisfies
\[
R^\star\le\rho,\qquad
\lambda^\star(R^\star-\rho)=0.
\]
\item If $\lambda^\star>0$ is dual optimal and has scalar optimal outcomes $(T_-,R_-)$ and $(T_+,R_+)$ with $R_-\le\rho\le R_+$, then a lottery over the two complete schedules attains expected risk $\rho$ and expected time $D_H(\rho)$. Thus $C_H(\rho)-D_H(\rho)$ is exactly the deterministic schedule penalty at that cap.
\end{enumerate}
\end{theorem}

\begin{proof}
For any attainable $(T,R)$, capped optimality gives $C_H(R)\le T$, hence
\[
\inf_{\rho\ge0}\{C_H(\rho)+\lambda\rho\}\le T+\lambda R.
\]
Minimising over outcomes gives one direction of \eqref{eq:scalar-cap-conjugacy}. Conversely, if $(T_\rho,R_\rho)$ attains $C_H(\rho)$, then
\[
J_H(\lambda)\le T_\rho+\lambda R_\rho
\le C_H(\rho)+\lambda\rho.
\]
Taking the infimum over $\rho$ proves equality.

For each $\lambda$, the affine function $\rho\mapsto J_H(\lambda)-\lambda\rho$ lies below $C_H$. Their supremum is therefore a lower semicontinuous, convex, nonincreasing minorant. If $a-\lambda\rho$ is any affine minorant of $C_H$ with $\lambda\ge0$, then
\[
a\le\inf_\rho\{C_H(\rho)+\lambda\rho\}=J_H(\lambda).
\]
Thus every such affine minorant lies below $D_H$, proving maximality.

If a scalar optimum satisfies cap feasibility and complementary slackness, then scalar optimality implies, for every capped feasible $(T,R)$,
\[
T+\lambda^\star R\ge T^\star+\lambda^\star R^\star,
\]
and complementary slackness gives $T\ge T^\star$. It also gives $J_H(\lambda^\star)-\lambda^\star\rho=T^\star$, proving strong duality. Conversely, equality throughout the weak duality chain forces scalar optimality and complementary slackness.

For part (d), set
\[
\theta:=\frac{R_+-\rho}{R_+-R_-}.
\]
The two outcomes lie on the same supporting line, so $T_-+\lambda^\star R_-=T_++\lambda^\star R_+=J_H(\lambda^\star)$. Taking the convex combination with weights $\theta$ and $1-\theta$ gives expected risk $\rho$ and expected time $J_H(\lambda^\star)-\lambda^\star\rho=D_H(\rho)$.
\end{proof}

\begin{lemma}[Monotonicity with respect to inert fraction]\label{lem:inert}
Fix a depth $z$ and incoming tissue state $P$. If two feasible gases satisfy $F_I(g_1)\le F_I(g_2)$, then the hold trajectory under $g_1$ has componentwise no larger tissue pressure than the hold trajectory under $g_2$ at every dwell time. Consequently, its instantaneous oversaturation penalty is no larger at every dwell time.
\end{lemma}

\begin{proof}
For each compartment,
\[
P_i^{g}(t)=P_\infty(g,z)+\bigl(P_i(0)-P_\infty(g,z)\bigr)e^{-k_it}.
\]
Since $P_\infty(g,z)=F_I(g)(\Pa(z)-w)$ is nondecreasing in $F_I(g)$,
\[
P_i^{g_1}(t)-P_i^{g_2}(t)=\bigl(P_\infty(g_1,z)-P_\infty(g_2,z)\bigr)(1-e^{-k_it})\le0.
\]
The maps $P_i\mapsto S_i$ and $S_i\mapsto\phi_i(S_i)$ are nondecreasing, so the penalty ordering follows.
\end{proof}

\begin{proposition}[Monotonicity in $\eta$ and $w$]\label{prop:eta-w}
\begin{enumerate}[label=(\alph*),leftmargin=2em]
\item Increasing $\eta$ weakly increases $F_{\mathrm{nar}}(g)$ for every gas, hence weakly shrinks $\calG(z)$ through the END constraint. The optimal value $J_\lambda^\star$ is therefore nondecreasing in $\eta$ whenever all other primitives are fixed.
\item If the feasible gas sets $\calG(z)$ are held fixed, increasing $w$ decreases $P_\infty(g,z)$ for all $g,z$. By \cref{lem:inert}, this weakly reduces tissue pressures during holds and cannot increase the oversaturation penalty along a fixed profile. Under frozen feasibility sets, the optimal scalarised value is therefore nonincreasing in $w$.
\item If feasibility windows are recomputed when $w$ changes, the net effect is not monotone in general. The END equation gives
\[
z'=\frac{F_{\mathrm{nar}}(g)}{0.79}z+\frac{F_{\mathrm{nar}}(g)-0.79}{0.79\gamma}(\Po-w),
\]
so
\[
\frac{\partial z'}{\partial w}=-\frac{F_{\mathrm{nar}}(g)-0.79}{0.79\gamma}.
\]
Thus increasing $w$ lowers computed END for $F_{\mathrm{nar}}>0.79$ and raises it for $F_{\mathrm{nar}}<0.79$. For ppO$_2$, increasing $w$ lowers $\ppOtwo$, which relaxes an upper ppO$_2$ bound but tightens a lower ppO$_2$ bound. Hence changing $w$ can either expand or shrink feasible gas sets depending on mix and depth.
\end{enumerate}
\end{proposition}

\begin{lemma}[Transit risk bounds]\label{lem:transit}
With adjacent grid spacing $\Delta z$, the risk accrued during one max rate ascent edge between adjacent depths is $O(\Delta z)$. Over a fixed total ascent interval, however, the sum of such edge risks is generally $O(1)$ and should not be discarded unless separately justified. More precisely, if
\[
\bar r:=\sup\Bigl\{\sum_i\phi_i(S_i(z,P)):\ (z,P)\text{ reachable on the chosen compact set}\Bigr\}<\infty,
\]
then total max- ate transit risk is at most $\bar r Z_{\max}/\dot z_{\max}$. Equivalently, with $\bar\ell_{\mathrm{tr}}:=1+\lambda\bar r$, the corresponding scalarised transit cost is at most $\bar\ell_{\mathrm{tr}}Z_{\max}/\dot z_{\max}$.
\end{lemma}

\begin{proof}
The transit time over one adjacent edge is $\Delta z/\dot z_{\max}$ while $\sum_i\phi_i(S_i)$ is bounded by $\bar r$ on reachable compact state sets; multiplying gives $\bar r\Delta z/\dot z_{\max}=O(\Delta z)$ for one edge. Summing over the monotone ascent gives total transit time at most $Z_{\max}/\dot z_{\max}$ and hence total transit risk at most $\bar r Z_{\max}/\dot z_{\max}$. The scalarised bound follows by adding the unit time cost and multiplying the risk contribution by $\lambda$.
\end{proof}

\section{Numerical methods and a-priori error bounds}\label{sec:numerics}

Fix a depth grid $\calZ$ with spacing $\Delta z$ and include all ppO$_2$/END feasibility cell boundaries together with $z_{\mathrm{exit}}$. Use max rate arcs between adjacent depths. At each stop $j$ choose a finite dwell menu $\mathcal T_j\subset\R_+$, with $\mathcal T_j=\{0\}$ below $z_{\mathrm{exit}}$; define the dwell quantisation radius
\[
\delta=\max_j\sup_{\tau\in[0,T_{\max}]}\inf_{\tau'\in\mathcal T_j}|\tau-\tau'|.
\]
For the capped problem, discretise risk to bins of width $\Delta r$ and, when using \cref{cor:product-grid-dp}, discretise tissue state as well.

By \cref{thm:state-obstruction}, store either tissue pressures $P\in\R_+^n$ or full normalised tensions $\theta_i=P_i/M_i(z)$. The clipped vector $\sigma_i=(P_i-M_i(z))_+/M_i(z)$ is an auxiliary score for evaluating the penalty, not a state for propagation or dominance.

\begin{theorem}[A-priori discretisation and cap error]\label{thm:apriori}
Fix an upper bound $T_{\max}$ on optimal time, a uniform bound $m$ on the number of stops, and a bound $E$ on risk bearing arcs. Work on a reachable compact set. Let $\delta_P$ be the maximum $\ell_1$ error introduced by one tissue state rounding step. There are constants $L_{\mathrm{stop}},L_{\mathrm{dwell}},C_P>0$, depending only on model parameters over this compact set, such that:
\begin{enumerate}[leftmargin=2em,label=(\alph*)]
\item \emph{Depth snapping.} Snapping all stop depths to $\calZ$ within their original gas feasibility cells perturbs total risk, including $\Psi$, by at most $L_{\mathrm{stop}}\Delta z\,T_{\max}$. If the initial and terminal depths are preserved, total max rate transit time remains exactly $z_{\mathrm{start}}/\dot z_{\max}$.
\item \emph{Dwell quantisation.} Replacing each optimal dwell $\tau_j^\star$ by the nearest menu value in $\mathcal T_j$ changes $R$ by at most
\[
L_{\mathrm{dwell}}\sum_j|\tau_j-\tau_j^\star|\le L_{\mathrm{dwell}}m\delta,
\]
and changes $T$ by at most $\sum_j|\tau_j-\tau_j^\star|\le m\delta$.
\item \emph{Tissue quantisation.} Nearest cell rounding can understate risk. Its contribution to the a-priori risk error is at most $C_PE\delta_P$. Componentwise upward enclosure rounding is instead certified by
\cref{thm:cap-enclosure}.
\item \emph{Risk discretisation.} Downward per arc rounding gives only
\[
R<\rho+E\Delta r.
\]
Choosing $\Delta r=\varepsilon/E$ limits this contribution to
$\varepsilon$; upward rounding is covered by
\cref{thm:cap-enclosure}.
\end{enumerate}
Consequently, for the scalarised problem,
\[
J_\lambda(\widehat\pi)-J_\lambda(\pi^\star)
\le C_1\Delta z+C_2m\delta+\lambda C_PE\delta_P
\]
for constants $C_1,C_2>0$. Under a cap, a nonconservative
implementation has the bicriteria guarantee
\[
R(\widehat\pi)
\le\rho+
O(\Delta z+m\delta+E\delta_P+E\Delta r).
\]
Exact cap feasibility requires the certificate of \cref{thm:cap-enclosure}, sufficient cap slack, or a separate risk repair step.
\end{theorem}

\begin{proof}
The depth and dwell estimates follow from the finite feasibility-cell partition, local Lipschitzness of the running and terminal risk, and Lipschitz dependence of the linear ODE flow on stop depth and duration. A monotone path made entirely of max rate transits traverses total vertical distance $z_{\mathrm{start}}$, independent of intermediate stop locations.

For the tissue grid, the exact compartment transition is nonexpansive in its incoming pressure under the $\ell_1$ norm. An error of at most $\delta_P$ introduced on each of at most $E$ transitions therefore accumulates to at most $E\delta_P$. Lipschitzness of the risk functional on the reachable compact set gives the term $C_PE\delta_P$. If rounding is componentwise upward, order preservation instead gives an upper enclosure at every stage.

Finally, if $\widehat r_e^-$ and $\widehat r_e^+$ are downward and upward roundings of one arc risk, then
\[
\widehat r_e^-\le r_e<\widehat r_e^-+\Delta r,\qquad
r_e\le\widehat r_e^+<r_e+\Delta r.
\]
Summing over at most $E$ arcs proves the risk rounding claims. Combining the four perturbation bounds gives the displayed scalarised and capped estimates.
\end{proof}

\begin{theorem}[Fixed compartment FPTAS for finite menu scalarisation]
\label{thm:fixed-n-fptas}
Consider the finite menu scalarised problem on a layered graph with at most $E$ arcs per path and $A_{\mathrm{tot}}$ explicitly listed actions.  Let $N$ be its binary input length.  Assume:
\begin{enumerate}[label=(\roman*),leftmargin=2em]
\item $n$ is fixed and $[0,\bar P]^n$ is a known common invariant enclosure for the exact and upper rounded transitions, with $\bar P$ included as the grid's top node;
\item each tissue transition $F_e$ is componentwise order preserving and nonexpansive in $\ell_1$.  The arc risk maps and terminal risk $\Psi$ are componentwise nondecreasing on the enclosure, with known bounds $L_j:=\max_{e\text{ at layer }j}\operatorname{Lip}(r_e)$ and $L_\Psi:=\operatorname{Lip}(\Psi)$;
\item the compulsory max rate transit time $T_0=z_{\mathrm{start}}/\dot z_{\max}$ is positive;
\item rational upper bounds for $\bar P,L_j,L_\Psi,T_0$, and $\lambda$ are included in the encoding $N$, and validated $p$-bit enclosures of transitions and costs are computable in time polynomial in $N+p$.
\end{enumerate}
Define
\[
 \Gamma:=n\left(\sum_{j=0}^{E-1}(j+1)L_j+(E+1)L_\Psi\right),
 \qquad
 \kappa:=\frac{\bar P\lambda\Gamma}{T_0}.
\]
For every $\varepsilon\in(0,1)$, an upper tissue grid dynamic program returns an exact action sequence $\widehat\pi$ satisfying
\[
             J_\lambda(\widehat\pi)
             \le(1+\varepsilon)J_\lambda(\pi^\star)
\]
in time
\[
 O\!\left(
 A_{\mathrm{tot}}
 \left(1+\left\lceil\frac{4\kappa}{\varepsilon}\right\rceil\right)^n
 \operatorname{poly}\!\left(N,\log\frac1\varepsilon\right)
 \right).
\]
Consequently this is an FPTAS on every fixed $n$ family for which $\kappa\le\operatorname{poly}(N)$.  In particular, the condition holds for a fixed physiological model and bounded physical horizon with only the depth and dwell menus growing.  If $n$ is part of the input, the displayed scheme is exponential in $n$ and is not an FPTAS.
\end{theorem}

\begin{proof}
Use a rectangular tissue grid of coordinate spacing
\[
                 \delta=\frac{\varepsilon T_0}{4\lambda\Gamma}
\]
and project every successor component upward; if $\lambda\Gamma=0$, state rounding is unnecessary.  Compute each transition coordinate with a validated upper enclosure error at most $\delta$ before upward grid projection, and include the initial state as a grid node.  For a fixed action sequence, let $P_j$ and $\widehat P_j$ be its exact and rounded states.  Order preservation and nonexpansiveness give inductively
\[
 P_j\le\widehat P_j,
 \qquad
 \|\widehat P_j-P_j\|_1\le 2(j+1)n\delta.
\]
Therefore the rounded risk of any path is an upper bound on its exact risk, and the Lipschitz estimates give
\[
 0\le \widehat R(\pi)-R(\pi)\le2\Gamma\delta
 =\frac{\varepsilon T_0}{2\lambda}.
\]
Use validated upper enclosures for direct arc risk and terminal evaluations, with total risk evaluation error at most $\varepsilon T_0/(2\lambda)$ along a path.  Writing $\widehat J$ for the resulting rounded objective and letting the dynamic program minimise it,
\[
 \begin{aligned}
 J_\lambda(\widehat\pi)
 &\le \widehat J_\lambda(\widehat\pi)
 \le \widehat J_\lambda(\pi^\star)\\
 &\le J_\lambda(\pi^\star)+\varepsilon T_0
 \le (1+\varepsilon)J_\lambda(\pi^\star),
 \end{aligned}
\]
because every path contains the compulsory transit and hence $J_\lambda(\pi^\star)\ge T_0$.  There are at most $1+\lceil\bar P/\delta\rceil$ grid values per tissue coordinate, multiplying by the listed actions and polynomial precision evaluation cost proves the runtime bound.  Under the encoded bounds and $\kappa\le\operatorname{poly}(N)$, the required precision is $p=O(N+\log E+\log(1/\varepsilon))$.
\end{proof}

\begin{corollary}[Exact cap FPTAS under risk repair]
\label{cor:cap-fptas}
Under the hypotheses of \cref{thm:fixed-n-fptas}, suppose additionally that the capped instance is feasible and there are encoded constants $K_{\mathrm{rep}},\eta_0>0$ and a polynomial time repair operator returning a schedule in the same finite menu class such that every schedule with $R\le\rho+\eta$, $0\le\eta\le\eta_0$, can be made cap feasible at an added time of at most $K_{\mathrm{rep}}\eta$.  Define
\[
\kappa_{\mathrm{cap}}
:=\bigl(\bar P\Gamma+\rho(E+1)\bigr)
\max\!\left\{\frac1{\eta_0},\frac{K_{\mathrm{rep}}}{T_0}\right\}.
\]
If $n$ is fixed and $\kappa_{\mathrm{cap}}\le\operatorname{poly}(N)$, then for every $\varepsilon\in(0,1)$ the finite menu capped problem has an FPTAS with runtime
\[
O\!\left(
A_{\mathrm{tot}}
\left(1+\left\lceil\frac{6\kappa_{\mathrm{cap}}}{\varepsilon}\right\rceil\right)^n
\left(1+\left\lceil\frac{3\kappa_{\mathrm{cap}}}{\varepsilon}\right\rceil\right)
\operatorname{poly}\!\left(N,\log\frac1\varepsilon\right)
\right).
\]
\end{corollary}

\begin{proof}
Let
\[
 \eta_\varepsilon=\min\{\eta_0,\varepsilon T_0/K_{\mathrm{rep}}\}.
\]
Project tissues downward and round each of the at most $E+1$ risk increments downward, using validated lower transition enclosures of coordinate error at most $\delta$ and validated lower arc risk and terminal evaluations, clamped at zero if necessary.  Choose tissue and risk meshes so that
\[
 2\Gamma\delta+(E+1)\Delta r+\text{evaluation error}
 \le\eta_\varepsilon.
\]
The rounded risk never exceeds true risk, so an exact optimum is admitted by the rounded dynamic program.  Its selected path therefore has $T(\widehat\pi)\le T^\star$ and $R(\widehat\pi)\le\rho+\eta_\varepsilon$.  Repair gives
\[
 R(\operatorname{Rep}\widehat\pi)\le\rho,\qquad
 T(\operatorname{Rep}\widehat\pi)
 \le T^\star+K_{\mathrm{rep}}\eta_\varepsilon
 \le(1+\varepsilon)T^\star.
\]
Taking, for example, $\delta=\eta_\varepsilon/(6\Gamma)$, $\Delta r=\eta_\varepsilon/[3(E+1)]$, and direct evaluation error at most $\eta_\varepsilon/3$ gives the displayed runtime; when $\Gamma=0$, the tissue grid is unnecessary.  The repair assumption is a polynomial time, same class strengthening of the qualitative error bound in \cref{rem:cap-repair}.
\end{proof}

Without repair, the same construction gives the honest bicriteria guarantee $T(\widehat\pi)\le T^\star$ and $R(\widehat\pi)\le\rho+\eta$, conversely, the upper enclosure of \cref{thm:cap-enclosure} gives exact cap safety but cannot guarantee relative time at a boundary optimum.  A checkable sufficient repair condition is an admissibly lengthenable dwell whose downstream risk derivative is bounded above by $-\beta<0$ near the cap, in which case $K_{\mathrm{rep}}=1/\beta$.  This supplies a continuous dwell repair, or a finite menu repair only when the same action is explicitly available in that menu, it is not automatic for an arbitrary finite menu.

\section{Worked multi-compartment instance}\label{sec:worked-example}

We now work through a fully reproducible stylised instance designed to test the mathematics. Let
\[
\Pa(z)=1+0.1z\ \text{bar},\qquad
w=0.0627\ \text{bar},\qquad
\dot z_{\max}=9\ \text{m/min},\qquad
z_{\mathrm{exit}}=1\ \text{m}.
\]
The initial state is generated by a 25 minute square exposure on air at 30 m, starting from surface air equilibrium:
\[
P_i^0=q_{30}+(q_0-q_{30})e^{-25k_i},\qquad
q_0=0.79(1-w),\qquad q_{30}=0.79(4-w).
\]
This is a bounce exposure, not a saturation exposure. In particular,
\[
q_{30}=3.110467>P_i^0,\qquad i=1,2,3,
\]
so the saturation decompression endpoint theorem in \cref{sec:robust} does not apply to this instance. The profile contains genuine on gassing phases, on entry to EAN50 at 22.627 m, for example, the inspired inert pressure is 1.60 bar and the slow compartment remains near 1.21 bar. A robust treatment of this example, if rate uncertainty were imposed, would have to recompute the 25 minute bottom loading under the same candidate rate used during ascent rather than freezing the nominal vector $P^0$. The numerical study below instead treats the three half times $(5,20,80)$ min, and hence their rate constants, as fixed. No rate uncertainty intervals are imposed in the worked
instance. We use $\phi_i(s)=c_is^2$ and the parameters in \cref{tab:worked-compartments}.

\begin{table}[t]
\centering
\caption{Compartment parameters for the worked instance. Pressures are in bar and half times are in minutes.}
\label{tab:worked-compartments}
\begin{tabular}{@{}rrrrrr@{}}
\toprule
$i$ & half time & $a_i$ & $b_i$ & $c_i$ & $P_i^0$ \\
\midrule
1 & 5  & 0.95 & 0.60 & 1 & 3.03640450 \\
2 & 20 & 0.80 & 0.55 & 2 & 2.11400475 \\
3 & 80 & 0.65 & 0.50 & 4 & 1.20203596 \\
\bottomrule
\end{tabular}
\end{table}

The gas alphabet is air, EAN50, and oxygen. We impose $0.16\le\ppOtwo\le1.60$ bar, $\overline{\END}=30$ m, and $\eta=0$. \Cref{thm:gas-envelope} selects
\[
g_{\min}(z)=
\begin{cases}
\text{air}, & 22.627<z\le30,\\
\text{EAN50}, & 6.627<z\le22.627,\\
\text{oxygen}, & 0\le z\le6.627.
\end{cases}
\]
The switch depths are the exact upper ppO$_2$ boundaries. Allowed hold depths are $(18,15,12,9,6,3)$ m. We compare the literal $\Psi\equiv0$ objective with the fixed 30 minute surface-air functional from \cref{prop:surface-tail}, in both cases at $\lambda=100$.

\paragraph{Scale of the sufficient no stop certificate.}
For this instance,
\[
q_{\max}=\bar P=3.110467,\qquad
B_f=2q_{\max}\log 2\left(\frac15+\frac1{20}+\frac1{80}\right)
=1.131906,\qquad \Theta=\frac{30}{9}=3.333333.
\]
With the $\ell_1$ tissue norm used in \cref{thm:no-stop}, a valid reachable box constant is
\[
L_r=\max_i\frac{2c_i(\bar P-M_i(0))_+}{M_i(0)^2}
=11.859158,
\]
attained by compartment 3. For $\Psi\equiv0$, this gives $B=44.744839$ and $B^{-1}=0.022349$. For the 30 minute surface tail,
\[
L_\Psi=\max_i\int_0^{30}
\frac{2c_i(P_i^{\mathrm{surf}}(s;\bar P\boldsymbol1)-M_i(0))_+}{M_i(0)^2}
e^{-k_i s}\,ds
=269.949388,
\]
so $B=350.302172$ and $B^{-1}=0.002855$. At $\lambda=100$, the two values of $\lambda B$ are respectively $4.47\times10^3$ and $3.50\times10^4$. Thus \cref{thm:no-stop} is a qualitative small-price phase result with deliberately global constants, not a predictor of the operating regime used below. A seeded differential evolution search, bounded local refinement, and adaptive quadrature give the computed stationary solutions in \cref{tab:worked-solutions}. For $H_{\mathrm{surf}}=30$, the best consolidated solution found has
\[
R_{\mathrm{dive}}=0.016095206,\qquad
\Psi_{30}=0.004263871,\qquad
R=0.020359077.
\]
Its surface tissue state is
\[
P(T)=(0.558141,\ 1.414281,\ 1.133049),
\]
with surface normalised tension $(0.360091,\ 1.047616,\ 0.985260)$. For every positive dwell, centred differences with step $10^{-4}$ give
\[
-\frac{\partial R}{\partial\tau_j}
=0.0100000+O(5\times10^{-8})
=\frac1\lambda,
\]
as predicted by \cref{thm:dwell-switching}.

\begin{table}[t]
\centering
\caption{Numerically optimised continuous dwells at $\lambda=100$. Dwell vectors follow the order $(18,15,12,9,6,3)$ m.}
\label{tab:worked-solutions}
\small
\begin{tabularx}{\linewidth}{@{}c X r r r@{}}
\toprule
$H_{\mathrm{surf}}$ & optimal dwell vector (min) & $T$ & $R$ & $J_{100}$ \\
\midrule
0 &
$(1.631435,\,0.818492,\,4.107394,\,2.825062,\,2.033898,\,0)$
& 14.749614 & 0.033041003 & 18.053714 \\
30 &
$(1.632153,\,0.818505,\,4.108968,\,2.826935,\,6.195454,\,0)$
& 18.915349 & 0.020359077 & 20.951257 \\
\bottomrule
\end{tabularx}
\end{table}

Neither computed solution uses a 3 m dwell. Because oxygen has $F_I=0$ throughout $[0,6.627]$ m, \cref{thm:zero-inert} moves any 3m oxygen dwell to the deepest allowed stop, 6m, with no worse risk, irrespective of $\Psi$. What terminal closure changes is the \emph{amount} of oxygen dwell, the 6m hold increases from 2.033898 to 6.195454 minutes, while the 3 m dwell remains zero. For $\Psi\equiv0$, $\partial R/\partial\tau_{3\mathrm m}=0.00857839>0$ diagnoses endpoint truncation, but adding $\Psi$ cannot create a 3m dwell in this unbudgeted zero inert model. The oxygen budget extension in \eqref{eq:oxygen-budget} is precisely what can invalidate that consolidation.

\paragraph{A realised unsupported schedule.}
Restrict every dwell to $\{0,1,\ldots,8\}$ minutes and use $H_{\mathrm{surf}}=0$. Exhaustive evaluation of all $9^6=531{,}441$ schedules gives 22 Pareto points, of which 21 are supported. The unique unsupported efficient schedule is
\[
\tau^B=(2,1,3,3,1,0),\qquad
(T_B,R_B)=(13.333333,\ 0.051048219).
\]
Its adjacent supported outcomes have
\[
(T_-,R_-)=(12.333333,\ 0.062112108),\qquad
(T_+,R_+)=(14.333333,\ 0.038816495).
\]
At $T_B$, their chord has risk $\tfrac12(R_-+R_+)=0.050464302<R_B$, so no scalarisation selects $B$. Under the hard cap $R\le0.0555$, no nine minute dwell vector is feasible, while $B$ is a time optimal Pareto representative among the ten minute vectors. This is a nonconvexity witness generated by the compartment dynamics and feasibility windows, rather than by arbitrary arc increments.

\begin{figure}[t]
\centering
\begin{subfigure}[t]{0.32\linewidth}
\centering
\begin{tikzpicture}
\begin{axis}[
 width=0.96\linewidth,height=5.4cm,
 xlabel={time (min)},ylabel={depth (m)},
 xmin=0,xmax=19,ymin=0,ymax=30,y dir=reverse,
 grid=major,
 legend style={font=\tiny,at={(0.98,0.98)},anchor=north east},
 tick label style={font=\scriptsize},
 label style={font=\small}]
\addplot[very thick,blue] coordinates
 {(0,30) (0.819222,22.627)};
\addlegendentry{air}
\addplot[very thick,orange!85!black] coordinates
 {(0.819222,22.627) (1.333333,18) (2.965486,18)
  (3.298820,15) (4.117325,15) (4.450659,12) (8.559626,12)
  (8.892960,9) (11.719895,9) (11.983561,6.627)};
\addlegendentry{EAN50}
\addplot[very thick,teal!70!black] coordinates
 {(11.983561,6.627) (12.053228,6) (18.248682,6)
  (18.582016,3) (18.915349,0)};
\addlegendentry{oxygen}
\end{axis}
\end{tikzpicture}
\caption{Optimised profile.}
\end{subfigure}
\hfill
\begin{subfigure}[t]{0.32\linewidth}
\centering
\begin{tikzpicture}
\begin{axis}[
 width=0.96\linewidth,height=5.4cm,
 ybar,bar width=3.3pt,
 symbolic x coords={18,15,12,9,6,3},xtick=data,
 xlabel={stop depth (m)},ylabel={dwell (min)},
 ymin=0,ymax=7,
 grid=major,
 legend style={font=\tiny,at={(0.03,0.97)},anchor=north west},
 tick label style={font=\scriptsize},
 label style={font=\small}]
\addplot[fill=gray!45,draw=gray!70!black] coordinates
 {(18,1.631435) (15,0.818492) (12,4.107394)
  (9,2.825062) (6,2.033898) (3,0)};
\addlegendentry{$H_{\mathrm{surf}}=0$}
\addplot[fill=purple!60,draw=purple!85!black] coordinates
 {(18,1.632153) (15,0.818505) (12,4.108968)
  (9,2.826935) (6,6.195454) (3,0)};
\addlegendentry{$H_{\mathrm{surf}}=30$}
\end{axis}
\end{tikzpicture}
\caption{Terminal closure.}
\end{subfigure}
\hfill
\begin{subfigure}[t]{0.32\linewidth}
\centering
\begin{tikzpicture}
\begin{axis}[
 width=0.96\linewidth,height=5.4cm,
 xlabel={time $T$},ylabel={risk proxy $R$},
 xmin=3,xmax=25,ymode=log,ymin=0.007,ymax=0.45,
 grid=major,
 tick label style={font=\scriptsize},
 label style={font=\small}]
\addplot[very thick,blue,mark=*,mark size=1.2pt] coordinates
 {(3.333333,0.407285322) (4.333333,0.305673397)
  (5.333333,0.230128212) (6.333333,0.175693904)
  (7.333333,0.140948379) (8.333333,0.118283941)
  (9.333333,0.101721255) (10.333333,0.087499720)
  (11.333333,0.074163070) (12.333333,0.062112108)
  (14.333333,0.038816495) (15.333333,0.028434976)
  (16.333333,0.020956832) (17.333333,0.015578244)
  (18.333333,0.012199549) (19.333333,0.009739506)
  (20.333333,0.008650906) (21.333333,0.008032619)
  (22.333333,0.007890888) (23.333333,0.007859375)
  (24.333333,0.007859024)};
\addplot[dashed,black] coordinates
 {(12.333333,0.062112108) (14.333333,0.038816495)};
\addplot[only marks,mark=diamond*,mark size=3pt,red] coordinates
 {(13.333333,0.051048219)};
\end{axis}
\end{tikzpicture}
\caption{Realised frontier.}
\end{subfigure}
\caption{Worked three compartment instance. Panel (a) uses the exact minimum inert envelope of the unbudgeted aggregate model. In panel (b), the 30 minute surface functional lengthens the consolidated 6 m oxygen hold from 2.033898 to 6.195454 minutes; both 3 m dwells are zero by \cref{thm:zero-inert}. In panel (c), the blue curve is the supported hull and the red diamond is the unique unsupported efficient point.}
\label{fig:worked-instance}
\end{figure}

\section{Sensitivity}\label{sec:sensitivity}

Let $(T_\lambda,R_\lambda)$ be an optimal outcome of \cref{prob:def} and let $J_\lambda^\star$ denote the scalarised value.

\begin{proposition}[Envelope and frontier slope]\label{prop:envelope}
$J_\lambda^\star$ is locally absolutely continuous in $\lambda$ with
\[
\frac{d}{d\lambda}J_\lambda^\star=R_\lambda\quad\text{a.e.}
\]
Moreover, $T_\lambda$ is nondecreasing and $R_\lambda$ nonincreasing in $\lambda$, and along differentiable portions of a supported efficient frontier,
\[
\frac{dT_\lambda}{dR_\lambda}=-\lambda.
\]
\end{proposition}

\begin{proof}
For any feasible $\pi$, the map $\lambda\mapsto T(\pi)+\lambda R(\pi)$ is affine. Hence
\[
J_\lambda^\star=\inf_\pi\{T(\pi)+\lambda R(\pi)\}
\]
is the infimum of affine functions and is concave and locally Lipschitz on intervals where it is finite. By Danskin's theorem for minima \cite{Danskin1966}, at differentiability points its derivative equals the common risk value of active optimal profiles; more generally the superdifferential is the convex hull of active risk values. Thus $dJ_\lambda^\star/d\lambda=R_\lambda$ a.e. for any measurable optimal selection at differentiability points. Concavity implies this derivative is nonincreasing, hence $R_\lambda$ is weakly nonincreasing.

For $\lambda_1<\lambda_2$, optimality gives
\[
T_1+\lambda_1R_1\le T_2+\lambda_1R_2,
\qquad
T_2+\lambda_2R_2\le T_1+\lambda_2R_1.
\]
Adding yields $(\lambda_2-\lambda_1)(R_2-R_1)\le0$, so $R_2\le R_1$; substituting back gives $T_2\ge T_1$. At differentiability points, differentiating $J_\lambda^\star=T_\lambda+\lambda R_\lambda$ gives $dT_\lambda/d\lambda=-\lambda\,dR_\lambda/d\lambda$, hence $dT_\lambda/dR_\lambda=-\lambda$ when $dR_\lambda/d\lambda\ne0$.
\end{proof}

\section{Robust saturation decompression under interval uncertainty}
\label{sec:robust}

The result in this section is an endpoint theorem for saturation decompression, not a generic robustness theorem for bounce diving. Its defining ordering is that every compartment begins at or above the largest inspired inert input subsequently encountered and that this input is nonincreasing thereafter. Exact bottom equilibrium, $P_i^0=q^\pi(0)$ for every $i$, is the canonical saturation case, the weak inequality also permits a conservative supersaturated initial bound \cite{Imbert2024Saturation}. A finite square exposure begun from a lower equilibrium has the opposite ordering. If $q_s<q_b$ and the bottom time is $t_b<\infty$, then
\[
P_i^0=q_b+(q_s-q_b)e^{-k_i t_b}<q_b.
\]
Such a bounce profile starts with on gassing compartments. If rate uncertainty is imposed, it need not reduce to one endpoint scenario. The exact two level calculation in \cref{prop:bounce-rate} below shows that the worst rate can be strictly interior to its uncertainty interval. Fix an open loop segmented profile $\pi$. For an uncertainty realisation $\theta$, write $q_\theta^\pi(t)$ for its inspired inert input and $M_{i,\theta}^\pi(t)$ for its ceiling along the profile. Assume $q_\theta^\pi$ is nonincreasing and $P_{i,\theta}^0\ge q_\theta^\pi(0)$ for every $i$ and $\theta$. Also assume
\[
P_{i,\theta}^0\le\overline P_i^0,\qquad
k_{i,\theta}\in[\underline k_i,\overline k_i],\qquad
q_\theta^\pi(t)\le\overline q_\pi(t),
\]
and
\[
M_{i,\theta}^\pi(t)\ge\underline M_i^\pi(t)>0.
\]
Assume $\overline q_\pi$ is right continuous, of bounded variation, and nonincreasing, with $\overline P_i^0\ge\overline q_\pi(0)$. The uncertainty set is called rectangular when the simultaneous endpoint choice
\[
P_i^0=\overline P_i^0,\quad
k_i=\underline k_i,\quad
q^\pi=\overline q_\pi,\quad
M_i^\pi=\underline M_i^\pi
\]
is admissible.

\begin{theorem}[Saturation decompression endpoint principle]
\label{thm:robust-endpoint}
Under the assumptions above,
\[
P_{i,\theta}(t)\le P_i^{\mathrm{wc}}(t)
\quad\text{for every }i,t,
\]
where
\[
\dot P_i^{\mathrm{wc}}
=\underline k_i\bigl(\overline q_\pi(t)-P_i^{\mathrm{wc}}\bigr),
\qquad
P_i^{\mathrm{wc}}(0)=\overline P_i^0.
\]
Consequently $R_\theta(\pi)\le R_{\mathrm{wc}}(\pi)$, where running risk uses $P^{\mathrm{wc}}$ and $\underline M^\pi$, and the common terminal functional $\Psi$ is evaluated at $P^{\mathrm{wc}}(T)$.

If the uncertainty is rectangular, the bound is attained:
\[
\sup_\theta R_\theta(\pi)=R_{\mathrm{wc}}(\pi).
\]
Hence, over any saturation/off gassing policy class satisfying these hypotheses and having a common rectangular endpoint scenario, the robust scalarised and capped problems reduce exactly to that one deterministic endpoint instance.
\end{theorem}

\begin{proof}
For fixed $k$, the scalar comparison principle for $\dot P=k(q-P)$ shows that increasing the initial state or input increases $P(t)$ at every time. It remains to compare rate constants. For nonincreasing $q$, integration by parts gives
\[
P_k(t)-q(t)
=e^{-kt}\bigl(P^0-q(0)\bigr)
+\int_{(0,t]}e^{-k(t-s)}\,d(-q)(s).
\]
Both terms are nonnegative when $P^0\ge q(0)$ and nonincreasing in $k$. At exact saturation the first term vanishes. The rate ordering is then generated entirely by the subsequent downward variation of inspired inert pressure, slower compartments retain more of every earlier, larger input. Thus, after first replacing $(P^0,q_\theta^\pi)$ by $(\overline P^0,\overline q_\pi)$, replacing $k_{i,\theta}$ by $\underline k_i$ can only increase tissue pressure. This proves the state ordering.

The map $(P,M)\mapsto(P-M)_+/M$ is nondecreasing in $P$ and nonincreasing in $M$. The running penalties and $\Psi$ are nondecreasing, so the risk ordering follows. Rectangularity makes every endpoint choice simultaneously attainable, proving equality profile by profile and hence the two robust optimisation reductions.
\end{proof}

\begin{example}[Numerical saturation endpoint check]
\label{ex:saturation-endpoint-check}
We exercise \cref{thm:robust-endpoint} on a separate saturation instance using the gases and compartments of \cref{sec:worked-example}, not its bounce initial condition. Begin at $22.627$ m on EAN50 at exact tissue equilibrium, ascend at $9$ m/min, switch to oxygen at $6.627$ m, and continue directly to the surface. With
\[
T=\frac{22.627}{9}=2.514111,
\qquad t_s=\frac{16}{9}=1.777778,
\]
the upper inspired inert pressure profile is
\[
\overline q(t)=
\begin{cases}
1.6-0.45t, & 0\le t<t_s,\\
0, & t_s\le t\le T.
\end{cases}
\]
It decreases from $1.6$ to $0.8$ bar on EAN50 and then jumps to zero.

Model a common inspired inert pressure calibration factor by
\[
q_\beta(t)=\beta\overline q(t),\qquad
P_{i,\beta}^0=q_\beta(0)=1.6\beta,
\qquad \beta\in[0.98,1],
\]
and independent uncertain half times
\[
H_i\in[h_i/1.2,1.2h_i],
\qquad (h_1,h_2,h_3)=(5,20,80)\ \mathrm{min}.
\]
Thus every scenario starts at saturation. Moreover,
\[
k_i=\frac{\log 2}{H_i},\qquad
[\underline k_i,\overline k_i]
=\left[\frac{\log 2}{1.2h_i},\frac{1.2\log 2}{h_i}\right],
\]
with
\[
\underline k=(0.115524530,0.028881133,0.007220283).
\]
The simultaneous choice $\beta=1$ and $H_i=1.2h_i$ is admissible, so it is the rectangular endpoint in \cref{thm:robust-endpoint}.

Keep $a_i,b_i,c_i$ from \cref{tab:worked-compartments}, take $\Psi\equiv0$, and evaluate
\begin{equation}\label{eq:saturation-check-risk}
R_{\beta,H}=\sum_{i=1}^3\int_0^T
c_i\left(
\frac{\bigl(P_i^{\beta,H}(t)-M_i(t)\bigr)_+}{M_i(t)}
\right)^2dt,
\qquad
M_i(t)=a_i+b_i(3.2627-0.9t).
\end{equation}
The tissue trajectory is available in closed form:
\[
P_i^{\beta,H}(t)=
\begin{cases}
\displaystyle
\beta\left(1.6-0.45t+
\frac{0.45}{k_i}(1-e^{-k_it})\right),
&0\le t\le t_s,\\[7pt]
\displaystyle
P_i^{\beta,H}(t_s)e^{-k_i(t-t_s)},
&t_s<t\le T.
\end{cases}
\]
At the endpoint this gives
\[
P_i^{\mathrm{wc}}(T)=
\left[0.8+\frac{0.45}{\underline k_i}
\left(1-e^{-\underline k_i16/9}\right)\right]
e^{-\underline k_i6.627/9}
\]
and, numerically,
\[
P^{\mathrm{wc}}(T)
=(1.398983812,1.546567604,1.586430700).
\]
Adaptive quadrature of \eqref{eq:saturation-check-risk} gives
\[
(R_1^{\mathrm{wc}},R_2^{\mathrm{wc}},R_3^{\mathrm{wc}})
=(0,0.005720183,0.160800538),
\]
so the theorem gives the continuum equality
\[
\boxed{
\sup_{\substack{\beta\in[0.98,1]\\
H_i\in[h_i/1.2,1.2h_i]}}
R_{\beta,H}
=R_{\beta=1,\,k=\underline k}
=0.166520721.}
\]
\Cref{tab:saturation-endpoint-corners} independently audits all 16 corners; each displayed row represents both choices of $H_1$, because compartment 1 remains below its ceiling throughout.
\end{example}

\begin{table}[H]
\centering
\caption{Corner audit for \cref{ex:saturation-endpoint-check}. ``Slow'' means $H_i=1.2h_i$ and ``fast'' means $H_i=h_i/1.2$.}
\label{tab:saturation-endpoint-corners}
\begin{tabular}{@{}cccc@{}}
\toprule
$\beta$ & $H_2$ & $H_3$ & $R_{\beta,H}$ \\
\midrule
$0.98$ & slow & slow & $0.133079469$ \\
$0.98$ & slow & fast & $0.129035555$ \\
$0.98$ & fast & slow & $0.131970267$ \\
$0.98$ & fast & fast & $0.127926353$ \\
$1.00$ & slow & slow & $0.166520721$ \\
$1.00$ & slow & fast & $0.161881908$ \\
$1.00$ & fast & slow & $0.164947100$ \\
$1.00$ & fast & fast & $0.160308287$ \\
\bottomrule
\end{tabular}
\end{table}

For the canonical surface tail functional in \cref{prop:surface-tail}, the ``common $\Psi$'' hypothesis means that surface gas input, surface ceilings, and tissue rate constants are held fixed across scenarios. If those tail parameters are uncertain as well, the endpoint comparison must be extended through the surface phase or the tail scenarios retained explicitly.

\begin{remark}[Saturation scope and bounce dives]
\label{rem:robust-bounce}
The monotonicity in $k$ reverses on a pure on gassing segment. If $q>P^0$ is constant, then
\[
P_k(t)=q+(P^0-q)e^{-kt}
\]
is increasing in $k$. Mixed on gassing and off gassing profiles can therefore require a genuine multi-scenario robust optimisation, no single endpoint choice of $k_i$ is worst throughout the trajectory. The worked bounce exposure in \cref{sec:worked-example} is explicitly outside the theorem. There
\[
q_{30}=0.79(4-0.0627)=3.110467,\qquad
P^0=(3.036405,\,2.114005,\,1.202036),
\]
so $P_i^0<q_{30}$ for every compartment. At the 22.627 m switch to EAN50, its inspired inert pressure is still 1.60 bar and the slow compartment is still on gassing; it crosses to off gassing later, near 15 m. Thus \cref{thm:robust-endpoint} is a strong exact result for saturation decompression, including the equilibrium case $P_i^0=q(0)$, and is not used to justify robustness of the worked bounce instance.
\end{remark}

\begin{proposition}[Rate dependence after a two-level bounce]
\label{prop:bounce-rate}
Fix $q_b>q_s\ge0$, $t_b,t_a>0$, and an initial tissue pressure $p_0\ge0$. For each $k\ge0$, let the compartment breathe the bottom input $q_b$ for time $t_b$, followed by the lower input $q_s$ for time $t_a$. Write $P_T(k)$ for its pressure at the end and define
\[
p_{\mathrm{crit}}:=\frac{t_bq_b+t_aq_s}{t_a+t_b}.
\]
Then
\begin{equation}\label{eq:bounce-terminal}
P_T(k)
=q_s+(q_b-q_s)e^{-kt_a}
+(p_0-q_b)e^{-k(t_a+t_b)}.
\end{equation}

If $p_0<p_{\mathrm{crit}}$, the terminal pressure is strictly increasing on $[0,k^\star]$ and strictly decreasing on $[k^\star,\infty)$, where
\begin{equation}\label{eq:bounce-kstar}
k^\star
=\frac1{t_b}
\log\!\left(
\frac{(t_a+t_b)(q_b-p_0)}
     {t_a(q_b-q_s)}
\right)>0.
\end{equation}
Its maximum is
\begin{equation}\label{eq:bounce-maximum}
P_T(k^\star)
=q_s+\frac{t_b}{t_a+t_b}(q_b-q_s)
\left(
\frac{t_a(q_b-q_s)}
     {(t_a+t_b)(q_b-p_0)}
\right)^{t_a/t_b}.
\end{equation}
If $p_0\ge p_{\mathrm{crit}}$, $P_T$ is strictly decreasing on $(0,\infty)$. Consequently, over $[\underline k,\overline k]$ the worst terminal pressure is attained at
\[
k_{\mathrm{wc}}=
\begin{cases}
\min\{\overline k,\max\{\underline k,k^\star\}\},
&p_0<p_{\mathrm{crit}},\\
\underline k,&p_0\ge p_{\mathrm{crit}}.
\end{cases}
\]
When $p_0<p_{\mathrm{crit}}$, checking only the two rate endpoints fails exactly when
\[
\underline k<k^\star<\overline k.
\]
If $k^\star$ coincides with an interval endpoint, the maximum remains an endpoint maximum. Any $k$ independent nondecreasing terminal penalty is maximised at the same rate, uniquely when it is strictly increasing on the attained pressure range.
\end{proposition}

\begin{proof}
Solving the two affine compartment equations successively gives \eqref{eq:bounce-terminal}. Differentiation yields
\[
P_T'(k)
=e^{-kt_a}\left[
(t_a+t_b)(q_b-p_0)e^{-kt_b}
-t_a(q_b-q_s)
\right].
\]
The bracket at $k=0$ equals
\[
t_b(q_b-p_0)-t_a(p_0-q_s),
\]
which is positive exactly when $p_0<p_{\mathrm{crit}}$. In that case the bracket is strictly decreasing and has the unique zero \eqref{eq:bounce-kstar}, proving strict unimodality. At the zero,
\[
e^{-k^\star t_b}
=\frac{t_a(q_b-q_s)}
       {(t_a+t_b)(q_b-p_0)}.
\]
Substitution in \eqref{eq:bounce-terminal} gives \eqref{eq:bounce-maximum}.

If $p_0\ge p_{\mathrm{crit}}$ but $p_0<q_b$, the bracket is nonpositive at zero and strictly decreases thereafter. If $p_0\ge q_b$, both terms in $P_T'(k)$ are negative. The interval and terminal penalty conclusions follow from these sign patterns.
\end{proof}

\begin{remark}[Bounce versus saturation]
\label{rem:bounce-versus-saturation}
If the compartment begins in shallow equilibrium, $p_0=q_s$, then
\[
P_T(k)=q_s+(q_b-q_s)e^{-kt_a}(1-e^{-kt_b}),
\qquad
k^\star=\frac1{t_b}\log\!\left(1+\frac{t_b}{t_a}\right).
\]
In half time coordinates $h=\log 2/k$, the worst compartment has
\[
h^\star=\frac{t_b\log 2}{\log(1+t_b/t_a)}.
\]
By contrast, at bottom saturation $p_0=q_b$,
\[
P_T(k)=q_s+(q_b-q_s)e^{-kt_a},
\]
which is strictly decreasing in $k$. Lower endpoint rate uncertainty is therefore correct in the saturation regime, while a genuine bounce can have an interior worst rate.

The same uncertain $k$ must be used to generate the loading accumulated during the bottom exposure and to propagate the decompression phase. Freezing a nominal $P^0$ while varying only the ascent rate removes this coupling and can miss the interior maximum. The proposition concerns terminal loading under a two level input. It does not identify the worst rate for the full running plus terminal risk of the ramped, multigas profile in \cref{sec:worked-example}. Nor is the proposition a robustness result for the worked instance as stated, its three rates, corresponding to $(h_1,h_2,h_3)=(5,20,80)$ min, are fixed. To quantify the enlargement needed before $h^\star=8.097263$ min matters, consider the log symmetric multiplicative bands
\[
\mathcal H_i(\rho)=[h_i/\rho,\rho h_i],
\qquad
\mathcal K_i(\rho)=[k_i/\rho,\rho k_i],
\qquad \rho\ge1,
\]
which are equivalent under $k_i=\log 2/h_i$. The critical half time belongs to the $i$th band exactly when
\[
\rho\ge\rho_i^\star
:=\max\left\{\frac{h^\star}{h_i},\frac{h_i}{h^\star}\right\},
\qquad
(\rho_1^\star,\rho_2^\star,\rho_3^\star)
=(1.619453,\,2.469970,\,9.879881).
\]
At equality $k^\star$ is still a band endpoint, it becomes a genuine interior uncertainty maximiser only for $\rho>\rho_i^\star$. Equivalently, $k^\star=0.085603$ min$^{-1}$ is to be compared with the fixed nominal rates
\[
(k_1,k_2,k_3)=(0.138629,\,0.034657,\,0.008664)
\ \text{min}^{-1}.
\]
Reaching $k^\star$ requires the 5 minute rate band to extend $38.25\%$ below its nominal rate, or the 20 and 80 minute bands to extend $147.00\%$ and $887.99\%$ above theirs. Thus the smallest common band that can expose the interior phenomenon must exceed the factor $1.619453$, and the fixed rate worked calculation has no such effect.
\end{remark}

\begin{corollary}[Finite critical rate reduction for segmented terminal loading]
\label{cor:critical-rates}
Let $0=t_0<t_1<\cdots<t_m=T$, and let the input be $q_j$ on $[t_{j-1},t_j)$. For one compartment with initial pressure $p_0$,
\begin{equation}\label{eq:segmented-rate}
P_T(k)
=q_m+\sum_{j=1}^{m-1}(q_j-q_{j+1})e^{-k(T-t_j)}
+(p_0-q_1)e^{-kT}.
\end{equation}
Unless $P_T$ is constant in $k$, its derivative has at most $m-1$ real zeros. Hence its maximum on a compact rate interval is found by checking at most $m+1$ candidates, the two endpoints and the interior stationary rates.

For $n$ compartments with rectangular, independent rate intervals, a componentwise nondecreasing terminal functional is maximised by choosing for each compartment one of its finite critical candidates that maximises its own terminal pressure. This is a conditional finite reduction once the compartmentwise uncertainty intervals have been specified. Distinct fixed compartment rates are not endpoints or candidates of a common uncertainty interval.
\end{corollary}

\begin{proof}
Variation of constants on each segment gives
\[
P_T(k)
=e^{-kT}p_0
+\sum_{j=1}^m q_j
\left(e^{-k(T-t_j)}-e^{-k(T-t_{j-1})}\right),
\]
which telescopes to \eqref{eq:segmented-rate}. Its derivative is a linear combination of at most $m$ exponentials with distinct real exponents. Such exponentials form a Chebyshev system, after factoring the exponential with smallest exponent, Rolle's theorem and induction show that a nonzero combination of $m$ terms has at most $m-1$ real zeros. The compact interval claim follows. Rectangularity permits the componentwise maximising rate choices to be combined in one terminal scenario.
\end{proof}

\begin{figure}[H]
\centering
\begin{tikzpicture}
\begin{axis}[
 width=0.76\linewidth,height=6.1cm,
 xmode=log,
 xlabel={compartment half time $h$ (min)},
 ylabel={terminal tissue pressure $P_T$ (bar)},
 xmin=1,xmax=320,ymin=0.7,ymax=3.2,
 xtick={1,2,5,10,20,50,100,200},
 grid=major,
 log ticks with fixed point,
 legend style={font=\small,at={(0.02,0.98)},anchor=north west},
 tick label style={font=\small},
 label style={font=\small}]
\addplot[very thick,blue,domain=1:320,samples=220]
 {0.740467+2.37*(1-2^(-25/x))*2^(-(10/3)/x)};
\addlegendentry{bounce: $p_0=q_s$}
\addplot[very thick,orange!85!black,domain=1:320,samples=220]
 {0.740467+2.37*2^(-(10/3)/x)};
\addlegendentry{saturation: $p_0=q_b$}
\addplot[only marks,mark=*,mark size=2.1pt,blue] coordinates
 {(5,2.186817) (20,1.964150) (80,1.188896)};
\addlegendentry{fixed Section~\ref{sec:worked-example} half times}
\addplot[dashed,red] coordinates {(8.097263,0.7) (8.097263,2.312524)};
\addplot[only marks,mark=diamond*,mark size=3pt,red]
 coordinates {(8.097263,2.312524)};
\node[anchor=south,font=\scriptsize,red]
 at (axis cs:8.097263,0.83) {critical $h^\star=8.097$};
\end{axis}
\end{tikzpicture}
\caption{Rate sensitivity after a finite bounce exposure. The two level comparison uses the surface and bottom air inputs of \cref{sec:worked-example}, a 25 minute bottom exposure, and a $30/9$ minute lower input phase. The bounce curve has a strict interior maximum, whereas the saturated curve is worst at the slow rate endpoint. The blue markers are the worked instance's fixed half times, not an uncertainty band. The dashed maximiser becomes an interior robust candidate only when a band strictly contains it, with the required widths quantified in \cref{rem:bounce-versus-saturation}. This is a terminal loading illustration of \cref{prop:bounce-rate}, not a robustness recomputation of the worked profile's ramped multigas integrated risk.}
\label{fig:bounce-rate}
\end{figure}

\section{Distinct inert species}\label{sec:multi-species}

Let $S=\{\Ntwo,\He\}$. For each compartment $i$ and species $s\in S$,
\[
\dot P_{i,s}(t)=k_{i,s}\bigl(F_s(g(t))(\Pa(z(t))-w)-P_{i,s}(t)\bigr),\qquad k_{i,s}>0,
\]
and define total inert pressure $P_i=\sum_{s\in S}P_{i,s}$. Ceilings $M_i(z)$ and penalties $\phi_i(S_i)$ are unchanged, with $S_i=(P_i-M_i(z))_+/M_i(z)$.

\begin{proposition}[Structure preserved]\label{prop:multi-structure}
Under the multi-species dynamics above, the compactness and relaxed existence part of \cref{lem:existence}, together with \cref{lem:grid,lem:transit,prop:polish}, remains valid after enlarging the state. The staged density conclusion of \cref{lem:bangbang} also remains valid. The exact gas envelope reduction does not extend in general, two gases can trade lower nitrogen input against lower helium input, and species specific rates prevent a scalar pointwise ordering. Pure gas controls remain dense in the relaxed class by chattering, but exact pure attainment and the hard cap purification conclusion require additional structure or positive cap slack.
\end{proposition}

Per-transition update cost becomes $c_{\mathrm{upd}}(n,|S|)=\Theta(n|S|)$; the constants in the a-priori error bounds change only through the enlarged state dimension and species specific rates.

\section{Online value functions and dynamic programming}\label{sec:online-dpp}

This section supplies the value function justification behind the finite recursions; it is complementary to prior receding horizon decompression implementations \cite{Feng2012RecedingHorizon,DiMuro2023Adaptive}.

Fix $\lambda>0$. Throughout this section we work with the bang-bang restriction on the vertical rate $u\in\{0,-\dot z_{\max}\}$ and monotone $z(\cdot)$, consistent with \cref{ass:kin} and justified for staged $\varepsilon$-optimal profiles by \cref{lem:bangbang}. Every infimum below inherits the first surface hit and terminal-transit rules of \cref{ass:kin}.

\subsection{State value and basic bounds}

Define the running cost
\[
\ell(z,P,g):=1+\lambda\sum_{i=1}^n\phi_i\!\left(\frac{(P_i-M_i(z))_+}{M_i(z)}\right),
\]
and the online value from $x=(z,P)\in[0,Z_{\max}]\times\R_+^n$ by
\begin{equation}\label{eq:V-def}
\begin{aligned}
V_\lambda(z,P):=\inf\Bigg\{&
\int_0^T\ell(z(t),P(t),g(t))\,dt+\lambda\Psi(P(T)):\\
&\dot z=u(t)\in\{0,-\dot z_{\max}\}\ \text{a.e.},\quad
\dot P_i=k_i(P_\infty(g,z)-P_i),\\
&g(t)\in\mathcal G(z(t))\ \text{a.e.},\quad
z(0)=z,\ P(0)=P,\ T=\inf\{t\ge0:z(t)=0\}\Bigg\}.
\end{aligned}
\end{equation}
Feasibility from any $x$ holds by following $u=-\dot z_{\max}$ and applying \cref{lem:measurable-selection} for $g$. The regularity estimates below are local, so fix a compact pressure bound $P_{\mathrm{bd}}<\infty$ and work on
\[
\mathcal K:=[0,Z_{\max}]\times[0,P_{\mathrm{bd}}]^n.
\]
Let
\begin{equation}\label{eq:defs}
\begin{aligned}
F_I^{\max}&:=\max_{g\in\mathcal G}F_I(g),
&k_{\min}&:=\min_i k_i,
&k_{\max}&:=\max_i k_i,\\
M_{\min}&:=\min_i\inf_{z\in[0,Z_{\max}]}M_i(z)
=\min_i(a_i+b_i\Po)>0.
\end{aligned}
\end{equation}
On trajectories starting in $\mathcal K$, pressures remain bounded by
\[
\bar P:=\max\{P_{\mathrm{bd}},F_I^{\max}(\Po+\gamma Z_{\max})\}.
\]
Consequently $S_i\in[0,\bar S]$ with $\bar S:=\max\{0,\bar P/M_{\min}-1\}$, and each $\phi_i$ is Lipschitz on $[0,\bar S]$ with constant $L_{\phi,i}$. Set
\begin{equation}\label{eq:ell-bar}
\bar\ell:=1+\lambda\sum_{i=1}^n\phi_i(\bar S),
\qquad
\bar\Psi:=\max_{P\in[0,\bar P]^n}\Psi(P),
\qquad
T_{\mathrm{nd}}(z):=\frac{z}{\dot z_{\max}},
\qquad
T_{\mathrm{hor}}:=\bar\ell\frac{Z_{\max}}{\dot z_{\max}}
+\lambda\bar\Psi.
\end{equation}
Then
$V_\lambda(z,P)\le\bar\ell T_{\mathrm{nd}}(z)+\lambda\bar\Psi
\le T_{\mathrm{hor}}$
for the no hold ascent comparison policy, and every optimal horizon $T^\star$ satisfies $T^\star\le V_\lambda(z,P)\le T_{\mathrm{hor}}$ since $\ell\ge1$.

\subsection{Small horizon DPP and measurable selectors}

Let $h_{\max}(z):=z/\dot z_{\max}$ and consider any $h\in(0,h_{\max}(z)]$ when $z>0$. At the boundary, $V_\lambda(0,P)=\lambda\Psi(P)$. Concatenation of feasible arcs is permitted \cite{yong2022stochastic} by \cref{ass:kin,ass:comp}.

\begin{theorem}[Dynamic programming principle]\label{thm:dpp}
For every state $x=(z,P)$ with $z>0$ and every $h\in(0,h_{\max}(z)]$,
\begin{equation}\label{eq:dpp}
V_\lambda(z,P)=\inf_{\substack{u(\cdot)\in\{0,-\dot z_{\max}\}\ g(\cdot)\in\calG(z(\cdot))}}
\left\{\int_0^h\ell(z(t),P(t),g(t))\,dt+V_\lambda\bigl(z(h),P(h)\bigr)\right\},
\end{equation}
where $(z(\cdot),P(\cdot))$ solves the dynamics from $(z,P)$ under $(u,g)$ on $[0,h]$. On the finite state and action discretisation used in \cref{sec:algorithms}, an $\varepsilon$-minimising selector exists. More generally, a Borel selector follows when the one step costs and transitions and the value function are Borel.
\end{theorem}

\begin{proof}
The upper bound follows by concatenating any admissible $(u,g)$ on $[0,h]$ with an $\varepsilon$-optimal tail from $(z(h),P(h))$ and then letting $\varepsilon\downarrow0$. The lower bound follows by truncating an $\varepsilon$-optimal complete control at $h$ and using the definition of the tail value. On a finite discretisation the selector is obtained by tie breaking among finitely many actions; the stated Borel extension is the standard finite valued measurable selection argument.
\end{proof}

Fix a uniform nonterminal step $h>0$ on a finite state/action discretisation and a one step tolerance $\varepsilon_s\in[0,h)$. Repeatedly apply an $\varepsilon_s$-minimising selector. Since every nonterminal one step cost is at least $h$,
\[
V_\lambda(x^+)\le V_\lambda(x)-(h-\varepsilon_s).
\]
Thus the feedback cannot hold forever and reaches the terminal state after at most
\[
N\le
\left\lceil\frac{V_\lambda(z_0,P_0)}{h-\varepsilon_s}\right\rceil
\le
\left\lceil\frac{T_{\mathrm{hor}}}{h-\varepsilon_s}\right\rceil
\]
steps. Its cost $J$ satisfies
\[
J\le V_\lambda(z_0,P_0)+N\varepsilon_s,
\]
by summing the one step inequalities and telescoping the value terms. Regularity estimates for the value function, including flow sensitivity, an explicit integrand Lipschitz constant, and the risk capped DPP, are deferred
to \cref{app:value-regularity}.

\section{Conclusion}

The unbudgeted aggregate model is more rigid than its original finite menu description suggested. Minimum inert feasible gas dominates every pure and relaxed gas policy along a fixed depth path, so relaxed existence purifies exactly and the gas branching factor becomes one (\cref{thm:gas-envelope,cor:gas-pruning}). This does not say that practical gas planning is trivial, it says that gas planning is degenerate in an aggregate model that places no cumulative price on oxygen. The dose state in \eqref{eq:oxygen-budget} restores the central tissue loading/oxygen exposure trade off. On a fixed depth path its value is convex in the available dose, its multiplier is an exact oxygen shadow price, and \cref{cor:oxygen-switching} gives the marginal inequality deciding when lower inert loading is worth the additional exposure. Likewise, zero inert consolidation at the deepest feasible stop is a diagnostic of the omitted depth sensitive resource (\cref{thm:zero-inert,rem:oxygen-diagnostic}). Maximal ascents and holds form a cost dense normal form, and the explicit sufficient bound in \cref{thm:no-stop} identifies a genuine no stop phase at small risk prices. Its evaluation in the worked instance places $\lambda=100$ thousands of times above the certified phase, making clear that the global constant is qualitative rather than predictive there. Monotone ascent now has a positive structural theorem rather than a disclaimer. The continuous rearrangement in \cref{thm:endpoint-monotone} preserves time, depth occupation, the rate cap, terminal transit, and state independent exposure while lowering every terminal tissue pressure, together with the gas envelope it derives a monotone optimum for the aggregate terminal model (\cref{cor:aggregate-endpoint-monotone}).  The safe off gassing exchange in \cref{prop:safe-offgas-exchange} extends the order to a nontrivial fixed-block running risk regime, with transit effects explicitly outside that comparison. Outside the rearrangement theorem's hypotheses, \cref{prop:monotonicity-impossible} permits a deeper than start excursion and constructs a one compartment integrated-stress value gap in which re-descent strictly beats every monotone profile.  Thus the unrestricted bidirectional class cannot generally be reduced to monotone ascent; the full model retains the operational constraint.

The main local result is not the generic Clarke condition. It is the exact dwell switching identity in \cref{thm:dwell-switching}, an added minute is valuable only when its reduction in suffix risk pays both its instantaneous risk and the time price. The identity yields a one pass gradient algorithm, an exact stop screening rule, and dominance of purely on gassing holds. The closed form one compartment problem and the worked multi-compartment instance show these mechanisms directly. In the latter, terminal closure lengthens the consolidated 6m oxygen hold from 2.03 to 6.20 minutes; it does not restore a 3m hold, which remains dominated in the unbudgeted zero inert model.

The algorithmic theory has matching exact and approximate boundaries. \cref{thm:state-obstruction} proves that clipped oversaturation is not a sufficient Markov state, even when two labels both have zero clipped load, whereas full tissue pressure supports safe dominance. The one sided construction in \cref{thm:cap-enclosure} then preserves a hard cap exactly; nearest state or downward rounded schemes incur the complete error $O(\Delta z+m\delta+E\delta_P+E\Delta r)$ (\cref{cor:product-grid-dp,thm:apriori}).  Exact Pareto label enumeration can remain superpolynomial, \cref{thm:compartment-hardness} proves NP-hardness using a single actual compartment flow, and \cref{prop:exponential-labels} gives $2^m$ exact nondominated labels with one tissue and bounded horizon. Approximation is nevertheless tractable in the fixed physiology regime, \cref{thm:fixed-n-fptas} gives a scalarised FPTAS under explicit polynomial conditioning, and \cref{cor:cap-fptas} gives the exact cap counterpart under risk repair.  The $K^n$ product grid keeps the variable dimension exponential dependence explicit.  The frontier duality theorem separately identifies which capped points scalarisation recovers (\cref{thm:frontier-duality}).

The remaining limitations are substantive. The aggregate treatment of nitrogen and helium creates the strong gas dominance and zero inert consolidation effects; species specific kinetics restore genuine gas trade offs (\cref{prop:multi-structure}). The base computations do not enforce the oxygen exposure extension, bubble dynamics, CO$_2$, workload, gas supply, or a calibrated injury probability. The uncertainty endpoint theorem is an exact saturation decompression result (and conservative supersaturated extension), including the physically important equality $P_i^0=q(0)$. In \cref{ex:saturation-endpoint-check}, the upper input/lower rate endpoint gives $\sup_\theta R_\theta=0.166520721$, matching the independent corner audit. It does not cover the mixed on/off gassing bounce exposure in \cref{sec:worked-example}, for which multiple scenarios are required (\cref{thm:robust-endpoint}). The sharp contrast is now explicit, even a two level square bounce can have a unique interior worst rate (\cref{prop:bounce-rate}), while an arbitrary segmented terminal input has a finite critical rate set (\cref{cor:critical-rates}). These results concern terminal loading and do not turn the full ramped integrated risk problem back into an endpoint calculation. The worked half times $(5,20,80)$ min are fixed, their rate plot is not a robust recomputation. The interior maximiser enters the log symmetric band only beyond width $1.619453$ (\cref{rem:bounce-versus-saturation}).

\bibliographystyle{alpha}
\bibliography{biblio}

@article{Boycott1908,
  author  = {Boycott, A. E. and Damant, G. C. C. and Haldane, J. S.},
  title   = {The Prevention of Compressed-Air Illness},
  journal = {The Journal of Hygiene},
  year    = {1908},
  volume  = {8},
  number  = {3},
  pages   = {342--443},
  doi     = {10.1017/S0022172400003399}
}

@techreport{Workman1965,
  author      = {Workman, Robert D.},
  title       = {Calculation of Decompression Schedules for Nitrogen--Oxygen and Helium--Oxygen Dives},
  institution = {U.S. Navy Experimental Diving Unit},
  number      = {NEDU Research Report 6-65},
  address     = {Washington, DC},
  month       = {May},
  year        = {1965}
}

@book{Buehlmann1984,
  author    = {B{\"u}hlmann, Albert A.},
  title     = {Decompression---Decompression Sickness},
  publisher = {Springer-Verlag},
  address   = {Berlin, Heidelberg},
  year      = {1984},
  doi       = {10.1007/978-3-662-02409-6}
}

@techreport{GerthDoolette2007,
  author      = {Gerth, Wayne A. and Doolette, David J.},
  title       = {{VVal-18} and {VVal-18M} Thalmann Algorithm Air Decompression Tables and Procedures},
  institution = {Navy Experimental Diving Unit},
  number      = {NEDU TR 07-09},
  address     = {Panama City, FL},
  month       = {May},
  year        = {2007},
  note        = {TA 01-07}
}

@manual{USNManualRev6,
  title        = {{U.S. Navy Diving Manual, Revision 6 with Change A}},
  organization = {Naval Sea Systems Command},
  address      = {Washington, DC},
  year         = {2011},
  note         = {Revision 6 issued in 2008; Change A dated 15 October 2011; SS521-AG-PRO-010 / 0910-LP-106-0957}
}

@manual{USNManualRev7,
  title        = {{U.S. Navy Diving Manual, Revision 7}},
  organization = {Naval Sea Systems Command},
  address      = {Washington, DC},
  year         = {2016},
  note         = {SS521-AG-PRO-010 / 0910-LP-115-1921; volumes 1--5; dated 1 December 2016}
}

@incollection{HamiltonThalmann2003,
  author    = {Hamilton, R. W. and Thalmann, E. D.},
  title     = {Decompression Practice},
  booktitle = {{Bennett and Elliott's} Physiology and Medicine of Diving},
  editor    = {Brubakk, Alf O. and Neuman, Tom S.},
  publisher = {Saunders/Elsevier},
  address   = {Edinburgh},
  edition   = {5th},
  year      = {2003},
  pages     = {455--500}
}

@incollection{SchreinerKelley1971,
  author    = {Schreiner, H. R. and Kelley, J. P.},
  title     = {A Pragmatic View of Decompression},
  booktitle = {Underwater Physiology IV},
  editor    = {Lambertsen, C. J.},
  publisher = {Academic Press},
  address   = {New York},
  year      = {1971},
  pages     = {205--219},
  doi       = {10.1016/B978-0-12-434750-2.50030-7}
}

@incollection{StatPearlsAlveolar2024,
  author    = {Hendrix, Joseph Maxwell and Burns, Bracken},
  title     = {Alveolar Gas Equation},
  booktitle = {StatPearls},
  publisher = {StatPearls Publishing},
  address   = {Treasure Island, FL},
  year      = {2024},
  note      = {Updated 11 September 2024},
  url       = {https://www.ncbi.nlm.nih.gov/books/NBK482268/}
}

@article{Pollock2015GF,
  author  = {Pollock, Neal W.},
  title   = {Gradient Factors: A Pathway for Controlling Decompression Risk},
  journal = {Alert Diver},
  year    = {2015},
  volume  = {31},
  number  = {4},
  pages   = {46--49},
  url     = {https://dan.org/alert-diver/article/gradient-factors/}
}

@proceedings{UHMSDeepStop2008,
  title     = {Decompression and the Deep Stop Workshop Proceedings},
  editor    = {Bennett, Peter B. and Wienke, Bruce R. and Mitchell, Simon J.},
  publisher = {Undersea and Hyperbaric Medical Society},
  address   = {Durham, NC},
  year      = {2008},
  note      = {UHMS/NAVSEA workshop, Salt Lake City, UT, 24--25 June 2008},
  isbn      = {0-930404-24-9}
}

@techreport{Thalmann1986,
  author      = {Thalmann, Edward D.},
  title       = {{Air--N2O2} Decompression Computer Algorithm Development},
  institution = {Navy Experimental Diving Unit},
  number      = {NEDU TR 8-85},
  address     = {Panama City, FL},
  month       = {August},
  year        = {1986}
}

@book{NOAADivingManual6,
  editor    = {McFall, Greg and Heine, John N. and Bozanic, Jeffrey E.},
  title     = {{NOAA} Diving Manual: Diving for Science and Technology},
  edition   = {6th},
  publisher = {Best Publishing Company},
  address   = {Flagstaff, AZ},
  year      = {2017},
  isbn      = {978-1-930536-88-3},
  note      = {{NOAA} 00-012-00269-1}
}

@book{Clarke1990,
  author    = {Clarke, Francis H.},
  title     = {Optimization and Nonsmooth Analysis},
  series    = {Classics in Applied Mathematics},
  volume    = {5},
  publisher = {SIAM},
  address   = {Philadelphia, PA},
  year      = {1990},
  doi       = {10.1137/1.9781611971309}
}

@book{Liberzon2012,
  author    = {Liberzon, Daniel},
  title     = {Calculus of Variations and Optimal Control Theory: A Concise Introduction},
  publisher = {Princeton University Press},
  address   = {Princeton, NJ},
  year      = {2012},
  doi       = {10.1515/9781400842643}
}

@book{Betts2010,
  author    = {Betts, John T.},
  title     = {Practical Methods for Optimal Control and Estimation Using Nonlinear Programming},
  edition   = {2nd},
  series    = {Advances in Design and Control},
  publisher = {SIAM},
  address   = {Philadelphia, PA},
  year      = {2010},
  doi       = {10.1137/1.9780898718577}
}

@article{Baker1998,
  author  = {Baker, Erik C.},
  title   = {Understanding {M-values}},
  journal = {Immersed},
  year    = {1998},
  volume  = {3},
  number  = {3},
  pages   = {23--27},
  url     = {https://www.dive-tech.co.uk/resources/mvalues.pdf}
}

@book{Pontryagin1962,
  author    = {Pontryagin, L. S. and Boltyanskii, V. G. and Gamkrelidze, R. V. and Mishchenko, E. F.},
  title     = {The Mathematical Theory of Optimal Processes},
  publisher = {Interscience Publishers},
  address   = {New York},
  year      = {1962}
}

@article{Danskin1966,
  author  = {Danskin, John M.},
  title   = {The Theory of Max--Min, with Applications},
  journal = {SIAM Journal on Applied Mathematics},
  year    = {1966},
  volume  = {14},
  number  = {4},
  pages   = {641--664},
  doi     = {10.1137/0114053}
}

@article{KuratowskiRyll1965,
  author  = {Kuratowski, Kazimierz and Ryll-Nardzewski, Czes{\l}aw},
  title   = {A General Theorem on Selectors},
  journal = {Bull. Acad. Polon. Sci. S{\'e}r. Sci. Math. Astronom. Phys.},
  year    = {1965},
  volume  = {13},
  pages   = {397--403}
}

@book{GareyJohnson1979,
  author    = {Garey, Michael R. and Johnson, David S.},
  title     = {Computers and Intractability: A Guide to the Theory of {NP}-Completeness},
  publisher = {W. H. Freeman},
  address   = {San Francisco, CA},
  year      = {1979}
}

@book{KellererPferschyPisinger2004,
  author    = {Kellerer, Hans and Pferschy, Ulrich and Pisinger, David},
  title     = {Knapsack Problems},
  publisher = {Springer},
  address   = {Berlin, Heidelberg},
  year      = {2004},
  doi       = {10.1007/978-3-540-24777-7}
}

@article{yong2022stochastic,
  author  = {Yong, Jiongmin},
  title   = {Stochastic Optimal Control---A Concise Introduction},
  journal = {Mathematical Control \& Related Fields},
  year    = {2022},
  volume  = {12},
  number  = {4},
  pages   = {1039--1136},
  doi     = {10.3934/mcrf.2020027}
}

@article{Fraedrich2018COTS,
  author  = {Fraedrich, Doug},
  title   = {Validation of Algorithms Used in Commercial Off-the-Shelf Dive Computers},
  journal = {Diving and Hyperbaric Medicine},
  year    = {2018},
  volume  = {48},
  number  = {4},
  pages   = {252--258},
  doi     = {10.28920/dhm48.4.252-258},
  pmid    = {30517958}
}

@article{Angelini2022Ceiling,
  author  = {Angelini, Sergio A. and Tonetto, Lorenzo and Lang, Michael A.},
  title   = {Ceiling-Controlled versus Staged Decompression: Comparison between Decompression Duration and Tissue Tensions},
  journal = {Diving and Hyperbaric Medicine},
  year    = {2022},
  volume  = {52},
  number  = {1},
  pages   = {7--15},
  doi     = {10.28920/dhm52.1.7-15},
  pmid    = {35313367},
  pmcid   = {PMC9016140}
}

@article{DeRidder2023GF,
  author  = {De Ridder, Sven and Pattyn, Nathalie and Neyt, Xavier and Germonpr{\'e}, Peter},
  title   = {Selecting Optimal Air Diving Gradient Factors for Belgian Military Divers: More Conservative Settings Are Not Necessarily Safer},
  journal = {Diving and Hyperbaric Medicine},
  year    = {2023},
  volume  = {53},
  number  = {3},
  pages   = {251--258},
  doi     = {10.28920/dhm53.3.251-258},
  pmid    = {37718300},
  pmcid   = {PMC10735712}
}

@article{Mitchell2024DCI,
  author  = {Mitchell, Simon J.},
  title   = {Decompression Illness: A Comprehensive Overview},
  journal = {Diving and Hyperbaric Medicine},
  year    = {2024},
  volume  = {54},
  number  = {1 Suppl},
  pages   = {1--53},
  doi     = {10.28920/dhm54.1.suppl.1-53},
  pmid    = {38537300},
  pmcid   = {PMC11168797}
}

@article{Imbert2024Saturation,
  author  = {Imbert, Jean-Pierre and Matity, Lyubisa and Massimelli, Jean-Yves and Bryson, Philip},
  title   = {Review of Saturation Decompression Procedures Used in Commercial Diving},
  journal = {Diving and Hyperbaric Medicine},
  year    = {2024},
  volume  = {54},
  number  = {1},
  pages   = {23--38},
  doi     = {10.28920/dhm54.1.23-38},
  pmid    = {38507907},
  pmcid   = {PMC11065503}
}

@article{Howle2017Trinomial,
  author  = {Howle, Laurens E. and Weber, Paul W. and Hada, Ethan A. and Vann, Richard D. and Denoble, Petar J.},
  title   = {The Probability and Severity of Decompression Sickness},
  journal = {{PLOS ONE}},
  year    = {2017},
  volume  = {12},
  number  = {3},
  pages   = {e0172665},
  doi     = {10.1371/journal.pone.0172665},
  pmid    = {28296928},
  pmcid   = {PMC5351842}
}

@techreport{Doolette2011NEDU,
  author      = {Doolette, David J. and Gerth, Wayne A. and Gault, Keith A.},
  title       = {Redistribution of Decompression Stop Time from Shallow to Deep Stops Increases Incidence of Decompression Sickness in Air Decompression Dives},
  institution = {Navy Experimental Diving Unit},
  number      = {NEDU TR 11-06},
  address     = {Panama City, FL},
  month       = {July},
  year        = {2011},
  note        = {TA 04-12},
  url         = {https://indepthmag.com/wp-content/uploads/2019/10/NEDU_TR_2011-06.pdf}
}

@article{Blasselle2019AdmissiblePressure,
  author  = {Alexis Blasselle and Michael Theron and Bernard Gardette and Emmanuel Dugrenot},
  title   = {A New Form of Admissible Pressure for {Haldanian} Decompression Models},
  journal = {Computers in Biology and Medicine},
  year    = {2019},
  volume  = {115},
  pages   = {103518},
  doi     = {10.1016/j.compbiomed.2019.103518}
}

@phdthesis{DiMuro2023Adaptive,
  author  = {Gianluca Di Muro},
  title   = {Innovations in Decompression Sickness Prediction and Adaptive Ascent Algorithms},
  school  = {Duke University},
  address = {Durham, NC},
  year    = {2023},
  url     = {https://dukespace.lib.duke.edu/items/60942160-1ce4-4746-93f4-8f8f6d35dcef}
}

@inproceedings{Feng2009Barrier,
  author    = {Le Feng and Christian R. Gutvik and Tor Arne Johansen and Dan Sui},
  title     = {Barrier Function Nonlinear Optimization for Optimal Decompression of Divers},
  booktitle = {Proceedings of the 48th IEEE Conference on Decision and Control Held Jointly with the 28th Chinese Control Conference},
  year      = {2009},
  pages     = {1800--1805},
  publisher = {IEEE},
  doi       = {10.1109/CDC.2009.5400696}
}

@article{Feng2010Multiparametric,
  author  = {Le Feng and Christian R. Gutvik and Tor Arne Johansen},
  title   = {Optimal Decompression through Multi-Parametric Nonlinear Programming},
  journal = {IFAC Proceedings Volumes},
  year    = {2010},
  volume  = {43},
  number  = {14},
  pages   = {810--815},
  doi     = {10.3182/20100901-3-IT-2016.00179}
}

@article{Feng2012RecedingHorizon,
  author  = {Le Feng and Christian R. Gutvik and Tor Arne Johansen and Dan Sui and Alf O. Brubakk},
  title   = {Approximate Explicit Nonlinear Receding Horizon Control for Decompression of Divers},
  journal = {IEEE Transactions on Control Systems Technology},
  year    = {2012},
  volume  = {20},
  number  = {5},
  pages   = {1275--1284},
  doi     = {10.1109/TCST.2011.2162516}
}

@article{Gutvik2011Optimal,
  author  = {Christian R. Gutvik and Tor Arne Johansen and Alf O. Brubakk},
  title   = {Optimal Decompression of Divers: Procedures for Constraining Predicted Bubble Growth},
  journal = {IEEE Control Systems Magazine},
  year    = {2011},
  volume  = {31},
  number  = {1},
  pages   = {19--28},
  doi     = {10.1109/MCS.2010.939141}
}

@book{HardyLittlewoodPolya1952,
  author    = {G. H. Hardy and J. E. Littlewood and George P{\'o}lya},
  title     = {Inequalities},
  edition   = {2},
  publisher = {Cambridge University Press},
  address   = {Cambridge},
  year      = {1952}
}

@phdthesis{Horn2003Optimization,
  author  = {Beverley J. Horn},
  title   = {Mathematical Optimisation of Diver Ascent Profiles at a Constant Risk of Decompression Illness},
  school  = {University of Canterbury},
  address = {Christchurch, New Zealand},
  year    = {2003},
  doi     = {10.26021/8717}
}

@article{Horn2006Isoprobabilistic,
  author  = {Beverley J. Horn and Graeme C. Wake and Timothy G. Anthony},
  title   = {Decompression Schedule Optimization with an Iso-Probabilistic Risk of Decompression Sickness},
  journal = {Aviation, Space, and Environmental Medicine},
  year    = {2006},
  volume  = {77},
  number  = {1},
  pages   = {13--19},
  note    = {PMID: 16422448}
}

@article{Lewis1983OptimalDecompression,
  author  = {Gilbert N. Lewis},
  title   = {A Mathematical Model for Decompression},
  journal = {Mathematical Modelling},
  year    = {1983},
  volume  = {4},
  number  = {5},
  pages   = {489--500},
  doi     = {10.1016/0270-0255(83)90052-0}
}

@phdthesis{Murphy2017Dissertation,
  author  = {F. Gregory Murphy},
  title   = {The Impact of Weighting Marginal {DCS} Events as Non-Events, Pharmacokinetic Gas Content Models, and Optimal Decompression Schedule Calculation},
  school  = {Duke University},
  address = {Durham, NC},
  year    = {2017}
}

@article{Parker1998Oxygen,
  author  = {E. C. Parker and S. S. Survanshi and P. B. Massell and P. K. Weathersby},
  title   = {Probabilistic Models of the Role of Oxygen in Human Decompression Sickness},
  journal = {Journal of Applied Physiology},
  year    = {1998},
  volume  = {84},
  number  = {3},
  pages   = {1096--1102},
  doi     = {10.1152/jappl.1998.84.3.1096}
}

@book{Rockafellar1970,
  author    = {R. Tyrrell Rockafellar},
  title     = {Convex Analysis},
  series    = {Princeton Mathematical Series},
  volume    = {28},
  publisher = {Princeton University Press},
  address   = {Princeton, NJ},
  year      = {1970},
  doi       = {10.1515/9781400873173}
}

@article{Schirato2026GradientOptimization,
  author  = {Sergio Rhein Schirato and Massimo Pieri and Riccardo Pelliccia and Alessandro Marroni and Costantino Balestra and Jos{\'e} Guilherme Chaui-Berlinck},
  title   = {A Gradient-Based Optimization Model for Predicting Decompression Sickness Risk},
  journal = {Frontiers in Physiology},
  year    = {2026},
  volume  = {17},
  pages   = {1852640},
  doi     = {10.3389/fphys.2026.1852640}
}

@techreport{Survanshi1996RealTime,
  author      = {S. S. Survanshi and P. K. Weathersby and Edward D. Thalmann},
  title       = {Statistically Based Decompression Tables {X}: Real-Time Decompression Algorithm Using a Probabilistic Model},
  institution = {Naval Medical Research Institute},
  address     = {Bethesda, MD},
  number      = {NMRI 96-06},
  year        = {1996},
  note        = {DTIC ADA308010},
  url         = {https://apps.dtic.mil/sti/tr/pdf/ADA308010.pdf}
}

\appendix
\section{Value function regularity and capped dynamic programming}\label{app:value-regularity}
The main text uses only the one step DP principle and finite action selectors. This appendix records local regularity estimates for the scalarised value, the analogous DPP for the risk capped value, and the additional repair condition needed before capped Lipschitz continuity can be claimed.

\subsection{Flow and integrand sensitivity}

We work on the compact set $\mathcal K$ fixed above, with pressure bound enlarged to $\bar P$ when necessary.

\begin{lemma}[Flow sensitivity]\label{lem:flow}
Fix a measurable control pair $(u,g)$ on $[0,T]$ and two initial states $(z,P),(\wt z,\wt P)\in\mathcal K$ for which the same open loop pair is feasible. Let $(z_t,P_t)$ and $(\wt z_t,\wt P_t)$ be the corresponding trajectories. Then, before either trajectory is stopped at the surface,
\[
|z_t-\wt z_t|=|z-\wt z|,
\]
and for all $t\in[0,T]$,
\[
\|P_t-\wt P_t\|_1\le e^{-k_{\min}t}\|P-\wt P\|_1+C_\infty(1-e^{-k_{\min}t})|z-\wt z|,
\]
with $C_\infty:=n(k_{\max}/k_{\min})F_I^{\max}\gamma$.
\end{lemma}

\begin{proof}
The depth identity is immediate from using the same rate control. For $P$, subtract the linear ODEs, use
\[
|P_\infty(g,z)-P_\infty(g,\wt z)|\le F_I^{\max}\gamma|z-\wt z|,
\]
sum over $i$, and apply Gr\"onwall's inequality.
\end{proof}

\begin{lemma}[Integrand Lipschitz bound]\label{lem:ell}
Let
\[
M_{\max}:=\max_i\sup_{z\in[0,Z_{\max}]}M_i(z)=\max_i(a_i+b_i(\Po+\gamma Z_{\max})),
\qquad L_{M,i}:=b_i\gamma.
\]
Then for all $(z,P),(\wt z,\wt P)\in\mathcal K$ and $g\in\calG$,
\[
|\ell(z,P,g)-\ell(\wt z,\wt P,g)|\le L_\ell\bigl(|z-\wt z|+\|P-\wt P\|_1\bigr),
\]
where one may take
\[
L_\ell:=\lambda\max\!\left\{
\frac1{M_{\min}}\sum_{i=1}^nL_{\phi,i},\quad
\sum_{i=1}^nL_{\phi,i}\left(\frac{L_{M,i}}{M_{\min}}+\frac{(\bar P+M_{\max})L_{M,i}}{M_{\min}^2}\right)
\right\}.
\]
\end{lemma}

\begin{proof}
For $S_i=(P_i-M_i(z))_+/M_i(z)$, the map $(z,P_i)\mapsto(P_i-M_i(z))/M_i(z)$ is Lipschitz on $\mathcal K$ because $M_i\ge M_{\min}>0$. The positive part map is $1$ Lipschitz and each $\phi_i$ is $L_{\phi,i}$ Lipschitz on $[0,\bar S]$. Summing the coordinate bounds gives the stated constant. The term $+1$ in $\ell$ cancels in the difference.
\end{proof}

\subsection{Lipschitz regularity of the value function}

\begin{theorem}[Local Lipschitz continuity of $V_\lambda$]\label{thm:Vlip}
Let $\mathcal K'\subset\mathcal K$ be compact and assume robust feasibility along the complete comparison trajectories, every open loop gas used by an $\varepsilon$-optimal comparison policy from $\mathcal K'$ remains feasible along the full perturbed trajectory, up to the surface, for all sufficiently small initial depth perturbations. Then $V_\lambda$ is locally Lipschitz in $(z,P)$ on $\mathcal K'$. Let $L_\Psi$ be a Lipschitz constant of $\Psi$ on $[0,\bar P]^n$ and put
\[
F_{\max}:=\sum_{i=1}^n k_i(\bar P+F_I^{\max}(\Po+\gamma Z_{\max})).
\]
For nearby $(z,P),(\wt z,\wt P)\in\mathcal K'$ one may take
\[
|V_\lambda(z,P)-V_\lambda(\wt z,\wt P)|\le C_1|z-\wt z|+C_2\|P-\wt P\|_1,
\]
with
\[
C_2:=\frac{L_\ell}{k_{\min}}+\lambda L_\Psi,
\qquad
C_1:=L_\ell(1+C_\infty)T_{\mathrm{hor}}
+\frac{\bar\ell}{\dot z_{\max}}
+\lambda L_\Psi\left(C_\infty+\frac{F_{\max}}{\dot z_{\max}}\right).
\]
Across gas feasibility boundaries one should read the conclusion piecewise, global Lipschitz continuity need not follow without an additional slack or regularity assumption.
\end{theorem}

\begin{proof}
Suppose first that $\wt z\ge z$. Fix $\varepsilon>0$ and choose an $\varepsilon$-optimal admissible pair $(u^\star,g^\star)$ for $(z,P)$ with horizon $T\le T_{\mathrm{hor}}$ that reaches the surface at time $T$. By robust feasibility, the same open-loop gas choices remain feasible from $(\wt z,\wt P)$ for nearby states along the complete comparison trajectory. Running this policy leaves a residual depth at most $\wt z-z$ at time $T$, which can be cleared by a no-hold ascent with running cost at most $\bar\ell(\wt z-z)/\dot z_{\max}$. Hence
\[
V_\lambda(\wt z,\wt P)-V_\lambda(z,P)
\le\int_0^T L_\ell\bigl(|\wt z-z|+\|\wt P_t-P_t\|_1\bigr)\,dt
+\frac{\bar\ell}{\dot z_{\max}}|\wt z-z|
+\lambda L_\Psi\|\wt P_{\mathrm{term}}-P_{\mathrm{term}}\|_1
+\varepsilon.
\]
Using \cref{lem:flow},
\[
\int_0^T\|\wt P_t-P_t\|_1\,dt
\le\frac1{k_{\min}}\|\wt P-P\|_1+C_\infty T_{\mathrm{hor}}|\wt z-z|.
\]
At time $T$, the state difference is at most $\|\wt P-P\|_1+C_\infty|\wt z-z|$. Clearing the residual depth changes the terminal state by at most $F_{\max}|\wt z-z|/\dot z_{\max}$. These estimates bound the terminal term by the $\lambda L_\Psi$ contributions in $C_1$ and $C_2$. Combining and letting $\varepsilon\downarrow0$ gives the one sided bound. If $\wt z<z$, swap the roles of the two states and stop the shallower comparison trajectory when it reaches the surface; the same residual depth estimate applies. This yields the symmetric Lipschitz estimate.
\end{proof}

\subsection{Risk capped DPP and the need for a repair condition}
\label{subsec:cap-dpp}

For the capped problem in \cref{prob:def}, augment the state with the remaining budget $r\in[0,\rho]$ and define
\begin{equation}\label{eq:W-def}
\begin{aligned}
W(z,P,r):=\inf\Bigg\{&\int_0^T1\,dt:\\
&\dot z\in\{0,-\dot z_{\max}\},\quad
\dot P_i=k_i(P_\infty-P_i),\quad g(t)\in\mathcal G(z(t)),\\
&\int_0^T\sum_i\phi_i(S_i(t))\,dt+\Psi(P(T))\le r,
\quad T=\inf\{t\ge0:z(t)=0\}\Bigg\}.
\end{aligned}
\end{equation}
Let $R_h(z,P,u,g):=\int_0^h\sum_i\phi_i(S_i(t))\,dt$ denote the risk accrued over $[0,h]$ under $(u,g)$.

\begin{theorem}[DPP for the cap]\label{thm:dpp-cap}
For any $h\in(0,h_{\max}(z)]$ with $z>0$,
\begin{equation}\label{eq:dpp-cap}
W(z,P,r)=\inf_{\substack{u(\cdot)\in\{0,-\dot z_{\max}\}\ g(\cdot)\in\calG(z(\cdot))}}
\left\{h+W\bigl(z(h),P(h),r-R_h(z,P,u,g)\bigr)\right\},
\end{equation}
with the convention $W(\cdot,\cdot,\tilde r)=+\infty$ if $\tilde r<0$. At the terminal boundary,
\[
W(0,P,r)=
\begin{cases}
0,&\Psi(P)\le r,\\
+\infty,&\Psi(P)>r.
\end{cases}
\]
Moreover $W$ is nonincreasing in $r$.
\end{theorem}

\begin{proof}
The DPP follows by the same truncation/concatenation argument as
\cref{thm:dpp}, tracking the running budget $r$; the terminal functional is enforced by the displayed boundary condition. Monotonicity in $r$ is immediate.
\end{proof}

\begin{remark}[Regularity needs a risk repair condition]
\label{rem:cap-repair}
No unconditional local Lipschitz claim for $W$ follows merely by staying away from a named budget boundary. A sufficient condition is a local risk repair error bound, there are $K,\delta_0>0$ such that any nearby schedule with risk at most $r+\delta$, $0\le\delta\le\delta_0$, can be modified to meet risk $r$ at an added time of at most $K\delta$. Combined with robust gas feasibility, the flow estimates of \cref{lem:flow,lem:ell} then imply local Lipschitz continuity. Without such a condition, the valid conclusion is only bicriteria continuity with an $O(\|(z,P)-(\widetilde z,\widetilde P)\|)$ relaxation of the cap.
\end{remark}

The constants $L_\ell,L_\Psi,C_\infty,\bar\ell,T_{\mathrm{hor}}$ determined above make the a-priori discretisation bounds in \cref{thm:apriori} explicit on a chosen compact $\mathcal K'\subset\mathcal K$. In particular, $L_{\mathrm{stop}}$ and $L_{\mathrm{dwell}}$ are controlled by the Lipschitz constants of the running and terminal risk together with the uniform horizon bound in \cref{eq:ell-bar}.
\end{document}